\documentclass[11pt,reqno]{amsart}
\usepackage{amscd,amsmath,amsopn,amssymb,amsthm,multicol}
\usepackage{hyperref}
\usepackage{amscd}
\usepackage{lscape}
\usepackage{slashed}
\usepackage{graphicx}
\usepackage{lscape}
\usepackage{verbatim}
\usepackage{color}
\usepackage[b]{esvect}
\usepackage[all,cmtip]{xy}
\usepackage{tikz-cd}
\usepackage{amsmath}

\definecolor{dblue}{rgb}{0.01,0.01,0.44}
\definecolor{red}{rgb}{0.57,0.11,0.15}
\DeclareMathOperator{\Ad}{Ad}
\DeclareMathOperator{\ad}{ad}
\DeclareMathOperator{\Aut}{Aut}

\DeclareMathOperator{\Id}{Id}

\DeclareMathOperator{\Cl}{C\ell}

\DeclareMathOperator{\Hgt}{ht}

\DeclareMathOperator{\Spin}{Spin}

\DeclareMathOperator{\rnk}{rnk}

\newcommand{\fr}{\mathfrak}
\newcommand{\al}{\alpha}
\newcommand{\be}{\beta}

\newcommand{\bb}{\mathbb}
\newcommand{\cal}{\mathcal}

\DeclareMathOperator{\SO}{SO}
\DeclareMathOperator{\Sp}{Sp}
 \DeclareMathOperator{\SU}{SU}
  
\DeclareMathOperator{\U}{U}
\DeclareMathOperator{\G}{G}
\DeclareMathOperator{\F}{F}
\DeclareMathOperator{\Oo}{O}
\DeclareMathOperator{\E}{E}
\DeclareMathOperator{\A}{A}
\DeclareMathOperator{\B}{B}
\DeclareMathOperator{\Cc}{C}
\DeclareMathOperator{\D}{D}
\DeclareMathOperator{\Ss}{S}

\DeclareMathOperator{\Gl}{GL}
\newcommand{\thickline}{\noalign{\hrule height 1pt}}

 \newtheorem{lemma} {Lemma} [section]
\newtheorem{theorem}[lemma]{Theorem}
\newtheorem{remark}[lemma] {Remark}
\newtheorem{prop} [lemma]{Proposition}
\newtheorem{definition}[lemma] {Definition}
\newtheorem{corol}[lemma] {Corollary}
\newtheorem{example}[lemma] {Example}

\begin{document}

 \title
{Spin structures on  compact homogeneous  pseudo-Riemannian  manifolds}
 \author{Dmitri V.~ Alekseevsky and Ioannis Chrysikos}
 \address{Institute for Information Transmission Problems, B. Karetny
per. 19, 127051, Moscow, Russia
and Faculty of Science,  University of Hradec Kr\'alov\'e,  Rokitanskeho 62, Hradec Kr\'alov\'e
 50003, Czech Republic}
 \email{dalekseevsky@iitp.ru}
 \address{Faculty of Science,  University of Hradec Kr\'alov\'e,  Rokitanskeho 62 Hradec Kr\'alov\'e
 50003, Czech Republic and Department of Mathematics and Statistics, Masaryk University, Kotl\'a\v{r}sk\'a 2, Brno  611 37, Czech Republic}
 \email{ioannis.chrysikos@uhk.cz}

%\thanks{The  author    was full-supported
  %by Masaryk University under the Grant Agency of Czech Republic, project no. P 201/ 12/ G028}

% \tableofcontents
 \begin{abstract}
%We    study  the existence     of    invariant   spin  structures  on   homogeneous  pseudo-Riemannian manifolds.
 %In terms of representation theory, we classify invariant spin structures or metaplectic structures on any flag manifold $G/H$ of a %compact semi-simple Lie group $G$ (both classical or exceptional). % Existence of  such  structure  does not  depend  on  an invariant pseudo-Riemannian metric.
  %We also  investigate   spin  structures  on compact and simply-connected   complex homogeneous   manifolds, i.e.  C-spaces,  viewing them as  principal toric bundles over flag manifolds.  We examine the existence of invariant spin structure  on compact proper  homogeneous  Lorentzian manifolds  $M=G/H$ of  a semi-simple Lie  group $G$.
  %Such manifolds are called  minimal admissible homogeneous Lorentzian manifolds,  if   no  coset of the form  $G/\tilde{H}$  with $\tilde{H} \supset H$ and  $\dim \tilde{H}> \dim H$, admits   a  $G$-invariant  Lorentzian  metric and give rise to $\Ss^{1}$-bundles over flag manifolds.  We   classify   all  minimal admissible homogeneous Lorentzian manifold of a  simple Lie  group $G$   up  to  dimension 11 and
  %    describe which  of  them  are  spin.  We also describe all $G$-spin minimal admissible homogeneous Lorentzian manifolds with second Betti number $b_{2}(M)=1, 2$.

     We    study  spin    structures   on  compact simply-connected  homogeneous  pseudo-Riemannian manifolds  $(M = G/H,g)$  of  a  compact   semisimple  Lie  group $G$.
  We classify   flag manifolds $F = G/H$   of a compact simple Lie group   which  are    spin.
This yields   also the  classification of all flag manifolds carrying an invariant metaplectic  structure.
Then we   investigate   spin  structures  on principal torus  bundles over flag manifolds $F=G/H$, i.e. C-spaces, or equivalently   simply-connected  homogeneous  complex  manifolds $M=G/L$ of  a  compact   semisimple Lie  group $G$.  We study the topology of $M$ and we provide a sufficient  and necessary condition for  the existence of an (invariant) spin structure, in terms of the Koszul form of $F$. We  also classify  all C-spaces  which are  fibered over an exceptional spin flag manifold and hence they are spin. %is also spin. hen the base  flag manifold  is not spin. %and we present examples.

 \medskip
 \noindent 2000 {\it Mathematics Subject Classification.}    53C10, 53C30, 53C50, 53D05.

 \noindent {\it Keywords}:   spin structure, metaplectic structure,  homogeneous pseudo-Riemannian manifold, flag manifold,   Koszul form, C-space

  \smallskip
 \centerline{\it Dedicated to the memory of M.~Graev}

  \end{abstract}
\maketitle
% \tableofcontents

 \section*{Introduction}

   This paper is devoted to a systematic study of invariant spin structures  and   metaplectic structures on homogeneous spaces $M=G/L$.  Spin structures,   $\text{spin}^{\bb{C}}$ structures and metaplectic structures  have crucial  role  in   differential geometry  and  physics.  For example, existence  of a  spin   structure  on   a manifold  is  assumed in  most physical models in supergravity  and string  theory,  which is  essential   for the  definition  of  Dirac  and   twistor  operators, Killing  and   twistor spinors,  for   formulation   and  description of  supersymmetry, etc (see \cite{Law, Baum,  AT, Fried, Srni,  FF, Kil2} for references in all these directions). The same time,   $\text{spin}^{\bb{C}}$ structures are ubiquitous in dimension 4 by means of Seiberg-Witten theory (cf. \cite{Fried}) and metaplectic structures are necessary  in geometric quantization (cf. \cite{forger, Hab}).

    If    $M=G/L$ is a simply-connected homogeneous  pseudo-Riemannian manifold which is  time-oriented  and space-oriented,  % homogeneous  pseudo-Riemannian manifold  %of a simply-connected Lie group $G$ modulo a compact  subgroup $L$,
    then  the existence  of a spin    structure  does not  depend   on a particular  invariant metric and its  signature,    but only on the topology  of  $M$.  Moreover, if $G$ is simply-connected and  such a (unique) structure  exists, then it will be  $G$-invariant, i.e. it is  defined   by  a  lift $\tilde{\vartheta} : L \to  \Spin(V),\,  V = T_oM, $  of   the isotropy    representation  $\vartheta : L \to \SO(V)$ of  the  stability   subgroup  to  the  spin group  $\Spin(V)$.  Since   the isotropy  representation  $\vartheta : \mathfrak{l} \to \mathfrak{so}(V)$  of  the  stability  subalgebra can be   always  lifted   to $\mathfrak{spin}(V)$,  the  lift  of  $L$, if exists,  may be obtained   by exponentiation of  the  lift $\vartheta(\mathfrak{l}) \subset \mathfrak{spin}(V)$ of  the isotropy Lie  algebra. Then, one can  describe   the associated  spinor bundle $\Sigma(M) = G \times_{\tilde{\vartheta}(L)}\Delta$,  where $\Delta$ is  the  spinor  module, and define spinorial objects (spinor fields, Dirac operators,  etc.)
 Hence, existence  of a  spin   structure allows us   to  construct  this bundle  explicitly   and   deal  with    spinor  geometry.
  %sufficient  to  is   sufficient   to   the conventional method of dealing with such objects is to require  the existence of a spin structure.
   Notice however  that do exist   homogeneous (pseudo)Riemannian   manifolds which are not spin, e.g. $\bb{C}P^{2}$ or $\bb{C}P^{2}-\{{\rm point}\}$  (although it carries countably many  $\text{spin}^{\bb{C}}$ structures, see \cite{plymen}).  Moreover, for $n\geq 5$ there are compact oriented manifolds which do not carry a $\text{spin}^{\bb{C}}$ structure, see \cite{Kil2}.   Therefore, spin or $\text{spin}^{\bb{C}}$ structures (or metaplectic structures) do not always exist.  However,  except of special structures  and particular constructions (see for example \cite{FKMS, Srni, Law}), only a few general classification  results     are known  about the existence  of  spin  structures on   homogeneous spaces.  For  example, existence of    a spin  structure  on  Riemannian symmetric spaces and  quadrics, is  investigated in \cite{Cahen2, Cahen}. In a more recent work \cite{Gadea},    invariant spin structures are also described  on pseudo-symmetric spaces and non-symmetric cyclic homogeneous Riemannian manifolds.

   Our results can be read as follows. After recalling some basic material in Section \ref{intro},  in Section  \ref{homogspin} we study invariant spin structures on  pseudo-Riemannian homogeneous spaces, using homogeneous fibrations.  Recall that given a smooth fibre bundle $\pi : E\to B$ with connected fibre $F$,  the tangent bundle $TF$ of $F$ is stably equivalent to $i^{*}(TE)$, where $i : F\hookrightarrow E$ is the inclusion map (cf. \cite{Sin}). Evaluating this result at the level of characteristic classes,  one can treat the existence of a spin structure on the total space $E$ in terms of Stiefel-Whitney classes of $B$ and $F$, in the spirit of the theory developed by Borel and Hirzebruch  \cite{Bo}. We apply these considerations for fibrations induced by  a tower of closed connected  Lie subgroups $L\subset H\subset G$ (Proposition \ref{genw2}) and we describe sufficient and necessary conditions for the existence of a spin structure on the associated total space (Corollary \ref{springer2}, see also \cite{Gadea}). %
  Next  we apply these results in several  particular cases.
  %         Using  these results,
   For example,  in   Section \ref{flags}   we classify   spin and  metaplectic structures on compact homogeneous K\"ahler manifolds of a compact connected semisimple Lie group $G$, i.e. (generalized) flag manifolds.

   Generalized  flag manifolds are homogeneous spaces of the form $G/H$, where $H$ is the centralizer of  torus in $G$. %Such manifolds have very rich complex geometry and their classification is well-known (cf. \cite{AP, AA, Bor}). %they have very rich symplectic and K\"ahler geometry.
  Here, we  explain how   the existence  and classification  of   invariant  spin or metaplectic   structures  can be treated  in term of representation theory (painted Dynkin diagrams) and provide a criterion  in terms of the so-called {\it Koszul numbers} (Proposition \ref{chernclass2}).  These are the integer  coordinates of the  invariant Chern  form  (which  represent   the  first Chern class  of of  an invariant  complex  structure $J$  of $F=G/H$), with  respect   to   the fundamental weights.
   %corresponding to the painted black simple roots  in the Dynkin diagram of $G$.       %Thus this forms has a very special character in the geometry and to%Thus its knowledge provides the same time geometric and topological informations.
By applying an algorithm given in \cite{AP} (slightly revised),  we compute the Koszul numbers for any flag manifold corresponding to a classical Lie group and provide necessary and sufficient conditions for the existence of a spin or metaplectic  structure for any such coset (Corollary \ref{CLASFLAGS}, Theorem \ref{apl1}).  In addition, we present an explicit description of all classical spin or metaplectic flag manifolds with $b_{2}(F)=1$, or $b_{2}(F)=2$ (Theorem \ref{OURGOOD1}, Table 1).  Then, we extend our study on flag manifolds associated to an exceptional Lie group and  provide  the Koszul numbers for any such space, a problem which  was left open in \cite{AP}.  There are 101 non-isomorphic  exceptional flag manifolds and we show that, up to equivalence,  37 of them admit a (unique) $G$-invariant spin or metaplectic structure (Theorems \ref{g2}, \ref{e7-8}, \ref{e8}).        For  such spaces we also compute another invariant, namely the  cardinality of the $T$-root system $R_{T}$, verifying the classification given in  \cite{Gr2}. For convenience, we summarize these results  together with   the Koszul form  in     Tables  2, 3 and 4.  It worths to remark that  the    Koszul  form  encodes  important geometric  information   about a flag manifold $F$  and other associated  spaces. For example,  it   defines an invariant  K\"ahler-Einstein metric on $F$   associated  with   an  invariant  complex  structure, a  Sasaki-Einstein metric on  the  associated  $\Ss^1$-bundle $S$  over $F$ and  a Ricci-flat  K\"ahler metric   on   the Riemannian   cone  over  $S$.

     In this direction we examine (invariant) spin structures on C-spaces, that is  simply-connected  compact complex homogeneous spaces.  Such manifolds were classified by Wang \cite{Wang} and they are   toric  bundles over    flag manifolds.  C-spaces may   admit    invariant  Lorentz  metric and    invariant   complex  structure   with zero  first Chern class, in contrast to flag manifolds.  Therefore, such homogeneous spaces may provide examples of  homogeneous Calabi-Yau structures with torsion, see     \cite{Fino, Gran, GGP} for  details and references.
    Here, we use the Tit's fibration $M=G/L\to F=G/H$ of a  C-space over a flag manifold $G/H$ of a compact semisimple Lie group $G$ and treat the first Chern class of $G/L$ in terms of the Koszul form associated to $G/H$. We examine the topology of $M$ (see  Propositions \ref{cspacespin}, \ref{c1w2}, \ref{newnew}) and  describe necessary and sufficient conditions for the existence of a spin structure (see Corollary \ref{concl} and Theorem \ref{final3}).   % (see also \cite{Hano, Gran}).
   It is an immediate conclusion that any  C-space $G/L$  fibered over a $G$-spin flag manifold $G/H$   is  automatically $G$-spin.  Thus, for example, our classification of spin flag manifolds enables us to describe all C-spaces that can be fibered over a $G$-spin flag manifold of an exceptional Lie group $G$  (Proposition \ref{cprop}).  We finally provide  a new construction which allows us to present C-spaces admitting an invariant spin structure, even when the base of the Tit's fibration is not spin (Corollary \ref{cspacespin2}). %Examples related to this case are described in some details.

   We   mention  that the results of this work  can be applied  for  the  classification of  spin structures on  homogeneous Lorentzian manifolds of a semisimple Lie group. This  application will be  presented  in a forthcoming work.
  % In particular, in the compact case such spaces are principal circle bundles over flag manifolds and hence a combination of the results %in Sections 2 and 3 give rise to a simple methodology for a complete classification of spin structures.

 \noindent {\bf Acknowledgements:} The second   author is grateful to  Yusuke Sakane (Osaka) for sharing his insight at an early stage  and acknowledges his steady influence.  He  also thanks  Svatopluk Kr\'ysl  (Prague) for  valuable references on  metaplectic structures.  Part of this work was completed under the support of   GA\v{C}R (Czech Science Foundation),  post-doctoral grant   GP14-24642P, which is gratefully acknowledged by the second author.
 Both authors acknowledge the institutional support of University of Hradec Kr\'alov\'e
and thank Masaryk University in Brno for hospitality.

% The  authors   thank the Department of Mathematics and Statistics in Masaryk University and Faculty of  Science of University of Hradec %Kr\'alov\'e for  support.
  %during   research stays and visiting positions.  %He  also thanks  Svatopluk Kr\'ysl  (Prague) for  valuable references on  metaplectic   %structures and gratefully  acknowledges   financial support  by  GA\v{C}R (Czech Science Foundation), post-doctoral grant  GP14-24642P.
%The second   author gratefully acknowledges   financial support  by  GA\v{C}R (Czech Science Foundation), post-doctoral grant  GP14-24642P.

\section{Preliminaries}\label{intro}

 In this section we collect basic facts on spin structures on pseudo-Riemannian manifolds. For a detailed exposition the reader may consult the books \cite{Baum2, Law, Baum, Fried}.
      Consider a  connected oriented  pseudo-Riemanian manifold   $(M^{n},g)$   of  signature $(p,q)$ and let        $\pi: P=\SO(M)\to M$ be  the $\SO_{p,q}$-principal  bundle of positively oriented orthonormal  frames.
 Then, the tangent bundle  $TM =\SO(M)\times_{\SO_{p,q}} \mathbb{R}^n$    admits  an orthogonal
    decomposition
    \begin{equation}\label{clasdec}
    TM = \eta_- \oplus  \eta_+,
    \end{equation}
       where  $\eta_- $ (resp.  $\eta_+$)  is a rank $p$  time-like subbundle, i.e. $g|_{\eta_-} <0$    (resp.  rank $q$  space-like subbundle), i.e. $g|_{\eta_+} >0$.   In general, this decomposition is not unique.
          Recall that   $(M^{n}, g)$ is called {\it time-oriented }  (resp. {\it  space-oriented}) if   $\eta_-$ (resp.  $\eta_+$) is oriented, which is equivalent to say that  the associated first  Stiefel-Whitney  class   vanishes,  $w_1(\eta_-) =0$  (resp.  $w_1(\eta_+) =0$).
Since $w_1(M) := w_1(TM) = w_1(\eta_-) + w_1(\eta_+)$,  the manifold $M$ is oriented if  and only if  $w_1(\eta_-) + w_1(\eta_+)=0$.

Consider the spin group $\Spin_{p, q}$ and let us denote by $\Ad : \Spin_{p, q} \to \SO_{p, q}$ the  $\bb{Z}_{2}$-cover over $\SO_{p, q}$.
\begin{definition}\label{def1}
 \textnormal{A  $\Spin_{p, q}$-structure   (shortly  {\it spin structure}) on a connected  oriented  pseudo-Riemannian manifold  $(M^{n}, g)$ of signature $(p, q)$ is  a   $\Spin_{p,q}$-principal   bundle  $\tilde{\pi} : Q= \mathrm{Spin}(M) \to M$ over $M$    which is   a
    $\mathbb{Z}_2$-cover  $\Lambda : \Spin(M)\to \SO(M)$ of  the orthonormal frame  bundle  $\pi: \SO(M)  \to M$  such  that  the following diagram is commutative:
        {\small{\[
    \xymatrix{
                                                \Spin(M)\times\Spin_{p, q} \ar[d]_{\Lambda\times \Ad}\ar[r] & \Spin(M) \ar[dr]^{\tilde{\pi}}  \ar[d]^{\Lambda} & \\
                                                    \SO(M)\times \SO_{p, q} \ar[r]  & \SO(M) \ar[r]^{\pi}& M }
         \]}}
         If such a pair $(Q, \Lambda)$ exists,  we shall call  $(M^{n}, g)$   a {\it pseudo-Riemannian spin manifold}.}

      \end{definition}

  %Let $(M^{n},g)$ be  a  time-  and  space-oriented  pseudo-Riemannian  manifold and let  us denote by  $\pi^+ :P^{+}:=\SO^{+}(M) \to M$ the   $\SO_{p,q}^+$-principal  bundle of  positively time-  and space-oriented orthonormal frames, where $\SO_{p, q}^{+}$ is the connected component of $\SO_{p, q}$.  We similarly write $\Spin^+_{p,q}$ for the connected component of $\Spin_{p, q}$. % for the %This subbundle  is the identity connected component of $P:=\SO(M)$.
  % Then, one  can define a $\Spin^{+}_{p, q}$-structure (shortly a  $\text{spin}^{+}$ structure),  as a reduction $\tilde{\pi}^+ : Q^{+}:=\Spin^+(M)  \to  M$  of   $P^{+}$ onto   $\Spin^{+}_{p, q}$, such that $Q^{+}\to P^{+}$ is  the double covering.

\begin{definition}
\textnormal{Two $\Spin_{p, q}$-structures $(Q_{1}, \Lambda_{1})$ and $(Q_{2}, \Lambda_{2})$ are said to be {\it equivalent}  if there is a  $\Spin_{p, q}$-equivariant map $\cal{U}: Q_{1}\to Q_{2}$  between the $\Spin_{p, q}$-principal bundles $Q_{1}, Q_{2}$ such that $\Lambda_{2}\circ \cal{U}=\Lambda_{1}$.}%  (and similarly for $\text{spin}^{+}$ structures).}
\end{definition}

 \begin{prop}\textnormal{(\cite[Prop.~1.1.26]{Karoubi}, \cite[Satz~2.2]{Baum2}, \cite[Prop.~9]{Cahen2})}\label{ba}
An oriented pseudo-Riemannian manifold $(M^{n}, g)$ of signature $(p,q)$ admits  a  $\Spin_{p, q}$-structure  if and only if $w_{2}(\eta_{-})+w_{2}(\eta_{+})=0$,
  or  equivalently, $w_2(TM) = w_1(\eta_-)\smile w_1(\eta_+)$. Here,  $\eta_{-}$ (resp. $\eta_{+}$) is the time-like (resp. space-like) subbundle  given in (\ref{clasdec}).  If this  condition  is satisfied, then inequivalent spin structures on $(M, g)$ are in bijective correspondence with elements in $H^{1}(M; \bb{Z}_{2})$.
\end{prop}

%Note  that  the   equivalence of  the  above  conditions  follows  from  the  formula   for  the Whitney product:
%\[   \label{Whitney product}
%w_{2}(M)=w_{2}(\eta_{-})+w_{1}(\eta_{-})\smile w_{1}(\eta_{+})+w_{2}(\eta_{+}),
%\]
%where $\smile : H^{k}(M; \bb{Z}_{2})\times H^{l}(M; \bb{Z}_{2})\to H^{k+l}(M; \bb{Z}_{2})$  is the cup product.

As a consequence, we see that

\begin{corol}
 Let $(M^{n}, g)$ be a connected oriented pseudo-Riemannian manifold. If $M$ is    space-oriented or time-oriented,  then it  admits   a spin structure if and only if its  second Stiefel-Whitney class vanishes. The same holds if  $(M^{n}, g)$ is a connected oriented Riemannian manifold.
\end{corol}
%\begin{proof}
%The relation $w_{1}(TM)=w_{1}(\eta_{+})+w_{1}(\eta_{-})$ shows that if we have two kinds of orientability, then we will have also the third one. Hence,  if for example $M$ is oriented and space-oriented, then it is also time-oriented, i.e. $w_{1}(TM)=w_{1}(\eta_{+})=w_{1}(\eta_{-})=0$ and our claim  follows due to the condition $w_2(M) = w_1(\eta_-)\smile w_1(\eta_+)$, given in Proposition \ref{ba}. For the Riemannian case see also \cite{Fried, Law}.
%\end{proof}

  \begin{remark}
\textnormal{The existence  of  spin structure   for   a connected oriented  Riemannian manifold $(M,g)$  does not  depend   on  the metric  $g$,  but only on  the  topology  of   $M$. % in particular on the condition  $w_2(M)=0$.
The  similar  result  is valid  also  for  Lorentzian  manifold  \cite{Mor}. However,  it is not   true  for  other  types of signature (cf. \cite[p.~78]{Baum2}).}
\end{remark}

Consider an almost complex manifold $(M^{2n}, J)$.  Then, $M$ carries a natural orientation induced by $J$. % since  $\Gl_{n}\bb{C}\subset\Gl_{2n}\bb{R}$.  %Moreover, the structure group  $\SO_{2n}$ of $\SO(M)$ is reduced to the unitary group $\U_{n}$.
%Denote  by $\Lambda^{r,s}(M)= \Lambda^{r, 0}(M)\otimes\Lambda^{0, s}(M)\subset \Lambda^{r+s}(T^*M)\otimes \mathbb{C}$ the   bundle  of  complex valued forms of  type  $(r,s)$, where $\Lambda^{r, 0}(M):=\Lambda^{r}(T^{*}M^{1, 0})$ and $\Lambda^{0, s}(M):=\Lambda^{s}(T^{*}M^{0, 1})$  are the spaces of  holomorphic and anti-holomorphic forms.
 %The complex line bundle
 %  $K_M:= \Lambda^{n, 0}(M)=\Lambda^{n}(T^{*}M^{1, 0})$  is the so-called canonical line bundle.
 Since  $TM$ is a  complex vector bundle, one can define the Chern classes  $c_{j}\equiv c_{j}(M):=c_{j}(TM, J)\in H^{2j}(M; \bb{Z})$  of $(TM, J)$. When $M$ is compact, $H^{2n}(M; \bb{Z})\cong\bb{Z}$ due to the natural orientation and   the top Chern class $c_{n}(M) $ coincides  with  the  Euler  class $e(M) \in H^{2n}(M,\mathbb{Z})$  of
 the   underlying real tangent     bundle. Moreover, $c_{1}(TM, J)= w_{2}(TM)  \, (\mathrm{mod} \, 2)$,  see \cite[p.~82]{Law}.
 %\begin{equation}   \label{w-2 and c_1}
  %\[ c_{1}(TM, J)= w_{2}(TM)  \, (\mathrm{mod} \, 2).
  %\]
 %\end{equation}
 \begin{prop}\textnormal{(\cite{AT})}
 An almost complex manifold $(M^{2n}, J)$ admits a spin structure if and only if $K_{M}$ admits a square root, i.e. there exists a complex line bundle $\cal{L}$ such that $\cal{L}^{\otimes 2}=K_{M}$,  where $K_M:= \Lambda^{n, 0}(M)=\Lambda^{n}(T^{*}M^{1, 0})$  is the   canonical line bundle.
 \end{prop}
 In the case of a compact complex manifold one can extend this result  to holomorphic line bundles.

  \begin{prop}\label{at}\textnormal{(\cite{Law, AT})}
Let   $M^{2n}$ be a compact   complex   manifold  with   complex   structure  $J$. Then
   $M$  admits  a  spin  structure  if  and only if    the  first Chern  class  $c_1(M) \in H^2(M, \mathbb{Z})$ is  even,  i.e. it is   divisible  by  2   in $H^2(M,\mathbb{Z})$.
 Moreover, spin structures      are in 1-1 correspondence with isomorphism classes of holomorphic line bundles $\cal{L}$ such that $\cal{L}^{\otimes 2}=K_{M}$. % where the reference isomorphism is holomorphic.
 \end{prop}

 Let   $(V = \mathbb{R}^{2n}, \, \omega )$ be   the  symplectic vector  space   and  $ \mathrm{Sp}(V) ={\rm Sp}_{n}(\bb{R}) := \mathrm{Aut}(V, \omega)$    the  symplectic  group.
   %${\rm Sp}_{n}(\bb{R})$ is a connected Lie group  with  $\pi_{1}({\rm Sp}_{n}(\bb{R}))=\bb{Z}$.
    Recall that the metaplectic group ${\rm Mp}_{n}(\bb{R})$ is   the unique  connected (double) covering of ${\rm Sp}_{n}(\bb{R})$ (cf.   \cite{Hab}).      Given a symplectic manifold $(M^{2n}, \omega)$,  we denote by ${\rm Sp}(M)$ the ${\rm Sp}_n(\mathbb{R})$-principal bundle
  of   symplectic  frames, i.e.   frames $e_1, \cdots , e_n, f_1, \cdots, f_n$  such that  $\omega(e_i,e_j)= \omega(f_i, f_j)=0,\, \omega(e_i, f_j) = \delta_{ij}$.
   \begin{definition}
 \textnormal{A {\it metaplectic structure} on a symplectic manifold $(M^{2n}, \omega)$  is a  ${\rm Mp}_n(\mathbb{R})$-equivariant lift  of the symplectic frame bundle ${\rm Sp}(M)\to M$ with respect to the double covering $\rho : {\rm Mp}_{n}\bb{R}\to{\rm Sp}_{n}\bb{R}$.}
 \end{definition}

  The  obstruction to the existence of a metaplectic  structure, is exactly the same as in the case of a spin structure on a Riemannian manifold. Recall that the first Chern class of $(M^{2n}, \omega)$  is defined as the first Chern class of   $(TM, J)$, where $J$ is a $\omega$-compatible almost complex structure.  Since the space of $\omega$-compatible almost complex structures is contractible, $c_{1}(TM, J)$ is independent of $J$ (cf. \cite{Hab}).
  \begin{prop}\textnormal{(\cite{Hab})}\label{meta}
 A symplectic manifold $(M^{2n}, \omega)$  admits    a metaplectic structure  if and only if $w_2(M)= 0$ or  equivalently,
 the first Chern class $c_{1}(M)$  is even.  In this  case,  the set of metaplectic structures on $(M^{2n}, \omega)$ stands in a bijective correspondence with $H^{1}(M; \bb{Z}_{2})$.
\end{prop}

\section{Spin  structures   on  compact  homogeneous pseudo-Riemannian    manifolds}\label{homogspin}

 %\section{Invariant   structures on homogeneous pseudo-Riemannian manifolds}
 \subsection{Invariant spin structures}
  In the following we shall   examine spin structures   on compact  and  simply-connected homogeneous  (pseudo)Riemannian manifolds $(M^{n}=G/L, g)$, endowed with an  almost  effective action   of a connected Lie group $G$. % i.e. $L$ contains no non-trivial (non-discrete) normal subgroup of $G$. %  In this case the isotropy representation $\vartheta : L\to \Aut(T_{o}M)$ $(o=eL)$ is almost faithful (has discrete centre).
     By an old theorem of Montgomery \cite{Mon}  it is known that given a compact and simply-connected  homogeneous space $G/L$,  one can always assume that $G$ is a compact, connected and simply-connected Lie group and  $L$ is a closed connected subgroup.  Finally,    %mention that since $G$ is compact,
 up to   a  finite  covering, $G$ is  a  direct product   of  a  torus  ${\rm T}^a$     with  a  simply-connected    compact  semisimple  Lie   group  $G'$,     which  still acts  transitively  on $M$.    Hence,    $M=G/L$  turns out to be isometrically isomorphic to the   coset  space  $G'/L',\,  L' = L \cap G'$, of a compact, connected and simply-connected semisimple Lie group, modulo  a closed connected subgroup. This  is the setting that we will use in the sequel, especially in Sections \ref{flags} and \ref{Cspace}.  However, let us  recall first some details of a  more special  case.

 Any   homogeneous manifold  $M = G/L$   with {\it  compact }  stabilizer $L$    admits   a  reductive  decomposition, i.e. an orthogonal splitting $\fr{g}=\fr{l}+\fr{q}$ with $\Ad_{L}\fr{q}\subset\fr{q}$.  Then, a $G$-invariant  Riemannian metric on $M$ is defined by an $\Ad_{L}$-invariant scalar product  $g_{o}$ in $\mathfrak{q}= T_{eL}M$.  If  the  isotropy  representation $\vartheta : L\mapsto \Ad_{L}|_{\fr{q}}$   is  reducible,  then $M=G/L$     admits  also  an invariant  pseudo-Riemannian metric,  as  the  following lemma  shows.

 \begin{lemma}\label{alek1}
    Let  $M= G/L$  be   a homogeneous manifold   with   compact stabilizer  $L$.   Then, \\
  (i)   $M $ admits  a $G$-invariant  metric $g_M$ of  signature $(p, n-p)$,
   if  and only if  the  tangent space  $\mathfrak{q}$ admits  a $\vartheta(L)$-invariant  $p$-dimensional  subspace  $\mathfrak{q}_-$. In this case,  the  metric $g$ is   defined  by
    a  $\vartheta(L)$-invariant  pseudo-Euclidean   metric  $g:=   - g_0 |_{\mathfrak{q}_-} \oplus  g_0| _{\mathfrak{q}_+} $ on $\fr{q}$,   where   $g_0$ is an $\Ad_{L}$-invariant  Euclidean  metric    on $\mathfrak{q}$, and   $\mathfrak{q}_+ $ is  the  $g_0$-orthogonal  complement  to $\mathfrak{q}_-$.
 Conversely,  any  $G$-invariant pseudo-Riemannian  metric  $g$  is   described    as   above.\\
 (ii) A $G$-invariant  decomposition    $TM = \eta_- \oplus \eta_+$  into a direct sum of    $g$-time-like  and    $g$-space-like   subbundles (where $g$ is the pseudo-Euclidean metric given above),  is  unique,    if  and only   if   the   decomposition  $\mathfrak{q} = \sum_{j=1}^d \mathfrak{p}_j$ of the   $\vartheta(L)$-module $\fr{q}=T_{eL}M$ into  irreducible submodules  is unique, i.e.   the submodules $\mathfrak{p}_j$ are mutually non-equivalent.
\end{lemma}
 \begin{proof} (i) The    restriction        $g$  of  an invariant pseudo-Riemannian metric  of  signature $(p,q)$ to  $\mathfrak{q} = T_oM$ $(o = eL)$  is a $\vartheta(L)$-invariant  pseudo-Euclidean  metric.   Since the isotropy  group  $\vartheta(L)$ is  compact, it preserves an  Euclidean metric   $g_0$  in  $\mathfrak{q}$.  Hence  $\vartheta(L)$ commutes  with  the symmetric  endomorphism   $A:= g_o^{-1}\circ g $.
  Denote  by  $\mathfrak{q}_+$  (resp. $\mathfrak{q}_{-}$)  the  sum  of $A$-eigenspaces  with positive (resp. negative) eigenvalues.  Then,
$  \mathfrak{q} = \mathfrak{q}_{-} + \mathfrak{q}_{+}$   is  an $\Ad_{L}$-invariant  orthogonal  decomposition    such   that
  $g|_{\mathfrak{q}_+} >0$ and  $g|_{\mathfrak{q}_-} <0  $.  This defines  a $G$-invariant  decomposition  $TM = \eta_- \oplus \eta_+$ into
    direct  sum of   time-like   and  space-like subbundles. % given by  $\eta_{\pm}:=G\times_{\vartheta_\pm}\fr{q}_{\pm}$, where $\vartheta_{\pm}:={\rm pr}_{\fr{q}_{\pm}}\circ\vartheta : L\to \Ad_{L}|_{\fr{q}_{\pm}}$ such that $\vartheta=\vartheta_{-}\oplus\vartheta_{+}$.
      The converse is obvious. \\
  (ii)  Assume  that for any $1\leq i\neq j\leq d$, the irreducible   submodules  $\fr{p}_{i}, \fr{p}_{j}$ are inequivalent, i.e. $\fr{p}_{i}\ncong \mathfrak{p}_j$.   Then,  the   decomposition
     $\mathfrak{q} = \sum_{j=1}^s \mathfrak{p}_j$ is  orthogonal  with  respect to   the  pseudo-Euclidean  metric $g$  (and more general, any  $\vartheta(L)$-invariant metric). By Schur's  lemma, the   endomorphism   $A|_{\mathfrak{p}_j} = \mu_j \mathrm{id}$  is  scalar and  the    restriction  $g|_{\mathfrak{p}_j}$ is positively,   or negatively  defined.  But then,
     $\mathfrak{q}_+ =\sum_{j, \, g_{\mathfrak{p}_j} >0} \mathfrak{p}_j$ is  uniquely   determined and our claims follows.  \end{proof}

 \begin{prop}\label{alek2} Let   $(M = G/L, g_M)$ be  a   homogeneous  pseudo-Riemannian manifold  with  compact  stabilizer  and reductive  decomposition  $\mathfrak{g}= \mathfrak{l}+ \mathfrak{q}$.  Then, $M=G/L$  is  time-oriented  (resp. space-oriented),    if  and only if    the linear   group  $\vartheta(L)|_{\mathfrak{q}_-}$  (resp.  $\vartheta(L)|_{\mathfrak{q}_-}$) is unimodular,  where   $\mathfrak{q}= \mathfrak{q}_- + \mathfrak{q}_+$ is  a  $\vartheta(L)$-invariant  orthogonal decomposition   which  induces the $G$-invariant  splitting  of $TM$  into   time-like  and  space-like  subbundles. In  particular, this is  the  case if  $L$ is  connected.
 \end{prop}
% \begin{proof}
 %Assume  that  $\vartheta(L)_{\mathfrak{q}_-}$ is unimodular i.e.  $\vartheta(L)_{\mathfrak{q}_-} \subset  \SO(\mathfrak{q}_-)\cong \SO_{p}$. Then it  preserves the  set $Fr(\mathfrak{q}_-)$ of right oriented frames in $\mathfrak{q}_-$. Using    left translations,  $Fr(\mathfrak{q}_-)$  can be  extended   to  the   principal bundle
%  $Fr(\eta)_-  = G \times_{\vartheta(L)} Fr(\mathfrak{q}_-)$ of  right  oriented  frames on  $\eta_-$,   which means  that  $\eta_-$ is   oriented. To complete the proof, it   remains  to note  that   orientability  of  a   time-like, or  space-like   distribution,   does not  depend   on   a  particular  distribution.
%  \end{proof}

%\begin{prop}  Let $(M^{n}=G/L, g, \fr{g}=\fr{l}+\fr{q})$  be a  time-oriented  and  space-oriented  reductive homogeneous pseudo-Riemannian manifold of signature $(p, q)$. Then, the  (connected )   Lie group $G$ naturally acts  on  the bundle  $\SO^+(M)$ of positively  time-  and  space-oriented orthonormal frames  of $M$,  which is  identified  with the homogeneous bundle $P^{+}:=\SO(M)^{+} =G \times _{L} \SO^{+}(\fr{q})$.
%\end{prop}

 \begin{definition}
\textnormal{A spin   structure $\tilde{\pi} :  Q \to M$ on  a   homogeneous pseudo-Riemannian manifold   $(M = G/L,g)$ is   called
 {\it $G$-invariant} if  the natural  action  of  $G$  on the  bundle  $P:=\SO(M)=G\times_{L}\SO(\fr{q})$  of positively oriented orthonormal  frames of $M$,   can be  extended  to  an  action  on the $\Spin_{p, q}\equiv\Spin(\fr{q})$-principal bundle  $ \tilde{\pi}: Q \to  M$.}% Similarly for ${\rm spin}^{+}$ structures.}
  \end{definition}
 Invariant spin structures on reductive homogeneous spaces can be described in terms of lifts of the isotropy representation into the spin group. In particular,
   \begin{prop}\label{springer}\textnormal{(\cite{Cahen2})}
Let   $(M=G/L,g)$ be an oriented   homogeneous pseudo-Riemannian manifold with a reductive decomposition $\fr{g}=\fr{l}+\fr{q}$.  Given a lift  of  the isotropy representation onto the spin group $\Spin(\fr{q})$, i.e. a homomorphism $\widetilde{\vartheta}: L \to  \Spin(\fr{q})$ such that $\vartheta=\Ad\circ\widetilde{\vartheta}$,  %  where $\Ad : \Spin(\fr{q})\to\SO(\fr{q})$ is the double covering,
then $M$ admits a $G$-invariant spin structure given by $Q =G \times_{\widetilde{\vartheta}} \Spin(\fr{q})$. %such that $\vartheta=\lambda\circ\widetilde{\vartheta}$,
%{\small{\[
%   \xymatrix{
%            & \Spin(\fr{q}) \ar[d]^{\lambda} \\
 %      L\ar@{.>}[ur]^{\widetilde{\vartheta}} \ar[r]_{\vartheta}       & \SO(\fr{q}) }\quad\quad \quad  \xymatrix{
%            }
%\]}}
  Conversely, if $G$ is simply-connected and $(M=G/L, g)$ has a spin structure, then $\vartheta$ lifts to $\Spin(\fr{q})$, i.e. the spin structure is $G$-invariant. Hence in this case there is a one-to-one correspondence between the set of spin structures on $(M = G/L,g)$ and the set of lifts of $\vartheta$ onto $\Spin(\fr{q})$.% If  in addition $M$ is  simply-connected and such a lift exists, then it   will be unique.
\end{prop}

   The above discussion, in combination with  Proposition  \ref{alek2},  yields that
   \begin{corol} \label{basics1}  Let  $(M= G/L, g)$ be  a   homogeneous pseudo-Riemannian manifold of  signature $(p,q)$ of a connected  Lie group $G$ modulo a compact  connected Lie subgroup  $L$.   Then,  $M=G/L$  admits    a  spin  structure  if  and only if  $w_2(M) =0$.  If  $M$  admits  an invariant  almost  complex  structure $J$,   then   this   condition is  equivalent  to  say that  $c_1(M,J)$ is  even in $H^{2}(M; \bb{Z})$. % Moreover,  if such a structure exists  and  $M$ is  simply-connected, then the spin structure  will be  unique  and  $G$-invariant.
 \end{corol}

\subsection{Homogeneous    fibrations and spin structures}
    Let $L\subset H\subset G$   be   compact  connected  subgroups   of a   compact  connected Lie  group $G$.
    Then, $\pi : M= G/L \to F =G/H$ is a homogeneous  fibration   with   base space $F=G/H$  and  fibre $H/K$.
     We   fix  an   $\Ad_{L}$-invariant reductive decomposition for $M = G/L$
  \[
    \mathfrak{g}= \mathfrak{l} + \mathfrak{q} =  \mathfrak{l} +   (\mathfrak{n}+\mathfrak{m}),   \quad \fr{q}:=\fr{n}+\fr{m}=T_{eL}M
    \]
    such that   $ \mathfrak{h}= \mathfrak{l} + \mathfrak{n}$ is  a  reductive  decomposition of $H/L$  and
   $  \mathfrak{g}= \mathfrak{h} +\mathfrak{m} =  (\mathfrak{l} +\mathfrak{n}) + \mathfrak{m}$
     is   a reductive decomposition of $F=G/H$.
     An $\Ad_{L}$-invariant (pseudo-Euclidean) metric  $g_{\mathfrak{n}}$ in  $\mathfrak{n}$  defines  a (pseudo-Riemannian) invariant metric in $H/L$  and an $\Ad_{H}$-invariant (pseudo-Euclidean) metric  $g_{\mathfrak{m}}$ in $\fr{m}$  gives rise to a (pseudo-Riemannian) invariant metric in  the  base $F=G/H$.  Then,   the   direct  sum metric $g_{\mathfrak{q}} = g_{\mathfrak{n}} \oplus g_{\mathfrak{m}}$  in $\mathfrak{q}$,   induces  an invariant   pseudo-Riemannian  metric
     in $M =G/L$    such  that  the projection  $\pi : G/L \to G/H$ is  a pseudo-Riemannian  submersion with  totally geodesic  fibres.
   % Recall  that     for  a homogeneous manifold $M$ with  connected  stabilizer    the  existence of  a  spin  structure   is   equivalent  to   $w_2(M)=0$  does not  depend on  invariant metric.

\begin{prop}\label{genw2} \textnormal{(see also \cite{Bo} for more general homogeneous bundles)}
     Let  $G$  be  a  compact, connected   Lie  group   and  $L \subset H\subset G$    connected,   compact  subgroups. Consider the   fibration  $\pi : M=G/L \to F= G/H$. Then \\
    % (i) The total space $M=G/L$ and the fiber $N=H/L$ are stably parallelizable.\\
     (i) The bundles $i^{*}(TM)$ and $TN$ are stably equivalent.  In particular, the Stiefel-Whitney classes of the fiber $N=H/L$ are in the image of the homomorphism $i^{*} : H^{*}(M; \bb{Z}_{2})\to H^{*}(N; \bb{Z}_{2})$, induced by the inclusion map $i : N=H/L\hookrightarrow M=G/L$.\\
     (ii) It is  $w_{1}(TM)=0$ and $w_2(TM) = w_2(\tau_{N}) + \pi^* (w_2(TF))$, where $\tau_{N}$ is the tangent bundle along the fibres.
        \end{prop}
\begin{proof}
%Since    the   stability  groups   of the  manifolds  $ G/L,  H/L, G/H$  are  connected, they are all oriented, $w_{1}(G/L)=w_{1}(H/L)=w_{1}(G/H)=0$. %  any invariant  metric is  space- and time-oriented, in particular
    (i) We    choose     an $G$-invariant  pseudo-Riemannian    metric   $g=g_{\fr{m}}\oplus g_{\fr{n}}$ on $M=G/L$ as  above,  such  that  $\pi : G/L \to G/H$ is  a   pseudo-Riemannian  submersion.  Then,   the tangent   bundle   admits the following orthogonal  decomposition:
          \begin{equation}\label{decomp1}
TM=G\times_{L}\fr{q}=G\times_{L}(\fr{n}+\fr{m})=(G\times_{L}\fr{n})\oplus(G\times_{L}\fr{m}):=\tau_{N}\oplus \pi^{*}(TF),
\end{equation}
   where we identify  the  pull-back of     the tangent   bundle  $TF=G\times_{H}\fr{m}$  with the homogeneous vector bundle $G\times_{L}\fr{m}$, i.e. $\pi^{*}(TF)\cong G\times_{L}\fr{m}$. Moreover, the tangent bundle along the fibres $\tau_{N}:=G\times_{L}\fr{n}\to G/L$ is   the homogeneous  vector bundle  whose fibres  are the tangent spaces $T_{eL}N\cong\fr{n}$ of the fibres $\pi^{-1}(x)\cong H/L:=N$ $(x\in F)$  (cf. \cite[6.7, 7.4]{Bo}).    Then,   $TN=H\times_{L}\fr{n}\cong i^{*}(\tau_{N})$. On the other hand
   \[
   (i^{*}\circ\pi^{*})(TF)=(\pi\circ i)^{*}(TF)=\epsilon^{\dim F}
   \]
   where $\epsilon^{t}$ is the trivial real vector bundle of rank $t$. Thus
  $
  i^{*}(TM)=\epsilon^{\dim F}\oplus TN.
$
  This proves our first claim, and consequently,  by  the naturality  of Stiefel-Whitney classes we get that
  \[
 w_{j}(i^{*}(TM))=i^{*}(w_{j}(TM))=w_{j}(TN),% \quad \forall j=1, \ldots, \dim M
 \]
 or equivalently,  $i^{*}(w_{j}(M))=w_{j}(N)$, where $i^{*} : H^{*}(M; \bb{Z}_{2})\to H^{*}(N; \bb{Z}_{2})$, see also \cite[7.4, p.~480]{Bo}.\\
   %  On the other hand, $\mathcal{H}M = (G\times_{L}\fr{m}) \simeq \pi^*TF$  it  the   horizontal   subbundle   which is  isomorphic  to  the  pull back of
    %the tangent   bundle  $TF$.
 (ii)     Since  both $L\subset H$ are connected,  the  vector bundles $\tau_{N}$ and  $TF$  are oriented, $w_1(\tau_{N}) = w_1(TF) =0$. The same is true for $TM$, i.e. $w_{1}(TM)=w_{1}(\tau_{N})+\pi^{*}(w_{1}(TF))=0$, which also follows by the connectedness of $L$.  Now, the relation for $w_{2}(TM)$ is  an immediate consequence of the  Whitney product formula $w_{s}(E\oplus E^{\perp})=\sum_{i=1}^{s}w_{i}(E)\smile w_{s-i}(E^{\perp})$:
          \begin{eqnarray*}
w_{2}(TM)&=&w_2(\tau_{N}\oplus \pi^*(TF)) =w_{2}(\tau_{N})+w_{1}(\tau_{N})\smile w_{1}(\pi^{*}(TF))+w_{2}(\pi^{*}(TF))\\
&=&w_{2}(\tau_{N})+w_{2}(\pi^{*}(TF))=w_{2}(\tau_{N})+\pi^{*}(w_{2}(TF)).
\end{eqnarray*}
%Now, it is easy to see that  $T^{\rm ver}M \cong G\times_{L}\fr{n}$ coincides with the tangent bundle of the fiber $H/L$, and our claim follows.
     % \[
  % w_2(N) = w_2(T^{\rm ver}M \oplus \pi^*(TF)) = w_2(T^{\rm ver}M) + w_1(T^{\rm ver}M \smile \pi^* (w_1(TM))+ \pi^* w_2(TM)= w_2 (T^vN ) + \pi^* w_2(TM).
   %\]
 %    It  remains  to prove   that  $w_2(T^vN)  = w_2(H/L)$. {\color{red}  Why  this  is  true ???}
     \end{proof}
%An important consequence of this proposition reads as follows:
\begin{corol}\label{springer2}Let $N\overset{i}{\hookrightarrow}M\overset{\pi}{\to}F$  be a homogeneous fibration  as in Proposition \ref{genw2}. Then: \\
(i) If $F=G/H$  is  spin,  then $M=G/L$ is spin if and only if $N=H/L$ is spin. \\
        (ii) If  $N=H/L$ is spin, then $M=G/L$ is spin if and only  if  $w_{2}(G/H)\equiv w_{2}(TF)\in \ker \pi^{*}\subset H^{2}(F; \bb{Z}_{2})$, where $\pi^{*} : H^{2}(F; \bb{Z}_{2})\to H^{2}(M; \bb{Z}_{2})$ is the induced homomorphism by $\pi$.  In particular, if  $N$ and $F$ are spin, so is $M$  with  respect  to   any  pseudo-Riemannian metric.
         \end{corol}
         \begin{proof}
         (i)
      Under our assumptions,  the
      existence  of a     spin  structure   on all manifolds $G/L, H/L, G/H$  is  equivalent  to   the vanishing of the associated second Stiefel-Whitney class $w_2$.  Now,   for the injection $i : N \hookrightarrow M$ it is $i^{*}\circ i_{*}=\Id$ and $i^{*}(\tau_{N})=TN$.  Hence,   we also conclude that $\tau_{N}=i_{*}(TN)$ and (\ref{decomp1}) takes the form   $TM=\tau_{N}\oplus \pi^{*}TF=i_{*}(TN)\oplus \pi^{*}TF$.      Since $w_{1}(TF)=w_{2}(TF)=0$, by   the second part of Proposition \ref{genw2} we obtain that (see also \cite{Gadea})
      \begin{equation}\label{secdec}
      w_{2}(TM)=w_{2}(\tau_{N})=w_{2}(i_{*}(TN))=i_{*}(w_{2}(TN)).
      \end{equation}
 % A proof    for the last equality  is given as follows: Assume in the opposite that $w_{2}(i_{*}(TN))\neq i_{*}(w_{2}(TN))$.  Then,
%$  i^{*}(w_{2} (i_{*}(TN)))=w_{2}(i^{*}(i_{*}(TN)))=w_{2}(TN)\neq i^{*}(i_{*}(w_{2}(TN)))=w_{2}(TN)$,   a contradiction.
Assume now that $M=G/L$ is spin, $w_{2}(TM)=0$. Then, by  (\ref{secdec}) it follows that $0=w_{2}(i_{*}(TN))=w_{2}(\tau_{N})$, so considering the pull-back via $i$ we get $i^{*}(w_{2}(\tau_{N}))=0$, or equivalently $w_{2}(i^{*}\tau_{N})=0$ which gives $w_{2}(TN)=0$, i.e. the fiber $N=H/L$ is spin. Conversely, assume that  $w_{2}(TN)=0$.  Then, $i_{*}(w_{2}(TN))=0$ and by (\ref{secdec}) we also get
$w_{2}(i_{*}(TN))=0$ or equivalently $w_{2}(\tau_{N})=0$.  Another way to prove this is as follows: if $w_{2}(TN)=0$, then $w_{2}(i^{*}\tau_{N})=0$ and so $i_{*}(w_{2}(i^{*}\tau_{N}))=w_{2}(\tau_{N})=0$.
Hence, one can easily get the desired result: $w_{2}(TM)=0$, i.e. $M=G/L$ is spin.      The proof of (ii) is easy.  \end{proof}
%(ii) If $N=H/L$ is spin, using arguments as above and applying Proposition \ref{genw2}, we obtain that $w_{2}(TM)=\pi^{*}(w_{2}(TF))$, so $w_{2}(TM)=0$ if and only if $ w_{2}(TF)\in \ker \pi^{*}$.    \end{proof}

\begin{remark}
\textnormal{Using Proposition \ref{genw2},  notice that  if $M$ is $G$-spin and $N$ is $H$-spin, then
  $\pi^{*}(w_{2}(TF))=w_{2}(\pi^{*}(TF))=0$, which  in general does not imply the relation $w_{2}(TF)=0$, i.e. $F$ is not necessarily spin.
For example, consider the  the {\it Hopf fibration} 
\[
\Ss^{1}\to\Ss^{2n+1}=\SU_{n+1}/\SU_{n}\to\bb{C}P^{n}=\SU_{n+1}/\Ss(\U_{1}\times\U_{n}).
\]
Although the sphere $\Ss^{2n+1}$ is   a spin manifold for any $n$ (its tangent bundle is stably trivial), recall that $\bb{C}P^{n}$ is spin only for $n=\text{odd}$, see for instance Example \ref{cpn}.}
\end{remark}

\subsection{Invariant metaplectic   structures}

   Recall  that     a   compact   homogeneous  symplectic  manifold   $(M =G/H,\omega)$   is  a  direct product  of  a    flag manifold $F $, i.e.   an adjoint orbit  of  a  compact semisimple  Lie group  (see Section \ref{flags} below)  and  a solvmanifold    with    an invariant  symplectic    structure.  In particular,
   any   simply-connected  compact  homogeneous symplectic  manifold  $(M=G/H, \omega)$   is symplectomorphic   to  a  flag manifold.
    Hence, based on  Proposition \ref{meta}   we get that
    \begin{prop}
    Simply-connected  compact   homogeneous  symplectic manifolds admitting a   metaplectic   structure,    are exhausted  by  flag manifolds   $F =G/H$ of   a   compact simply-connected semisimple Lie  group $G$, whose  second  Stiefel-Whitney   class vanishes, i.e. $w_2(F)=0$.  For  any  (invariant)  symplectic   form  such  a  structure is   unique.
   \end{prop}

% The  classification of flag manifolds  with $w_2 =0$ is   given in   the following section.

  \section{Spin  structures on   flag manifolds}\label{flags}

\subsection{Basic facts about flag manifolds}\label{bflags}    Recall  that   a  flag manifold   is   an   adjoint   orbit $F =\mathrm{Ad}_G w =   G/H$ $(w\in\fr{g}=T_{e}G)$ of  a  compact connected semisimple  Lie  group $G$.  Any flag manifold  is a direct product of  flag  manifolds  of    simple   groups.      The    stabilizer $H$ is    the  centralizer  of a  torus in $G$,   hence connected.  Two   flag manifolds $G/H_{1}, G/H_{2}$ of  $G$   are called  {\it isomorphic} if  the  stability  subgroups  $H_{1}$ and $H_{2}$  are  conjugated in $G$.

        Let $F=G/H$ be a flag manifold of  a compact semisimple Lie group $G$ and let us denote by   $B$   the    Killing  form   of  the Lie  algebra  $\mathfrak{g}$. Consider the   $B$-orthogonal  reductive  decomposition
     \[
   \mathfrak{g} = \mathfrak{h} + \mathfrak{m} = (Z(\mathfrak{h}) +  \mathfrak{h}') + \mathfrak{m}
   \]
     where $Z(\mathfrak{h})$ is  the  center   and $\mathfrak{h}'$  is  the  semisimple part of  $\mathfrak{h}$.  %We shall write $H'$ for the connected Lie subgroup with Lie algebra $\fr{h}'$.
      The $H$-module   $\mathfrak{m}$   coincides with  the
         the tangent space  $T_{eH}F$ and  the isotropy representation $\vartheta : H \to \Aut(\fr{m})$ is equivalent with   $\Ad_{H}|_{\fr{m}}$.

\smallskip
            {\bf Notation.}  Let us fix some notation for the following of the article.   We denote  by   $\fr{a}$ a Cartan  subalgebra  of $\fr{h}$ (hence  also  of $\mathfrak{g}$)  and  by
                    $\mathfrak{a}^{\mathbb{C}}$ its complexification.  It   defines   the  root  space  decomposition
                     $\fr{g}^{\bb{C}}=\fr{a}^{\bb{C}}+ \mathfrak{g}(R):=  \fr{a}^{\bb{C}}  + \sum_{\al\in R}\fr{g}_{\al},$
                     where $R$  is  the  root  system of  $(\mathfrak{g}^{\mathbb{C}}, \mathfrak{a}^{\mathbb{C}})$.    Set
                     \[
                     \mathfrak{a}_0:= i \mathfrak{a}, \quad \fr{z}=Z(\fr{h}), \quad \mathfrak{t}:=  i \fr{z}\subset \mathfrak{a}_0.
                     \]
                      The  Killing  form $B$  of $\mathfrak{g}^{\mathbb{C}}$  induces  an Euclidean metric $( \ , \  )$ in $\mathfrak{a}_0$,  which  allows us to  identify the spaces  $\mathfrak{a}_0^*$ and $\mathfrak{a}_0$. % Note  that   the roots   (and   fundamental  weights)    are  elements of   $\mathfrak{a}_0^* \cong \mathfrak{a}_0$.   Given a root $\al\in R$, let $H_{\al}\in\fr{a}^{\bb{C}}$ be the vector  defined by $B(H_{\al}, X)=\al(X)$ for any $X\in\fr{a}^{\bb{C}}$. Set  $h_{\al}:=2H_{\al}/B(H_{\al}, H_{\al})=2\al/(\al, \al)$ such that $\al(h_{\al})=2$.
         The    root  system $R_H \subset R$   of  $(\mathfrak{h}^{\mathbb{C}}, \mathfrak{a}^{\mathbb{C}})$  consists of  roots   which  vanish  on $ Z(\mathfrak{h})$.  Roots  in $R_F:= R \setminus R_H$ are  called  {\it complementary}.

      We fix  a  fundamental   system  $\Pi_W =  \{\al_1, \cdots, \al_u \}  $ of  $R_H$ and   extend it   to  a  fundamental  system $\Pi = \Pi_W \sqcup \Pi_B$ of $R$, such that $\ell:=\rnk G=u+v$. The    roots   from  $\Pi_W$  are  called  {\it white}  and   roots  from $\Pi_B = \{\beta_1, \cdots, \beta_v\}$  are called  {\it black}. Graphically,   the   decomposition $\Pi= \Pi_W \sqcup \Pi_B$ is   represented  by  painted  Dynkin  diagram (PDD),  that is   the   Dynkin  diagram of $\Pi$  with  black nodes   associated  with   black  roots.

       A PDD  determines  the      flag manifold  $F = G/H$   together  with  an invariant  complex  structure $J$.
       The   semisimple  part  $\mathfrak{h}'$ of   $\mathfrak{h} = \mathfrak{h}' + i \mathfrak{t}$ is  defined  as   the   subalgebra of $\mathfrak{g}$
        associated  with white   subgiagram  $\Pi_W$   and    the  center $i \mathfrak{t}$ is   defined   by
        $    \fr{t}=\fr{z}\cap\fr{a}_{0}=\{h\in\fr{a}_{0} : (h, \al_{i})=0,  \, \forall \ \al_{i}\in\Pi_{W}\}$.  A basis of $\fr{t}^{*}\cong \mathfrak{t}$  is given by the ``black"  fundamental weights $(\Lambda_1, \cdots, \Lambda_v)$  associated  to  the black  simple  roots $\beta_i \in \Pi_B$. These are linear forms   defined by            the conditions
\[
 (\Lambda_i| \beta_j):= \frac{2(\Lambda_i, \beta_{j})}{(\beta_{j}, \beta_{j})} = \delta_{ij}, \quad (\Lambda_i| \alpha_k) =  0.
 \]
  % Such ``black"  fundamental weights $(\Lambda_1, \cdots, \Lambda_v)$    form a basis  of  the  space
 %$\fr{t}^{*}\cong \mathfrak{t}$ .
\noindent   To  define   the  complex  structure,  we denote   by $R^+$  the positive  roots,  defined  by $\Pi$ and   set
   $R_H^+ = R_H \cap R^+,\,  R_F^+ = R_F \cap R^+$. Then,  the  decomposition  $R = (-R_{F}^-) \sqcup R_H \sqcup R_{F}^+$
   of  $R$ into a  disjoint  union of   closed  subsystems
   induces the generalized Gauss  decomposition   of  $\mathfrak{g}^{\mathbb{C}}$, i.e. the direct sum decomposition  % into  direct  sum of   three    subalgebras ( generalized Gauss  decomposition):
$$   \mathfrak{g}^{\mathbb{C}} = \mathfrak{n}^- + \mathfrak{h}^{\mathbb{C}} + \mathfrak{n}^+ = \mathfrak{g}(-R_F^+) + (\mathfrak{a}^{\mathbb{C}}+ \mathfrak{g}(R_H))  + \mathfrak{g}(R_F^+),$$ with $\fr{n}^{\pm}:= \mathfrak{g}(\pm R_F^+)$.
Thus, the  complexification of  $\fr{g}=\fr{h}+\fr{m}$  is   given  by
 $    \mathfrak{g}^{\mathbb{C}} = \mathfrak{h}^{\mathbb{C}} + ( \mathfrak{g}(-R_F^+ ) + \mathfrak{g}(R_F^+))$
and the invariant  complex  structure $J$ is  defined  by  the  condition  that $\mathfrak{n}^{\pm}$   are $\pm i$ eigenspaces
of  $J$ in  the complexified  tangent  space  $T_o^{\bb{C}}F = \mathfrak{m}^{\bb{C}}= \mathfrak{n}^- + \mathfrak{n}^+$.

 The   subalgebras $\mathfrak{p}  = \mathfrak{h}^{\mathbb{C}} + \mathfrak{n}^+, \, \mathfrak{p}^-  = \mathfrak{h}^{\mathbb{C}} + \mathfrak{n}^-  $
         are  called  opposite  parabolic subalgebras associated  with a PDD.  The connected  group of  authomorphisms  of  the  complex manifold
         $(F = G/H, J)$ is  the   complex  group $G^{\mathbb{C}}$  and   the  stability  subgroup  of  the point   $o = eH$ is   the parabolic   subgroup $P$,  generated  by  the parabolic  subaglebra  $\mathfrak{p}$.  So  as a  complex manifold,  the flag manifold  is  identified  with   the   quotient  $F = G^{\mathbb{C}}/P$. Moreover, any parabolic subalgebra is conjugate to a subalgebra of the form $\fr{p}  = \mathfrak{h}^{\mathbb{C}} + \mathfrak{g}(R_F^+)$
  associated   to  a PDD. To summarize
   \begin{prop} \textnormal{(\cite{AP, Alek, Ale1})} \label{complex}
 Invariant  complex  structures  on a  flag manifold  $F = G/H$ bijectively  correspond   to   extensions   of a  fixed fundamental  system $\Pi_W$ of  the subalgebra  $\mathfrak{h}^{\mathbb{C}}$  to a  fundamental  system  $\Pi= \Pi_W \cup \Pi_B$ of    $\mathfrak{g}^{\mathbb{C}}$, or equivalently, to parabolic subalgebras $\mathfrak{p}$ of $\fr{g}^{\bb{C}}$  with reductive part $\fr{h}^{\bb{C}}$.
  \end{prop}

 %                         Let $R^+ = R_H^+ \cup R_F^+ =( R_H \cap R^+ )\cup (R_F \cap R^+)$     be      the     system of positive  roots   with respect to $\Pi$ and $B$.

 \subsection{T-roots and   decomposition of isotropy   representation }
                \begin{definition} \textnormal{ A {\it $T$-root} is the restriction
            on $\mathfrak{t}$ of a complementary root $\al\in R_{F}=R\backslash R_{H}$, via the linear map  $\kappa : \mathfrak{a}^*  \to  \mathfrak{t}^*, \,   \alpha \mapsto  \alpha|_{\mathfrak{t}}$. We denote by  $R_T:= \kappa(R_F) \subset \mathfrak{t}^*$  the set of  all $T$-roots.}
         \end{definition}

 The  system   $R_{T} \subset \mathfrak{t}^*$  of  $T$-roots is not necessary  a  root  system in abstract  sense, but it  has  many properties of   a root  system (for  details see  \cite{Alek, Ale1, Gr1, Gr2}). The integer $v:=\sharp(\Pi_{B})$ is called the {\it rank} of $R_{T}$.  Consider  now the weight lattice associated to $R$, that is
\[
\cal{P}=\{ \Lambda \in \mathfrak{a}_0^*  \  : \    (\Lambda | \alpha)   \in  \mathbb{Z}, \   \forall \alpha \in R \} = \mathrm{span}_{\mathbb{Z}}(\Lambda_1, \cdots , \Lambda_\ell) \subset  \mathfrak{a}_0^*
\]
 and set  $\cal{P}_T:= \{ \lambda \in \cal{P} :  (\lambda, \alpha)=0,  \ \forall \ \alpha \in R_H  \}  \subset \cal{P}$.
\begin{lemma}\label{Tlattice}   $\cal{P}_T$  vanishes  on  the Cartan  subalgebra $\mathfrak{a}' = \mathfrak{a}\cap \mathfrak{h}'$  of  the  semisimple part  $\mathfrak{h}'$ and   defines   a lattice   in   $\mathfrak{t}^*$, which is called  $T$-weight lattice. It is generated by the fundamental weights $\Lambda_{1}, \cdots, \Lambda_{v}$ corresponding to the black simple roots $\Pi_{B}=\Pi\backslash\Pi_{W}$.
\end{lemma}

\begin{proof}  A  form  $\lambda \in \cal{P}_T$ is orthogonal  to  $ \mathrm{span}_{\mathbb{R}}(R_H)=  i (\mathfrak{a}')^*. $  Thus  it vanishes  on
$B^{-1} \circ ( \mathfrak{a}')^* = \mathfrak{a}'$   and  defines  a  lattice  $\cal{P}_T$ in  $\mathfrak{t}^*$  with  $\mathrm{rank} (\cal{P}_T) = \dim \mathfrak{t} =v$.  \end{proof}

 Recall  that  there  is  a natural  1-1  correspondence  between  $T$-roots $\xi\in R_{T}=\kappa(R_{F})$ and
  irreducible  $H$-submodules  $\mathfrak{f}_{\xi}$ of the  complexified tangent  space  $\mathfrak{m}^{\mathbb{C}} =T_o^{\mathbb{C}}F$, and  also  between  positive $T$-roots  $ \xi \in  R_T^+ =\kappa(R_{F}^{+})$ and  irreducible $H$-submodules
  $\mathfrak{m}_{\xi} \subset \mathfrak{m}$,  given by  (see also \cite{Sie, Alek, AP, AA})
\[
  R_T \ni \xi  \longleftrightarrow    \mathfrak{f}_{\xi} : = \mathfrak{g}(\kappa^{-1}(\xi))=\sum_{\beta \in R_{F}, \, \kappa(\beta)=\xi} \mathfrak{g}_{\beta}, \quad \text{and} \quad      R^+_T \ni \xi    \longleftrightarrow     \mathfrak{m}_{\xi} := (\mathfrak{f}_{\xi}   + \mathfrak{f}_{-\xi}) \cap \mathfrak{m},
  \]
  \noindent respectively.  Thus, $\fr{m}^{\bb{C}}=\sum_{\xi\in R_{T}}  \mathfrak{f}_{\xi}$ and  $\fr{m}=\sum_{\xi\in R_{T}^{+}} \mathfrak{m}_{\xi}$ are orthogonal decompositions with $\dim_{\mathbb{C}}\mathfrak{f}_{\xi}=   \dim_{\mathbb{R}}\mathfrak{m}_{\xi}  = d_{\xi}$,
  where
  $d_{\xi} := \sharp(\kappa^{-1}( \xi)) $ is  the  cardinality  of   $\kappa^{-1}(\xi)$. As a corollary  we get the  following   description of  $G$-invariant   pseudo-Riemannian metrics on $F$.
   \begin{corol} \label{metric}
      Any   $G$-invariant pseudo-Riemannian metric  $g$ on a flag manifold $F=G/H$  is  defined   by an $\Ad_{H}$-invariant  pseudo-Euclidean  metric  on $\mathfrak{m}$,
       given  by $g_{o}:= \sum_{i=1}^{d}x_{\xi_{i}}B_{\xi_{i}}$,  where   $B_{\xi_{i}}:=-B|_{\mathfrak{m}_{i}}$ is  the  restriction of    $-B$  on  the irreducible    submodule  $\fr{m}_{i}$ associated to the positive $T$-root $\xi_{i}\in R_{T}^{+}$, and $x_{\xi_{i}} \neq 0$ are real numbers, for any $i=1, \ldots, d:=\sharp(R_{T}^{+})$.
 The   signature of  the   metric $g$  is  $(2N_-, 2N_+)$,
 where $N_-:=  \sum_{\xi_{i} \in R^+_T : x_{\xi_{i}}<0} d_{\xi_{i}},\,\, N_+:=  \sum_{\xi_{i} \in R^+_T :  x_{\xi_{i}}>0} d_{\xi_{i}}.$
 In particular, the metric  $g$   is  Riemannian if  all $x_{\xi_{i}}>0$,  and no  metric is  Lorentzian.
  \end{corol}
%\begin{comment}
\begin{proof}  By Schur's  lemma,   the  restriction  of a $G$-invariant  metric $g$  on the irreducible  submodule $\mathfrak{m}_{\pm \xi}$,  is proportional  to  the   restriction of  the  Killing  form.
  This implies  the  first  claim.  Since  $\dim \mathfrak{m}_{\pm \xi} \geq 2$   and  $-B$ is  positively defined  on $\mathfrak{m}$,    the   formula  for  signature  holds. The last claim follows since   the   restriction of   an invariant  metric    to the irreducible  submodule  $\mathfrak{m}_{\xi}$  is positively, or  negatively  defined  and $\dim \mathfrak{m}_{\xi} \geq 2$.
  \end{proof}
%\end{comment}
\begin{comment}
Recall  that the {\it height} or {\it Dynkin mark}  of a simple root  $\al_{j}\in\Pi=\{\al_1, \ldots, \al_\ell\}$ for some $1\leq j\leq \ell$, is the coefficient $\Hgt(\al_{j})$ of the expression of the {\it highest root} $ \tilde{\al}\in R^{+}$ of $\fr{g}^{\bb{C}}$, in terms of simple roots:
  \[
  \tilde{\al}:=\Hgt(\al_{1})\al_{1}+\ldots+\Hgt(\al_{j})\al_{j}+\cdots+\Hgt(\al_{\ell})\al_{\ell}.
 \]
  \begin{lemma} \textnormal{(\cite[Lem.~2]{CS})}
Given a flag manifold $F=G/H$ with $b_{2}(M)=1$, i.e. $\Pi_{B}=\{\al_{i_{o}}\}$ for some simple root $\al_{i_{o}}\in\Pi$, it holds that  $d:=\sharp(R^{+}_{T})=\Hgt(\al_{i_{o}})$.
\end{lemma}
\end{comment}

\subsection{Line bundles, Koszul numbers and the first Chern class}
%\subsubsection{Invariant differential forms on flag manifolds}
  Let $F=G/H$  be  a flag manifold       with  reductive   decomposition
$\mathfrak{g} =\fr{h}+\fr{m}= (\mathfrak{h}' + i\mathfrak{t})+ \mathfrak{m}$.  % and  denote by $\pi : G\to F=G/H$ the quotient map. Then the map
The complex  of $G$-invariant  differential  forms  $\Omega(F)^G$ is naturally identified  with   the  subcomplex
 $$\Lambda(\mathfrak{m}^*)^H   = \{\omega \in \Lambda(\mathfrak{g}^*)^H,\,  X\lrcorner \omega=0 \,\,  \forall  X \in \fr{h} \} \subset \Lambda(\mathfrak{g}^*)^H $$
 of  the  complex $\Lambda(\mathfrak{g}^*)^H$ of  $\Ad_H$-invariant   exterior  forms of   the Lie  algebra $\mathfrak{g}$.    Let us denote   by $\Lambda^k_{cl}(\mathfrak{m}^*)^{H}$   the   space  of   $\Ad_{H}$-invariant closed  $k$-forms   and   by
    $H^k(\mathfrak{m}^*)^H :=  \Lambda^k_{cl}(\mathfrak{m}^*)^{H}/d \Lambda^{k-1}(\mathfrak{m}^*)^H     \simeq  H^k(F,\mathbb{R} )$  the   cohomology  group.  Consider also   the basis  $\omega_{\alpha}$   of   $(\mathfrak{m}^{\mathbb{C}})^*$,  dual  to the basis   $\{E_{\alpha},\, \alpha  \in  R_F\}$  of  $\mathfrak{m}^{\mathbb{C}}$, i.e. $\omega^{\al}(E_{\be})=\delta^{\al}_{\be}$ and   $\omega^{\al}(\fr{h}^{\bb{C}})=0$. We  need  the  following well-known result.

  \begin{prop}\textnormal{(\cite{Bo, Alek, AP, Tak})}\label{mainuse1}
 There is  a natural isomorphism  (transgression)  $\tau: \fr{t}^{*}\to \Lambda^{2}_{cl}(\fr{m}^*)^{H}\cong  H^2(\mathfrak{m}^*)^H \simeq  H^2(F, \mathbb{R}) $ between   the space $\fr{t}^{*}$ and the space $\Lambda^{2}_{cl}(\fr{m}^{*})^{H}$ of $\Ad_{H}$-invariant  closed real 2-forms on $\fr{m}$
 (identified  with the space of    closed  $G$-invariant real 2-forms  on $F$),    given by
    \[
  \fr{a}_{0}^{*}\supset\fr{t}^{*}\ni \xi \mapsto \tau(\xi)\equiv \omega_{\xi} :=\frac{i}{2\pi}d\xi = \frac{i}{2\pi}\sum_{\al\in R_{F}^{+}}  (\xi | \al) \omega^{\al}\wedge \omega^{-\al}\in \Lambda^{2}_{cl}(\fr{m}^{*})^{H}\cong H^2(F, \mathbb{R}),
  \]
where $(\xi | \al):=2(\xi, \al)/(\al, \al)$. In particular, $\tau(\cal{P}_{T})\cong H^{2}(F, \bb{Z})$  and   $b_{2}(F)=\dim\fr{t}=v=\rnk R_{T}$. % where $(\xi | \al):=2(\xi, \al)/(\al, \al)$. % and it is such that   $\tau(\cal{P}_{T})=H^{2}(F, \bb{Z})$.
     \end{prop}

From now on we   assume that $G$  is simply-connected.   Let    $\mathcal{X}(H) = \mathrm{Hom}(H, {\rm T}^{1})$  be  the    group of  (real)  characters   of  the  stability  subgroup $H =Z(H)\cdot H'={\rm T}^{v}\cdot H'$  of  the  flag manifold $F = G/H=G^{\bb{C}}/P$,    and   $\mathcal{X}(P) = \mathrm{Hom}(P, \mathbb{C}^*)$  the   group  of   holomorphic   characters  of  the parabolic  subgroup $P$.
 In  the  case  of    a  full   flag manifold  $F = G/{\rm T}^{\ell} = G^{\mathbb{C}}/B_{+}$  (where  $B_{+}$ is  the Borel  subgroup),  it is   well-known  that
 the character groups   of ${\rm T}^{\ell}$ and $B_{+}$  are  isomorphic   to  the weight  lattice $\mathcal{P}$ (generated by  fundamental  weights). In particular,
a  weight  $\lambda \in \mathcal{P} \subset {\mathfrak{a}}^*_0$   defines a  real   character  $\chi_{\lambda} \in \mathcal{X}({\rm T}^{\ell})$, given  by
 \[
  \chi_{\lambda}(\mathrm{exp}2\pi iX)= \mathrm{exp}(2 \pi i\lambda(X)), \quad \forall X \in \mathfrak{a}_0, %\quad \chi_{\Lambda}(\mathrm{exp} 2 \pi i a ) = \mathrm{exp}(2 \pi i \Lambda(a)) )
   \]
   and  a   holomorphic   character    $\chi^{\mathbb{C}}_{\lambda} \in  \mathrm{Hom}(B_{+}, \mathbb{C}^*)= \mathrm{Hom}(({\rm T}^{\ell})^{\mathbb{C}} \cdot B_{+}', \mathbb{C}^*) $   which is    the  natural  extension of  $\chi_{\lambda}$  such  that  the  kernel of  $\chi^{\mathbb{C}}_{\lambda}$     contains    the  unipotent  radical  $B'_{+} = [B_{+},B_{+}]$.

   The  natural  generalisation   to    any  flag manifold can be described    as follows.

%Now,  the  following  proposition is  a generalization of   the  well-known  fact (cf.  \cite{GOV}),  that  the group of   characters
 %$\mathcal{X}({\rm T}^{\ell})  = \mathrm{Hom }({\rm T}^{\ell}, {\rm T}^{1})$  of the maximal torus  ${\rm T}^{\ell}\subset H\subset G$    is %identified (when $G$ is simply-connected)  with   the  weight lattice $\cal{P} \subset \mathfrak{a}_0^*$,  via  the map

 \begin{prop}  \label{Tcharac}  The   group $\mathcal{X}(H)$  of  real characters    and   the  group  $\mathcal{X}(P)$ of   holomorphic  characters   are  isomorphic  to  the  lattice  $\mathcal{P}_T$  of  $T$-weights.  In particular,    any $\lambda \in  \mathcal{P}_T$    defines   the  character  $\chi_{\lambda} : Z(H)\cdot H'  \to  {\rm T}^1 $ with $H' \subset \mathrm{ker}\chi_{\lambda}$ and has an extension ${\chi^{\mathbb{C}}_{\lambda}} : P  = H^{\mathbb{C}}\cdot N^+  \to  \mathbb{C}^*$  with  $N^+ \subset \mathrm{ker}{\chi^{\mathbb{C}}_{\lambda}} $, where $N^{+}$ denotes the closed connected subgroup of $G^{\bb{C}}$ with Lie algebra  $\fr{n}^{+}:=\sum_{\al\in R_{F}^{+}}\fr{g}_{\al}$.
    \end{prop}

   \begin{proof}  Recall  that  the  lattice  $\mathcal{P}_T$ of  $T$-weights  is  the  sublattice of $\mathcal{P}$    which  annihilates  the Cartan    subalgebra $\mathfrak{a}'  \subset \mathfrak{a}$ of  the  semisimple  part $\mathfrak{h}'$  of $ \mathfrak{h}$.  Since  $\lambda \in \cal{P}_T$ vanishes  on   $\mathfrak{a}'$, it  can be  extended  to   a  real homomorphism   $\lambda : \mathfrak{h}  \to i \mathbb{R} $  with $ \mathfrak{h}' \subset\mathrm{ker} \lambda$  and also  to  a   complex  homomorphism  $\lambda^{\mathbb{C}} : \mathfrak{p}  = \mathfrak{h}^{\mathbb{C}} + \mathfrak{n}^+  \to \mathbb{C}$,  with $\mathfrak{n}^+ \subset \mathrm{ker} \lambda^{\mathbb{C}}$. The  exponents  of  these   homomorphisms   define  the  desired characters  $\chi_{\lambda}  : H \to  {\rm T}^1$ and
$\chi^{\mathbb{C}}_{\lambda}  : P \to \mathbb{C}^*$, respectively.
\end{proof}

Let us  associate   to   a character  $\chi = \chi_{\lambda}$   a   homogeneous principal circle  bundle    $F_{\lambda} = G/H_{\lambda} \to  F = G/H$,    where  $H_{\chi} = \mathrm{ker}(\chi_{\lambda})$. Similarly,  to any holomorphic    character    ${\chi^{\mathbb{C}}_{\lambda}} : P \to \mathbb{C}^*  = {\rm GL}(\mathbb{C}_{\lambda})$  (or equivalently, a $T$-weight $\lambda\in\cal{P}_{T}$) we assign  the 1-dimensional $P$-module $\bb{C}_{\lambda}$ (with underlying space $\bb{C}$) and a   holomorphic  line  bundle  over $G^{\bb{C}}/P$, with fiber $\bb{C}_{\lambda}$, given by
 $\mathcal{L}_{\lambda} = G^{\mathbb{C}} \times_P \mathbb{C}_{\lambda} \to  F = G^{\mathbb{C}}/P$. % associated   with   the   representation ${\chi^{\mathbb{C}}_{\lambda}} : P \to \mathbb{C}^*  = {\rm GL}(\mathbb{C}_{\lambda})$.
We consider the splitting $R=R_{F}^{-}\sqcup R_{H}\sqcup R_{F}^{+}$ associated with the generalized Gauss decomposition
 \[
 \fr{g}^{\bb{C}}=\fr{n}^{-}+\fr{h}^{\bb{C}}+\fr{n}^{+}.
 \]
 The root system of the parabolic subalgebra $\fr{p}=\fr{h}^{\bb{C}}+\fr{n}^{+}$ is given by $R_{\fr{p}}=R_{H}\cup R^{+}=R_{H}\sqcup R_{F}^{+}$.   Recall that
 $\fr{p}$ corresponds to a fundamental grading  of $\fr{g}^{\bb{C}}$, i.e. a direct sum decomposition
 \begin{equation}\label{grad}
 \fr{g}^{\bb{C}}=\fr{g}_{-k}+\cdots+\fr{g}_{0}+\cdots+\fr{g}_{k}, \quad [\fr{g}_{i}, \fr{g}_{j}]\subset\fr{g}_{i+j} \ \forall i, j\in \bb{Z}, \quad \fr{g}_{-k}\neq \{0\},
 \end{equation}
 with the additional requirement that $\fr{g}_{-k}$ is generated by $\fr{g}_{-1}$. The depth $k$ of  grading is given by the degree ${\rm deg}_{\Pi_{B}}(\tilde{\al}):=m_{i_{1}}+\cdots+m_{i_{v}}$ of the highest root $\tilde{\al}=\sum_{i=1}^{\ell} m_{i}\al_{i} \in R^+$
associated    with  the subset
  $\Pi_{B}=\Pi\backslash\Pi_{W}  =\{\be_{i_{1}}, \ldots, \be_{i_{v}}\} $ of black simple roots.
     In particular,  by setting $R_{i}:=\{\al\in R : {\rm deg}(\al)=i\}$ for any $i\in\bb{Z}$,
      $\fr{g}_{i}:= \fr{g}(R_i)$  for  $i \neq 0$   and  $\fr{g}_{0}:= \fr{a}^{\bb{C}} + \fr{g}(R_0)  =  \fr{h}^{\bb{C}}$,
       we  get  a  depth $k$ grading  $\fr{g}= \sum_{i=-k}^{k}\fr{g}_{i}$  associated   with the  parabolic   subalgebra $\fr{p} = \sum_{0}^{k}\fr{g}_{i}$.

For the description of the holomorphic  line  bundle over $G^{\bb{C}}/P$,  associated  to a black fundamental  weight $\Lambda_{\be}$ $(\be\in\Pi_{B})$, we proceed as follows.  Let $\Pi = \{\alpha_0 = \beta, \alpha_1, \cdots, \alpha_{N} \}$
be a system of  simple  roots  such that $1+N=\ell$  and let  $\{\Lambda_{\alpha_i}\}$ be the associated       fundamental  weights. % Recall that $\Lambda_{\be}$ is defined by the equations
% \[
% (\Lambda_{\be}, h_{\al})=\left\{\begin{tabular} {ll}
% 1, & if \  $\al=\be$, \\
% 0, & if  \ $\al\in \Pi-\{\be\}$.
 %\end{tabular}\right.
  %\]
Any  root  has  a  decomposition
\[
\alpha = m_0(\alpha)\beta + \sum_{\alpha_i \in \Pi\setminus \{\beta\}} m_i(\alpha)\alpha_i.
\]
 % $\alpha = m_0(\beta)\beta + \sum_{i>0} m_i(\alpha)\alpha_i $.
For    the maximal  root  we  write $\tilde{\al} = m_{\beta}\beta + \sum_{i>0}m_i \alpha_i$, where the coefficients  $m_i$ are  the so-called  Dynkin marks (cf. \cite{AA}). The   root  $\beta$   defines  a depth  $m_{\beta}$  gradation  of  $\mathfrak{g}^{\mathbb{C}}$,  such  that
$\mathfrak{g}_k = \mathfrak{g}(R_{\beta}^k)$  where $R_{\beta}^k :=\{ \gamma \in R : m_{\be}(\gamma)\equiv m_0(\gamma) =k \}$.
We set
\[
  \sigma = \sigma_G := \frac12\sum_{\gamma \in R^+} \gamma=  \Lambda_{\beta} + \sum_{i>0} \Lambda_{\alpha_i}.
  \]
 We may  decompose $\sigma$ as $ \sigma=  \sigma^0_{\beta} + \sigma^+_{\beta}$,  where  $\sigma^0_{\beta}:=(1/2)\sum_{\gamma\in R_{\beta}^{0}}\gamma$ and  $\sigma^+_{\beta}:=\sum_{k=1}^{m_{\beta}} \sigma^k_{\beta}$, respectively, with $\sigma^k_{\beta}:= (1/2)\sum_{\gamma \in R^k_{\beta}} \gamma$.
%\[
 %\sigma=  \sigma^0_{\beta} + \sigma^+_{\beta}:= \sigma^0_{\beta} +(\sum_{k=1}^{m_{\beta}} \sigma^k_{\beta}),
 %\]
%where  $\sigma^k_{\beta} = \sum_{\gamma \in R^k_{\beta}} \gamma$.\\
We also  normalize  the root vectors $E_{\alpha}$ $(a\in R)$ as
 $B(E_{\alpha}, E_{- \alpha})= \frac{2}{(\alpha, \alpha)}$ and  put  $H_{\alpha} := [E_{\alpha}, E_{- \alpha}] =  \frac{2 B^{-1}\alpha}{(\alpha, \alpha)}$, such that
    $\Lambda_{i}(H_{\alpha_j})  = \delta_{ij}$, for any $1\leq i, j\leq \ell$ (cf.  \cite{GOV}). Recall   that     $\be(H_{\al})=\frac{2(\be, \al)}{(\al, \al)}\in\bb{Z}$    are the so-called Cartan integers and $\al(H_{\al})=2$ for any $\al\in R$.
%We  have  to   explain  the   normalization   of  $E_{\alpha}$, it is  not  explained in our text.

\begin{lemma}\label{koz1}
 It is $2\sigma ^+_{\beta} = k_{\beta} \Lambda_{\beta}$,  where  $k_{\beta}: = 2\sigma^+_{\beta}(H_{\beta}) \geq 2$.
  %is such that  $k\geq 2$. $k_{\beta}\geq 2\sharp R^1_{\beta} + 4 \sharp R^2_{\beta} + 6 \sharp R^3_{\beta} + \cdots$. % with
% $H_{\beta }= [E_{\beta}, E_{-\beta}]$. %In particular the integer $k_{\beta}$ is such that $k_{\beta}\geq 2$.
  \end{lemma}
\begin{proof}
By definition, for  any  $\alpha, \gamma \in \Pi$ it is   $\Lambda_{\alpha}(h_{\gamma}) = \delta_{\alpha, \gamma}$.
Since for  any $\alpha \in \Pi$  it is  $\sigma(H_{\alpha})=1$   and   similarly $\sigma^0_{\beta}(H_{\alpha}) =1$  for   any $\alpha \in \Pi\setminus\{\beta\}$,  we  have  that  $\sigma^+_{\be}(H_{\alpha})= 0$  for  any  $\alpha  \in \Pi\setminus\{\beta\}$.
Since  $H_{\beta}, \, H_{\alpha},\, \alpha \in  \Pi\setminus \{\beta \}$ form  a basis  of  $\mathfrak{a}^{\mathbb{C}}$,  our first assertion
   follows.  Now  we  estimate  $k_{\beta}$
using  the  fact  that   $\sigma_{G}(H_{\al})=1$ for any $\al\in\Pi$ and $\sigma_{\be}^{0}(H_{\be})\leq 0$ since $\gamma(H_{\be})\leq 0$ for any $\gamma\in\Pi\setminus\{\be\}$. We obtain that
 \[
 k_{\be}=2\sigma^{+}_{\beta}(H_{\be})=2(\sigma_{G}-\sigma_{\be}^{0})(H_{\be})=2\Big(\sigma_{G}(H_{\be})-\sigma_{\be}^{0}(H_{\be})\Big)=2-2\sigma_{\be}^{0}(H_{\be}) \geq 2.
\]
 \end{proof}

Set now  $\Pi = \Pi_B \sqcup \Pi_W = \{\beta_1, \cdots, \beta_v\} \sqcup \{ \alpha_1, \cdots, \alpha_u\}$ $(v+u=\ell)$ and  denote by $\xi :=-\hat{\be}:= -\kappa(\beta)$  the negative  $T$-root  induced by    $-\beta$ with $\beta\in\Pi_{B}$ and by $\mathfrak{f}_{\xi}:= \mathfrak{g}(\kappa^{-1}(\xi))$  the corresponding  $H^{\mathbb{C}}$-submodule. % The following results generalises the result of \cite{Akh}
 \begin{prop}\label{bline}
      (i) $\mathfrak{f}_{\xi} \subset  \mathfrak{n}^-   \simeq \mathfrak{g}^{\mathbb{C}} /\mathfrak{p} \simeq T_o G^{\mathbb{C}}/P$  is a $P$-submodule   and  it defines a  $G^{\mathbb{C}}$-invariant   holomorphic   subbundle $E_{\xi}$   of  the    holomorphic tangent  bundle $T (G^{\mathbb{C}}/P )$,  of rank   $d_{\xi}$.  \\
    (ii) The  top  exterior   bundle  $\Lambda^{d_{\xi}}(E_{\xi})$  associated to  the holomorphic vector  bundle $E_{\xi}$ is  isomorphic to the  holomorphic  line  bundle  $\cal{L}_{\Lambda_{\be}}^{\otimes k_{\be}}$,  where $k_{\be}$ is the positive integer defined by Lemma \ref{koz1}.\\
  (iii)   The  Chern  class  of  the    line  bundle $\mathcal{L}_{\Lambda_{\beta}}$ is  the  cohomology   class of  the     form  $\omega_{\Lambda_{\beta}}$.
  %(iv)  Any line bundle over $F=G^{\bb{C}}/P=G/H$ has the form $\cal{L}=\mathcal{L}_{\xi_{1}}^{\otimes n_{1}}\otimes\cdots\otimes \mathcal{L}_{\xi_{v}}^{\otimes n_{v}}$.
   \end{prop}
   \begin{proof}
(i)
  The generalized Gauss decomposition $  \mathfrak{g}^{\mathbb{C}}=  \mathfrak{g}(R^-_F) + (\mathfrak{a}^{\mathbb{C}} + \mathfrak{g}(R_H))+ \mathfrak{g}(R^+_F)  = \mathfrak{g}(R_F^-) + \mathfrak{p}$   allows  to identify   the tangent $\mathfrak{p}$-module
   $T_{o}F = \mathfrak{g}^{\mathbb{C}}/ \mathfrak{p}$ of $F=G^{\mathbb{C}}/P$      with $\mathfrak{n}^{-}$. The Killing form  $B$ induces an isomorphism  of  the  cotangent $\mathfrak{p}$-module  $ T^*_oF = (\fr{g}^{\bb{C}}/\mathfrak{p})^{*}$     with  $\fr{n}^{+}$.
   Since  $\fr{f}_{\xi}:=\mathfrak{f}_{-\hat{\be}} \subset \mathfrak{n}^- \equiv \mathfrak{g}^{\mathbb{C}}/\mathfrak{p}$ is already a
   $\mathfrak{h}^{\mathbb{C}}$-module, it is    sufficient  to  check  that  it is also a $\mathfrak{n}^+ = \mathfrak{g}(R^+_F)$-module.
   This is  equivalent  to  the   condition
   $$   ( \kappa^{-1}(\xi) +R^+_F) \cap R \subset R_H \cup R^+_F . $$
   Any   root  in  $\kappa^{-1}(\xi)$   has   the   form
   $\gamma = -\beta + \sum m_i \alpha_i$ with   $\alpha_i \in \Pi_W$
    and    any  root   in $R^+_F$  is   positive  and  can be  written  as
      $ \delta = p\beta + \sum p_j \beta_j + \sum q_i \alpha_i,$
        where   $\beta_j \in \Pi_B$.   Since  any  root   has   coordinates of  the  same  sign  with  respect  to  a    fundamental  system  $\Pi$,  the  root  $\gamma +\delta $    either  belongs  to  $R_H$ (if $ p=1, p_j =0$),  or  is a positive root from $R^+_F$ (if $p>1$ ). In  both cases it belongs  to   $R_H \cup R^+_F $.

    \noindent    An alternative way to prove (i) reads as follows.  For any $\gamma = \sum k_i \beta_i  + \sum \ell_j \alpha_j \in  R$
 set
$$   R_{\geq \gamma}:= \{\gamma' = \sum k'_i \beta_i + \ell'_j \alpha_j \in R   :  k'_i \geq k_i \}.$$
Then  $\mathfrak{g}(R_{\geq \gamma}) \subset \mathfrak{g}^{\mathbb{C}} $ is a  $\mathfrak{p}$-submodule.  If $\gamma \in R^+_F$, then
$\mathfrak{g}(R_{\geq \gamma})$  is  a  submodule of  $\mathfrak{g}(R^+_F) = T^*(G^{\mathbb{C}}/P)$  and     for  $\gamma \in R^-_F$  it   defines
 a  submodule  of   $\mathfrak{g}(R^-_F)$, hence  an invariant     holomorphic  subbundle  of  the holomorphic tangent bundle $T(G^{\mathbb{C}}/P)$.
Consider a black simple root  $\beta = \beta_0 \in \Pi_B$     with  Dynkin mark $m = m_{\beta}$ and
  let  $\gamma \in R^-_F$  denotes a   negative root    such that   $ m_0(\gamma) = -k$  i.e.   $\gamma = -k \beta + \cdots$, with $1\leq k\leq m_{\be}$.  We set
    $R_{\xi}:= \kappa^{-1}(\xi)\subset R_{F}^{-}$ such that $\fr{f}_{\xi}=\sum_{\al\in R_{\xi}}\bb{C}E_{\al}\subset \mathfrak{n}^-$ and  $R_{k \xi} = \kappa^{-1} (k \xi)$.  Then it is not difficult to see that
    \[
    R_{\geq \gamma} = R_{k \xi} \cup R_{(k-1)\xi} \cup \cdots \cup R_{\xi}\cup R_H \cup R^+_F
    \]
    and the     associated     $\mathfrak{p}$-submodule  of  $\mathfrak{g}(R^-_F)$ is given by   $\mathfrak{g}(R_{k\xi}) + \cdots + \mathfrak{g}(R_{\xi})$.  Note  that  $R_{k \xi} = R_{-\beta}^k$. In particular, for  $k=1$,  the induced $\fr{p}$-module coincides with  the irreducible module $\fr{f}_{\xi}=\mathfrak{g}(R_{\xi})$.

(ii) Let  $E_{\xi}$  be  the associated invariant
     holomorphic subbundle  of  $T(G^{\mathbb{C}}/P)$  of  rank $d_{\xi}$  and write  $\vartheta_{\xi}$ for   the   restriction of  the  isotropy representation of    $P$    to $\fr{f}_{\xi}=\mathfrak{g}(R_{\xi}) = E_{\xi}|_o$. The top    exterior power $L_{\xi} := \Lambda^{d_{\xi}}E_{\xi}$  of  $E_{\xi}$ is the holomorphic line bundle over $F=G^{\bb{C}}/P$, associated to the holomorphic  character $ \chi : P \to  \mathbb{C}^*$ defined by
     $\chi(p):= \det \vartheta_{\xi}(p)$, for any $ p \in P$.
    The kernel of $\chi$ contains $P' := [P,P] =( H')^{\mathbb{C}} \cdot N^+$,  where
     $(H')^{\mathbb{C}}$ is  the  semisimple part of  $H^{\mathbb{C}}$.  Let $E_{\gamma_1}, \cdots E_{\gamma_s}$  be a basis
     of  $\mathfrak{g}(R_{\xi})$ which  consists of  root vectors.  Then    for  any  $\dot a  \in  \mathfrak{a}^{\mathbb{C}}$   we have  that $\vartheta_{\xi}(\dot a) E_{\gamma_j} = \mathrm{ad}_{\dot a}E_{\gamma_j} = \gamma_j(\dot a)E_{\gamma_j} $. Hence,    for   $a = \mathrm{exp}\dot a  \in   Z(H^{\mathbb{C}})$ our character  is   given   by
\[
 \chi (a) = \det \vartheta_{\xi}(a) = \Pi_{j=1}^s  \mathrm{exp}\vartheta(\dot a)E_{\gamma_j}=
      \mathrm{exp}(\sum_j (\gamma_j(\dot a ))),
      \]
      which since $G^{\bb{C}}$ is simply-connected is equivalent to  say that $\dot{\chi}(H_{\al})={\rm tr}( \dot{\theta}_{\xi}(H_{\al}))=\sum_{\gamma\in R_{\xi}}\gamma(H_{\al})$, for any  $H_{\al}\in\fr{a}^{\bb{C}}$.  So   the  character $\chi$ is  the  character  $\chi_{\Lambda}$ associated  with  the  weight    form
     \[
     \Lambda := \sum_{j} \gamma_j  = \sum_{\gamma  \in R^1_{-\beta}} \gamma=  2\sigma_{\beta}^1.
     \]
       However, by Lemma \ref{koz1} we know that $
     2\sigma_{\be}^{+}=k_{\be}\Lambda_{\be}$ and since we need  to restrict  to $R_{\xi}$, we finally get $2\sigma_{\beta}^1=k_{\be}\Lambda_{\be}$. Therefore  it  follows that $\cal{L}_{\Lambda_{\be}}^{\otimes k_{\be}}\cong \Lambda^{d_{\xi}}(E_{\xi})$.

 (iii)
Given the holomorphic line bundle $\cal{L}_{\Lambda_{\be}}$,  its curvature $\Phi$ (with respect to the Chern connection) induces a real (1, 1)-form $\frac{i}{2\pi}\Phi$ which represents the first Chern class, i.e.
\[
c_{1}(\cal{L}_{\Lambda_{\be}})_{\bb{R}}=\frac{i}{2\pi}[\Phi]\in H^{2}(F; \bb{R})
\]
where $c_{\bb{R}}$ means the   the image of some cohomology class $c\in H^{2}(F; \bb{Z})$, under the map $H^{2}(F; \bb{Z})\to H^{2}(F; \bb{R})$ and
$[\Phi]$ denotes the cohomology class of the closed 2-form $\Phi$. By Proposition \ref{mainuse1}, we conclude that $\Phi$ is represented by
the closed $\Ad_{H}$-invariant real 2-form
\[
\omega_{\Lambda_{\be}}= \frac{i}{2 \pi} d \Lambda_{\be}  =\frac{i}{2\pi}\sum_{\al\in R_{F}^{+} }  (\Lambda_{\be} | \al) \omega^{\al}\wedge \omega^{-\al}  \in \Lambda^2(\mathfrak{m}^*)^{H} = \Omega^2_{cl}(F)^{G}.
\]
%(iv) The last assertion  is an immediate consequence of (ii) and Lemma \ref{Tlattice}.

  \end{proof}

\begin{example}
\textnormal{Consider  two special  cases (we always assume that $G$ is simply-connected):\\}
\textnormal{(a)  Let  $F= G/H = G^{\bb{C}}/P$ be  a flag manifold with $b_{2}(F)=1$, i.e. a so-called {\it minimal flag manifold} (and  $P$ is a maximal parabolic subgroup).  In this case  $\Pi_{B}=\{\be:=\al_{i_{o}}\}$ for some  simple  root   $\be:=\alpha_{i_o}$,  $\mathfrak{t} = iZ(\mathfrak{h}) \cong \mathbb{R}$   and  the  $T$- weight
 lattice  $\cal{P}_T$ coincides with  $\mathbb{Z}\Lambda_{\beta}$,  where   $\Lambda_{\beta}$ is  the   fundamental  weight   associated  with   $\be$.  So, the  Picard group of  isomorphism classes of holomorphic line bundles  $\cal{P}{\rm ic}(F)= H^{1}(F, {\mathbb C}^*) = \cal{P}_T$ over $F=G^{\bb{C}}/P=G/H$, is isomorphic to $\bb{Z}$.\\}
   \textnormal{(b) Let  $F=G/{\rm T}^{\ell}=G^{\bb{C}}/B_{+}$ be the  manifold of  full flags. Then,  $b_{2}(F)=\ell$, $\Pi_{B}=\Pi$   and  $\cal{P}_T =\cal{P}$ coincides with    the  weight lattice  in
   $\mathfrak{t }= \mathfrak{a}_0$.   For the   complex   tangent   space   $T_o(G^{\mathbb{C}}/B_{+})$ we get the identification $T_o(G^{\mathbb{C}}/B_{+}) = \sum_{\beta \in R^+} \mathfrak{g}_{-\beta} $.
    A  simple   root  $\beta \in \Pi$  defines    the  complex  line    $\mathfrak{g}_{- \beta} \subset \mathfrak{m}^{\mathbb{C}}$  which is invariant  under  the  isotropy   representation   of  $B_{+}$. The  associated line  bundle defined  by   the  induced   representation of  $B_{+}$ in  $\mathfrak{g}_{\beta}$  is   the homogeneous  line       bundle  $\mathcal{L}_{\Lambda_{\beta}}$ associated  with   the  fundamental    weight $\Lambda_{\beta}$, corresponding  to $\beta\in\Pi$. Thus,    $\cal{P}{\rm ic}(F) = H^{1}(F, {\mathbb C}^*)=\cal{P}=\mathbb{Z}^{\ell}$.}%since it      is generated  by  $\mathcal{L}_{\Lambda_{\beta}}$ with  $\beta  \in \Pi$.}
 \end{example}

%\subsubsection{Chern    class of  an  invariant  complex  structure  and  Koszul numbers}

%It is well-known that the first Chern  class $c_1(\mathcal{L}_{\xi_{j}})$ of the holomorphic line bundle $\mathcal{L}_{\xi_{j}}$, where
%$\xi_{j}=\be_{j}|_{\fr{t}}$ $(j=1, \ldots , v)$,  is  the  cohomology  class  of the associated   curvature two-form
%\[
% \omega_{\xi_{j}}  =  \frac{i}{2 \pi} d \xi_{j}  =
%\frac{i}{2\pi}\sum_{\al\in R_{F}^{+}}  (\xi_{j} | \al) \omega^{\al}\wedge \omega^{-\al}  \in \Lambda^2(\mathfrak{m}^*)^{H} = \Omega^2_{cl}(F) .
%\]
% In particular, $\cal{P}{\rm ic}(F):=H^{1}(G^{\mathbb C}/P, \underline{\mathbb C}^*)\ni \cal{L}_{\lambda}  \overset{c_{1}}{\longrightarrow}  c_{1}(\cal{L}_{\lambda})\in H^2(F, \mathbb Z)$   is an isomorphism (cf. \cite{Bo, Tak}).

 Consider now the linear forms $\sigma_G:= \frac{1}{2}\sum_{\alpha \in R^+}\al$ and $\sigma_H:=\frac{1}{2} \sum_{\alpha \in R^+_H} \alpha$.  %Since $R\subset \cal{P}$, $2\sigma_{G}$ is evidently a weight.
Recall that  $\sigma_{G}=\sum_{i=1}^{\ell}\Lambda_{i}\in\cal{P}^{+}$, where  $\cal{P}^{+}\subset\cal{P}$ is the subset of strictly positive dominant weights.
We also set  $\cal{P}^{+}_{T}:=\cal{P}^{+}\cap \cal{P}_{T}$ and  define the  {\it  Koszul  form}  of  the flag manifold $(F=G^{\bb{C}}/P=G/H, J)$,  by
              \[
              \sigma^J:=2(\sigma_G - \sigma_H)=  \sum_{\alpha \in R^+_F}\alpha.
              \]
                             \begin{lemma}  \textnormal{(\cite{Alek, AP})}  The  Koszul    form   is   a linear  combination of  the  fundamental  weights  $\Lambda_1, \cdots, \Lambda_v$ associated to the black roots, with positive integers   coefficients, given as follows:
\[
\sigma^J = \sum_{j=1}^{v} k_j \Lambda_j = \sum_{j=1}^{v} (2 + b_j) \Lambda_j\in\cal{P}_{T}^{+}, \ \  \text{where} \  \ k_j = \frac{2(\sigma^J,\be_j)}{(\be_j, \be_j)},\,\,   b_j = -\frac{2(2\sigma_H,\be_j)}{(\be_j,\be_j)} \geq 0.
\]
        \end{lemma}
     The  integers  $k_j\in\bb{Z}_{+}$     are   called  {\it Koszul numbers} associated to the complex structure $J$ on $F=G^{\bb{C}}/P=G/H$ and they form the  {\it Koszul vector}  $\vec{k}:=(k_{1}, \ldots, k_{v})\in\bb{Z}_{+}^{v}$.

  \begin{prop}\label{chernclass}\textnormal{(\cite{Bo, Alek, AP})}
      The first  Chern   class $c_1(J)\in H^{2}(F; \bb{Z})$ of    the invariant  complex  structure  $J$ in $F$, associated  with the decomposition  $\Pi= \Pi_W \sqcup \Pi_B$,
  is  represented  by  the closed invariant  2-form  $\gamma_J :=  \omega_{\sigma^J}$,  i.e. the  Chern form of the complex manifold $(F, J)$.
   \end{prop}

  Propositions  \ref{at} and \ref{meta},  in combination with the  above results, yield  that:
  \begin{prop} \label{chernclass2} A   flag manifold  $F = G/H$  admits  a $G$-invariant   spin  (or metaplectic)  structure,  if  and only is  the first Chern  class  $c_1(F, J)$ of   an invariant   complex  structure  $J$  on  $F$ is   even, that is   all Koszul  numbers  are  even.
 In  this case,   such  spin  (or metablectic) structure is   unique.
  \end{prop}
% Notice that in the pseudo-Riemannian case, a similar result  for $\Spin_{p, q}$-structures fails due to the dependence of such a structure on a pseudo-Riemanian structure. For full flag manifolds $G/{\rm T}^{\ell}$,  an example describing this interesting  case was given in  \cite[p.~75]{Alag}.

   \begin{example}\label{FULL}

   \textnormal{Consider the manifold of full flags $F=G/{\rm T}^{\ell}=G^{\bb{C}}/B_{+}$.  Since  the Weyl  group   acts transitively on  Weyl  chambers,  there is unique (up  to  conjugation) invariant complex  structure $J$. The canonical line bundle $K_{F}:=\Lambda^{n}TF$ corresponds to the dominant weight $\sum_{\al\in R^{+}}\al=2\sigma_{G}=2(\Lambda_{1}+\cdots + \Lambda_{\ell})$, hence all the Koszul numbers equal to $2$ and $F$ admits a unique   spin structure.}
    \end{example}

\begin{corol} \label{intChern} The   divisibility by two  of  the  Koszul numbers  of  an invariant   complex  structure $J$ on a (pseudo-Riemannian)    flag manifold $F=G/H=G^{\bb{C}}/P$,    does not  depend on
 the  complex  structure.
  \end{corol}
  \begin{proof}
  If $c_{1}(F, J)$ is divisible by two in $H^{2}(F; \bb{Z})$, i.e.  the Koszul numbers are even with respect to  an invariant complex structure $J$, then  $F$ admits a $G$-invariant spin structure, which is unique since $F=G/H=G^{\bb{C}}/P$ is simply-connected.  Using now    Corollary \ref{basics1},  for any other complex structure
   $J'$, we conclude that the associated first Chern class $c_{1}(F, J')$ must be even, as well.
    \end{proof}

 \begin{corol}
 On a spin or metaplectic  flag manifold $F=G^{\bb{C}}/P=G/H$ with a fixed invariant complex structure $J$, there is a unique isomorphism class of holomorphic line bundles $\cal{L}$ such that $\cal{L}^{\otimes 2}=K_{F}$.   \end{corol}
 %\begin{proof}
 %If $F=G^{\bb{C}}/P=G/H$ is $G$-spin or $G$-metaplectic, then the Koszul numbers are even. Consequently, $\sqrt{K_F}=\cal{L}_{\lambda}$ for $\lambda=\sigma^{J}\in\cal{P}^{+}_{T}\subset\cal{P}_{T}$. The uniqueness follows since $F$ is simply-connected, see also Proposition \ref{at}.
 %\end{proof}

\subsection{Invariant spin and metaplectic structures on  classical flag manifolds}
Proposition  \ref{chernclass2} reduces   the  classification of    $G$-invariant   spin or metaplectic   structures   on   a  given  flag manifold $F=G/H$,    to  the  calculation  of    Koszul numbers   of  an invariant   complex  structure $J$ on $F$. In particular, due to Corollary \ref{intChern} it is sufficient to fix the complex structure  $J$ induced by the natural  invariant ordering $R_{F}^{+}=R^{+}\backslash R^{+}_{H}$ (see \cite{AP, Ale1} for details on the notion of an invariant ordering).

 Flag manifolds   of the groups  $\A_n = \SU_{n+1}, \B_n= \SO_{2 n +1}, \Cc_n =  \Sp_n, \D_n = \SO_{2 n}$   fall into four classes:
\begin{eqnarray*}
\A(\vec n)&=&  \SU_{n+1}/\U_1^{n_0} \times \Ss(\U_{n_1} \times \cdots \times \U_{n_s}), \quad \vec n = (n_0, n_1, \cdots, n_s),\quad  \sum n_j  = n+1, \, n_0 \geq 0,  n_j >1;\\
\B(\vec n)&=&  \SO_{2n+1}/\U_1^{n_0}\times \U_{n_1}\times \cdots\times  \U_{n_s} \times \SO_{2r+1},\quad \vec n =\sum n_j +  r,\,n_0 \geq 0,  n_j >1,  r\geq 0;\\
 \Cc(\vec n)&=&  \Sp_{n}/\U_1^{n_0} \times \U_{n_1}\times \cdots\times  \U_{n_s} \times \Sp_{r},\quad  \vec n =\sum n_j + r,\, n_0\geq 0,  n_j >1, r\geq 0;\\
\D(\vec n)&=&  \SO_{2n}/\U_1^{n_0} \times \U_{n_1}\times  \cdots \times \U_{n_s} \times \SO_{2r},\quad \vec n = \sum n_j +  r,\, n_0 \geq 0, n_0 \geq 0,  n_j >1, r\neq 1,
\end{eqnarray*}
with $\vec{n}=(n_{0}, n_{1}, \cdots, n_{s}, r)$ for the groups $\B_{n}, \Cc_{n}$ and $\D_{n}$. %  Notice that  $\Ss(\U_{n_1}\times\cdots\times\U_{n_{s}})$ is the   Lie group of all diagonal matrices $\diag(A_{1}, \ldots, A_{n})$ with $A_{i}\in\U_{n_{i}}$ and ${\rm det} A_{1}\cdots{\rm det} A_{n}=1$.  Its rank is given by $n_1+\cdots +n_{s}-1$.
It will be useful to recall that the fibration $\SU_{p}\hookrightarrow \U_{p}\overset{\det}{\to}\U_{1}=\Ss^{1}={\rm T}^{1}$ implies that $\U_{p}=\U_{1}\times_{\bb{Z}_{p}}\SU_{p}$, where the cyclic group $\bb{Z}_{p}$ acts on each factor by left translations. Hence, we shall often  write  $\U_{p}=\U_{1}\cdot\SU_{p}=\Ss^{1}\cdot\SU_{p}$ and  identify   $\U_{1}/\bb{Z}_{p}=\Ss^{1}/\bb{Z}_{p}$ with its finite covering $\Ss^{1}$.  For example,  $F=\SO_{2n+1}/\U_{p}\times \SO_{2(n-p)+1}=\SO_{2n+1}/\U_{1}\cdot \SU_{p}\times \SO_{2(n-p)+1}=\SO_{2n+1}/\U_{1}\times_{\bb{Z}_{p}}\SU_{p}\times \SO_{2(n-p)+1}$.
 %Recall that  the fibration $\SU_{p}\hookrightarrow \U_{p}\overset{\det}{\to}\U_{1}=\Ss^{1}={\rm T}^{1}$ implies that $\U_{p}=\U_{1}\times_{\bb{Z}_{p}}\SU_{p}$, where the cyclic group $\bb{Z}_{p}$ acts on each factor by left translations. Hence, we shall often  write  $\U_{p}=\U_{1}\cdot\SU_{p}=\Ss^{1}\cdot\SU_{p}$ and  identify   $\U_{1}/\bb{Z}_{p}=\Ss^{1}/\bb{Z}_{p}$ with its finite covering $\Ss^{1}$.  For example,  $F=\SO_{2n+1}/\U_{p}\times \SO_{2(n-p)+1}=\SO_{2n+1}/\U_{1}\cdot \SU_{p}\times \SO_{2(n-p)+1}=\SO_{2n+1}/\U_{1}\times_{\bb{Z}_{p}}\SU_{p}\times \SO_{2(n-p)+1}$. %since it is diffeomorphic to $\SU_{n_{1}}\times\cdots\times\SU_{n_{s}}\times_{\bb{Z}_{n_{1}\cdots n_{s}}}\U_{1}^{s-1}$, where $\U_{1}^{s-1}=\U_{1}\times\cdots\times\U_{1}$  ($s-1$ factors),

  {\bf Notation for root systems.} For  the standard basis of $\bb{R}^{n_{0}}$ we write $\{\theta_{1}, \ldots, \theta_{n_{0}}\}$,  while standard bases of $\bb{R}^{n_{a}}$ $(a=1, \ldots, s)$ are given by $\{\epsilon^{a}_{1}, \ldots, \epsilon^{a}_{n_{a}}\}$, such that   \[
  \epsilon=\{\theta_{1}, \ldots, \theta_{n_{0}}, \epsilon^{1}_{1}, \ldots, \epsilon^{1}_{n_{1}}, \epsilon^{2}_{1}, \dots, \epsilon^{2}_{n_{2}}, \ldots, \epsilon^{s}_{1}, \ldots, \epsilon^{s}_{n_{s}}\}
  \]
   is the standard basis of $\bb{R}^{n+1} =  \bb{R}^{n_0} \times \bb{R}^{n_1}\times \cdots \times \bb{R}^{n_s}$.
    For $\B_{n}, \Cc_{n}$ and $\D_{n}$ we extend this notation to $\{\theta_{j}, \epsilon^{a}_{i}, \pi_{k}\}$ for the standard basis of $\bb{R}^{n} = \bb{R}^{n_0} \times \bb{R}^{n_1}\times \cdots \times \bb{R}^{n_s}\times \bb{R}^r $ with $1\leq k\leq r$, $1\leq j\leq n_{0}$ and $a=1, \ldots, s$, with $1\leq i\leq n_{a}$.
    For  the expressions of the root systems $R, R_{H}$ associated to the flag manifold $G(\vec{n})$ we shall use these of \cite{AA, Chry2} (for a description of root systems see also \cite{Freud}).
% Let us describe now the    standard  invariant  complex  structure $J_0$ on  each  such manifold,  in terms of  painted Dynkin diagrams.
 \begin{lemma}  \label{wellknown1}
  The  standard complex  structure   on a  flag manifold $G(\vec n)$ for one of the groups $G =\A_n, \B_n, \Cc_n, \D_n$, is
 the  complex  structure   associated  with the   following  standard painted Dynkin  diagram:
 %$$    \bullet - \bullet - \cdots -\bullet - \circ - \cdots - \circ -\bullet - \circ - \cdot - \circ - \bullet - \cdots $$
 \[
a) \A_{n} : \begin{picture}(400, 0)(0, 5)
  % \put(-20, 0.5){\circle*{5}}
    %\put(-18, 0.5){\line(1,0){17}}
\put(0, 0.5){\circle*{5}}
\put(0, 5){\line(0,3){10}}
\put(0, 15){\line(1,0){25}}
\put(33, 15){\makebox(0,0){{\tiny{$n_{0}$}}}}
\put(41, 15){\line(1,0){26}}
\put(67, 15){\line(0,-3){10}}

%\put(0,7){\makebox(0,0){$\al_{1}$}}
 \put(2, 0.5){\line(1,0){14}}
\put(18, 0.5){\circle*{5}}
%\put(18,7){\makebox(0,0){$\al_2$}}
 \put(20, 0.5){\line(1,0){14}}
\put(32, 0){ $\ldots$}
\put(51, 0.5){\line(1,0){14}}
 \put(67, 0.5){\circle*{5}}
%  \put(67,10){\makebox(0,0){$\al_{n_{0}}$}}
\put(68, 0.5){\line(1,0){13}}
\put(83, 0.5){\circle{5}}
\put(83, 5){\line(0,3){10}}
\put(83, 15){\line(1,0){11}}
\put(107, 15){\makebox(0,0){{\tiny{$n_{1}-1$}}}}
\put(119, 15){\line(1,0){12}}
\put(131, 15){\line(0,-3){10}}
%\put(89,6){\makebox(0,0){$\al_{k+1}$}}
 \put(85, 0.5){\line(1,0){14}}
\put(97, 0){ $\ldots$}
\put(115, 0.5){\line(1,0){14}}
\put(131, 0.5){\circle{5}}
\put(133, 0.5){\line(1,0){14}}
\put(149, 0.5){\circle*{5}}
 \put(149,10){\makebox(0,0){$\be_{n_{1}}$}}
  \put(150.9, 0.5){\line(1,0){13}}
  \put(166, 0.5){\circle{5}}
  \put(166, 5){\line(0,3){10}}
\put(166, 15){\line(1,0){11}}
\put(190, 15){\makebox(0,0){{\tiny{$n_{2}-1$}}}}
\put(203, 15){\line(1,0){13}}
\put(216, 15){\line(0,-3){10}}
    \put(168, 0.5){\line(1,0){14}}
     \put(180, 0){ $\ldots$}
  \put(200, 0.5){\line(1,0){13.6}}
   \put(216, 0.5){\circle{5}}
  %  \put(216,6){\makebox(0,0){$\al_{l+m+n-1}$}}
    \put(218, 0.5){\line(1,0){13.6}}
\put(232, 0.5){\circle*{5}}
 \put(232,10){\makebox(0,0){$\be_{n_{2}}$}}
  \put(233.9, 0.5){\line(1,0){13}}
  \put(249, 0.5){\circle{5}}
    \put(251, 0.5){\line(1,0){13.6}}
  \put(262.5, 0){ $\ldots$}
  \put(282.5, 0.5){\line(1,0){13.6}}
   \put(299, 0.5){\circle{5}}
  %  \put(216,6){\makebox(0,0){$\al_{l+m+n-1}$}}
    \put(301, 0.5){\line(1,0){13.6}}
\put(315, 0.5){\circle*{5}}
 \put(315,10){\makebox(0,0){$\be_{n_{s-1}}$}}
   \put(316, 0.5){\line(1,0){13}}
   \put(331, 0.5){\circle{5}}
   \put(331, 5){\line(0,3){10}}
\put(331, 15){\line(1,0){11}}
\put(355, 15){\makebox(0,0){{\tiny{$n_{s}-1$}}}}
\put(368.7, 15){\line(1,0){13.5}}
\put(382, 15){\line(0,-3){10}}

      \put(333, 0.5){\line(1,0){13.6}}
   \put(346, 0){ $\ldots$}
\put(366.5, 0.5){\line(1,0){13.6}}
   \put(382, 0.5){\circle{5}}
          \end{picture}
 \]
with $\sharp(\Pi_{B}):=v=n_{0}+s-1$, such  that  the first $n_0$  nodes   are   black  and  all other  black nodes   are  isolated.  In terms of the   standard  basis $\Pi=\{\al_1, \ldots, \al_{n}\}$ of $\SU_{n+1}$,  we indicate  the black   simple   roots as  follows
\[
\be_{1}:=\al_1=\theta_{1}-\theta_{2},  \ldots,     \beta_{n_{0}-1}:=\al_{n_{0}-1}=\theta_{n_{0}-1}-\theta_{n_{0}}, \ \beta_{n_{0}}:=\al_{n_{0}}=\theta_{n_{0}}-\epsilon^{1}_{1},
\]
and $\be_{n_{1}}=\al_{n_{1}}, \ \be_{n_2}=\al_{n_{1}+n_{2}-1}, \ \be_{n_{3}}=\al_{n_{1}+n_{2}+n_{3}-2},  \ldots,  \be_{n_{s}-1}=\al_{n_{1}+\ldots+n_{s-1}-(s-2)}$, such that $G(\vec n)=\A(\vec n)$.
Similarly for the other groups, with  $\sharp(\Pi_{B}):=v=n_{0}+s$, such  that  the  first $n_0$  nodes   are   black  and  all other  black nodes   are  isolated:

\smallskip
  \[
 b) \ \B_{n} :  \begin{picture}(350,0)(-10, 0)
 %  \put(-20, 0.5){\circle*{5}}
    %\put(-18, 0.5){\line(1,0){17}}
\put(0, 0.5){\circle*{5}}
\put(0, 5){\line(0,3){10}}
\put(0, 15){\line(1,0){25}}
\put(33, 15){\makebox(0,0){{\tiny{$n_{0}$}}}}
\put(41, 15){\line(1,0){26}}
\put(67, 15){\line(0,-3){10}}

%\put(0,7){\makebox(0,0){$\al_{1}$}}
 \put(2, 0.5){\line(1,0){14}}
\put(18, 0.5){\circle*{5}}
%\put(18,7){\makebox(0,0){$\al_2$}}
 \put(20, 0.5){\line(1,0){14}}
\put(32, 0){ $\ldots$}
\put(51, 0.5){\line(1,0){14}}
 \put(67, 0.5){\circle*{5}}
% \put(67,7){\makebox(0,0){$\al_{n_{0}}$}}
\put(68, 0.5){\line(1,0){13}}
\put(83, 0.5){\circle{5}}
\put(83, 5){\line(0,3){10}}
\put(83, 15){\line(1,0){11}}
\put(107, 15){\makebox(0,0){{\tiny{$n_{1}-1$}}}}
\put(119, 15){\line(1,0){12}}
\put(131, 15){\line(0,-3){10}}

%\put(89,6){\makebox(0,0){$\al_{k+1}$}}
 \put(85, 0.5){\line(1,0){14}}
\put(97, 0){ $\ldots$}
\put(115, 0.5){\line(1,0){14}}
\put(131, 0.5){\circle{5}}
\put(133, 0.5){\line(1,0){14}}
\put(149, 0.5){\circle*{5}}
 \put(149,10){\makebox(0,0){$\be_{n_{1}}$}}
  \put(150.9, 0.5){\line(1,0){13}}
  \put(166, 0.5){\circle{5}}
    \put(168, 0.5){\line(1,0){14}}
     \put(180, 0){ $\ldots$}
  \put(200, 0.5){\line(1,0){13.6}}
   \put(216, 0.5){\circle{5}}
  %  \put(216,6){\makebox(0,0){$\al_{l+m+n-1}$}}
    \put(218, 0.5){\line(1,0){13.6}}
\put(232, 0.5){\circle*{5}}
 \put(232,10){\makebox(0,0){$\be_{n_s}$}}
  \put(233.9, 0.5){\line(1,0){13}}
  \put(249, 0.5){\circle{5}}
  \put(249, 5){\line(0,3){10}}
\put(249, 15){\line(1,0){26}}
\put(283, 15){\makebox(0,0){{\tiny{$r$}}}}
\put(290, 15){\line(1,0){24}}
\put(314, 15){\line(0,-3){10}}

  \put(251, 0.5){\line(1,0){13.6}}
  \put(262.5, 0){ $\ldots$}
  \put(280, 0.5){\line(1,0){13}}
    \put(295, 0.5){\circle{5}}
  \put(297, 1.1){\line(1,0){14}}
\put(297, -0.6){\line(1,0){14}}
\put(301, -1.5){\scriptsize $>$}
%\put(103.5, -3){$\Longleftarrow$}
\put(314, 0.5){\circle{5}}

         \end{picture}
 \]
 \smallskip
  \[
c) \  \Cc_{n} :  \begin{picture}(350,0)(-10, 0)
 %  \put(-20, 0.5){\circle*{5}}
    %\put(-18, 0.5){\line(1,0){17}}
\put(0, 0.5){\circle*{5}}
\put(0, 5){\line(0,3){10}}
\put(0, 15){\line(1,0){25}}
\put(33, 15){\makebox(0,0){{\tiny{$n_{0}$}}}}
\put(41, 15){\line(1,0){26}}
\put(67, 15){\line(0,-3){10}}

%\put(0,7){\makebox(0,0){$\al_{1}$}}
 \put(2, 0.5){\line(1,0){14}}
\put(18, 0.5){\circle*{5}}
%\put(18,7){\makebox(0,0){$\al_2$}}
 \put(20, 0.5){\line(1,0){14}}
\put(32, 0){ $\ldots$}
\put(51, 0.5){\line(1,0){14}}
 \put(67, 0.5){\circle*{5}}
% \put(67,7){\makebox(0,0){$\al_{n_{0}}$}}
\put(68, 0.5){\line(1,0){13}}
\put(83, 0.5){\circle{5}}
\put(83, 5){\line(0,3){10}}
\put(83, 15){\line(1,0){11}}
\put(107, 15){\makebox(0,0){{\tiny{$n_{1}-1$}}}}
\put(119, 15){\line(1,0){12}}
\put(131, 15){\line(0,-3){10}}

%\put(89,6){\makebox(0,0){$\al_{k+1}$}}
 \put(85, 0.5){\line(1,0){14}}
\put(97, 0){ $\ldots$}
\put(115, 0.5){\line(1,0){14}}
\put(131, 0.5){\circle{5}}
\put(133, 0.5){\line(1,0){14}}
\put(149, 0.5){\circle*{5}}
 \put(149,10){\makebox(0,0){$\be_{n_1}$}}
  \put(150.9, 0.5){\line(1,0){13}}
  \put(166, 0.5){\circle{5}}
    \put(168, 0.5){\line(1,0){14}}
     \put(180, 0){ $\ldots$}
  \put(200, 0.5){\line(1,0){13.6}}
   \put(216, 0.5){\circle{5}}
  %  \put(216,6){\makebox(0,0){$\al_{l+m+n-1}$}}
    \put(218, 0.5){\line(1,0){13.6}}
\put(232, 0.5){\circle*{5}}
 \put(232,10){\makebox(0,0){$\be_{n_{s}}$}}
  \put(233.9, 0.5){\line(1,0){13}}
  \put(249, 0.5){\circle{5}}
   \put(249, 5){\line(0,3){10}}
\put(249, 15){\line(1,0){26}}
\put(283, 15){\makebox(0,0){{\tiny{$r$}}}}
\put(290, 15){\line(1,0){24}}
\put(314, 15){\line(0,-3){10}}

  \put(251, 0.5){\line(1,0){13.6}}
  \put(262.5, 0){ $\ldots$}
  \put(280, 0.5){\line(1,0){13}}
    \put(295, 0.5){\circle{5}}
  \put(297, 1.1){\line(1,0){14}}
\put(297, -0.6){\line(1,0){14}}
\put(300, -1.5){\scriptsize $<$}
%\put(103.5, -3){$\Longleftarrow$}
\put(314, 0.5){\circle{5}}

         \end{picture}
 \]
 \smallskip
    \[
d) \  \D_{n} :  \begin{picture}(350,0)(-10, 0)
  % \put(-20, 0.5){\circle*{5}}
    %\put(-18, 0.5){\line(1,0){17}}
\put(0, 0.5){\circle*{5}}
\put(0, 5){\line(0,3){10}}
\put(0, 15){\line(1,0){25}}
\put(33, 15){\makebox(0,0){{\tiny{$n_{0}$}}}}
\put(41, 15){\line(1,0){26}}
\put(67, 15){\line(0,-3){10}}

%\put(0,7){\makebox(0,0){$\al_{1}$}}
 \put(2, 0.5){\line(1,0){14}}
\put(18, 0.5){\circle*{5}}
%\put(18,7){\makebox(0,0){$\al_2$}}
 \put(20, 0.5){\line(1,0){14}}
\put(32, 0){ $\ldots$}
\put(51, 0.5){\line(1,0){14}}
 \put(67, 0.5){\circle*{5}}
  %\put(67,7){\makebox(0,0){$\al_{n_{0}}$}}
\put(68, 0.5){\line(1,0){13}}
\put(83, 0.5){\circle{5}}
\put(83, 5){\line(0,3){10}}
\put(83, 15){\line(1,0){11}}
\put(107, 15){\makebox(0,0){{\tiny{$n_{1}-1$}}}}
\put(119, 15){\line(1,0){12}}
\put(131, 15){\line(0,-3){10}}

%\put(89,6){\makebox(0,0){$\al_{k+1}$}}
 \put(85, 0.5){\line(1,0){14}}
\put(97, 0){ $\ldots$}
\put(115, 0.5){\line(1,0){14}}
\put(131, 0.5){\circle{5}}
\put(133, 0.5){\line(1,0){14}}
\put(149, 0.5){\circle*{5}}
 \put(149,10){\makebox(0,0){$\be_{n_{1}}$}}
  \put(150.9, 0.5){\line(1,0){13}}
  \put(166, 0.5){\circle{5}}
    \put(168, 0.5){\line(1,0){14}}
     \put(180, 0){ $\ldots$}
  \put(200, 0.5){\line(1,0){13.6}}
   \put(216, 0.5){\circle{5}}
  %  \put(216,6){\makebox(0,0){$\al_{l+m+n-1}$}}
    \put(218, 0.5){\line(1,0){13.6}}
\put(232, 0.5){\circle*{5}}
 \put(232,10){\makebox(0,0){$\be_{n_{s}}$}}
  \put(233.9, 0.5){\line(1,0){13}}
  \put(249, 0.5){\circle{5}}
   \put(249, 5){\line(0,3){10}}
\put(249, 15){\line(1,0){26}}
\put(283, 15){\makebox(0,0){{\tiny{$r$}}}}
\put(290, 15){\line(1,0){24}}
\put(314, 15){\line(0,-3){20}}

  \put(251, 0.5){\line(1,0){13.6}}
  \put(262.5, 0){ $\ldots$}
  \put(280, 0.5){\line(1,0){13}}
    \put(295, 0.5){\circle{5}}
\put(297.3, 1){\line(2,1){10}}
\put(297.3, -1){\line(2,-1){10}}
%\put(115, 3){\line(-1,1){3}}
%\put(115, -3){\line(1,1){3}}
\put(309.5, 6){\circle{5}}
\put(309.5, -6){\circle{5}}
%\put(123.5, 14){\makebox(0,0){$\al_{\ell-1}$}}
%\put(118, -16){$\al_\ell$}
         \end{picture}
 \]

 \smallskip
\noindent In these cases, $\be_{n_{s}}:=\al_{n_{1}+\ldots+n_{s}-(s-1)}$. Such  a complex  structure   exists   for  any    flag manifold of a  classical Lie  group.% \dashline{4}[0.7](0,18)(60,18)
 \end{lemma}

     Koszul numbers    of  classical flag manifolds    had  been described  in  \cite{AP} in terms of      painted   Dynkin   diagrams.  A revised version of the algorithm  in \cite{AP} is given as follows:
             \begin{prop}  \label{APP} \textnormal{({\bf Revised version of} \cite{AP, Ale1})} Let $F = G/H$ be  a flag  manifold  of  a  classical Lie  group $\A_{n} , \B_n, \Cc_n, \D_n$   with an invariant  complex  structure $ J$.   Then, the  Koszul number $k_j$   associated  with   the black  simple  root $\be_{j}\in \Pi_B$   equals to $2 + b_j$,  where   $b_j$ is  the number of  white  roots   connected  with   the  black  root   by a  string of   white  roots  with  the   following  exceptions:

 (a) For $\B_{n}$,   each  long white  root of  the  last  white  string  which  corresponds  to  the simple factor  $\SO_{2r +1}$   is  counted  with multiplicity two, and  the last  short    white   root  is  counted  with multiplicity one.  If the  last simple root is painted black, i.e.   $\be_{s_{n}}=\al_{n}$, then the coefficient $b_{n_{s}}$ is twice the number of white roots which are connected with this root.

 (b) For  $\Cc_{n}$,  each     root of  the  last  white  string  which corresponds  to the factor  $\Sp_r$ is  counted  with multiplicity  two.

 (c) For $\D_{n}$, the last white  chain which defines the root system of $\D_{r}=\SO_{2r}$ is considered as a chain of length $2(r-1)$. If $r=0$ and  one of  the last right   roots is  white  and  other is  black,  then the  Koszul number $k_{n_{s}}$ associated to this black end root $\beta_{n_{s}}$,   is   $2(n_s-1)$. %(this is correct, but different form what is written in \cite{AP}!!!)} %  where  $n_s$ is  the number of  white  roots   connected  with  i
  \end{prop}

A direct computation yields now the following   description of the Koszul vector.% associated to  the  standard   complex  structure $J_0$ on a classical flag manifold.
  \begin{corol} \label{CLASFLAGS}   The  Koszul vector $\vec{k}:= (k_1, \cdots, k_v)\in\bb{Z}_{+}^{v}$   associated to  the  standard   complex  structure $J_0$ on  a    flag manifold   $G(\vec n)$  of  classical  type,
 is given  by
\[
\begin{array}{ccl}
 \A(\vec{n}):   & \,\, \vec{k} = & (2,\cdots, 2, 1+ n_1, n_1+n_2, \cdots,  n_{s-1}+ n_s ), \\
  \B(\vec{n}):   & \,\, \vec{k} = &(2,\cdots, 2, 1+ n_1, n_1+n_2, \cdots,  n_{s-1}+ n_s, n_s+ 2r ), \\
  \Cc(\vec{n}):   & \,\, \vec{k} = &(2,\cdots, 2, 1+ n_1, n_1+n_2, \cdots,  n_{s-1}+ n_s, n_s+2r +1), \\
  \D(\vec{n}):   & \,\, \vec{k} = &(2,\cdots, 2, 1+ n_1, n_1+n_2, \cdots,  n_{s-1}+ n_s, n_s+2r-1).
 \end{array}
\]
If $r=0$,  then  the  last Koszul number (over  the  end   black root)  is $2n_s $  for $\B(\vec{n})$, $n_s+1$  for  $\Cc(\vec{n})$ and $2(n_{s}-1)$ for $\D(\vec{n})$.
\end{corol}

Due to  Proposition \ref{chernclass2} and Corollary \ref{CLASFLAGS} we get  the   following  classification of  spin  flag manifolds  (the same conclusions hold also  for $G$-metaplectic structures):
  \begin{theorem}\label{apl1}
(a) The   flag manifold $\A(\vec{n})$ with $n_{0}>0$ is $G$-spin  if and only if all the numbers $n_{1}, \ldots, n_{s}$ are odd. If  $n_{0}=0$, then $\A(\vec{n})$ is $G$-spin, if and only if   the numbers $n_{1}, \ldots, n_{s}$ have the same parity, i.e. they are all odd or all even.

(b)   The flag manifold $\B(\vec{n})$ with $n_{0}>0$ and $r>0$ does  not admit  a ($G$-invariant) spin structure. If $n_{0}>0$ and $r=0$, then  $\B(\vec{n})$  is $G$-spin, if and only if all the numbers $n_{1}, \ldots, n_{s}$ are odd. If $n_{0}=0$ and $r>0$, then $\B(\vec{n})$  is $G$-spin if and only if all the numbers $n_{1}, \ldots, n_{s}$ are even. Finally, for $n_{0}=0=r$, the flag manifold  $\B(\vec{n})$  is $G$-spin if and only if all the numbers $n_{1}, \ldots, n_{s}$ have the same parity.

(c) The flag manifold $\Cc(\vec{n})$ with $n_{0}>0$ is $G$-spin, if and only if all the numbers $n_{1}, \ldots, n_{s}$ are odd, independently of $r$. The same holds if $n_{0}=0$.

(d) The flag manifold $\D(\vec{n})$ with $n_{0}>0$ is $G$-spin, if and only if all the numbers $n_{1}, \ldots, n_{s}$ are odd, independently of $r$.  If $n_{0}=0$  and $r>0$, then $\D(\vec{n})$ is $G$-spin, if and only if all the numbers $n_{1}, \ldots, n_{s}$ are odd. Finally, for  $n_{0}=0=r$, the flag manifold $\D(\vec{n})$ is $G$-spin, if and only if
 the numbers $n_{1}, \ldots, n_{s}$ have the same parity.
\end{theorem}
% \begin{comment}
\begin{proof}
{\bf Case of   $\A(\vec{n})$.} Assume first that $n_{0}>0$. We use   Corollary \ref{CLASFLAGS} and examine the divisibility of  Koszul numbers by two. We see that the Koszul numbers $1+ n_1, n_1+n_2, \cdots,  n_{s-1}+ n_s$ are all even if and only if  all $n_{1}, \ldots, n_{s}$ are odd.
Assume now that $n_{0}=0$. Then the Koszul vector is given  by $\vec{k} =  (n_1+n_2, \cdots,  n_{s-1}+ n_s)\in\bb{Z}_{+}^{s-1}$, with $n_{1}+n_{2}$ being the Koszul number of the first black simple root $\be_{n_{1}}$. Hence, in this case $\A(\vec{n})$ is $G$-spin, if and only if all $n_{i}$ have the same parity.\\ % for any $i=1, \ldots, s$.\\
 {\bf Case of   $\B(\vec{n})$.} Assume that both  $n_{0}\neq 0$ and $r\neq 0$. Then, the Koszul numbers $1+ n_1, n_1+n_2, \cdots,  n_{s-1}+ n_s$ are   even if and only if  all $n_{1}, \ldots, n_{s}$ are  odd. But  the Koszul number   of the last black root is given by $n_{s}+2r$ which is odd, independently of $r\neq 0$. Thus, the first Chern class is not even and $\B(\vec{n})$ is not   spin.  Assume  now that $n_{0}>0$ and $r=0$. Then, although the Koszul vector changes $\vec{k}=(2,\cdots, 2, 1+ n_1, n_1+n_2, \cdots,  n_{s-1}+ n_s, 2n_{s})\in\bb{Z}^{n_{0}+s}_{+}$, the first Chern class is even,    if and only if all   $n_{i}$ $(i=1, \ldots, s)$ are odd.  For   $n_{0}=0$,   by  Corollary \ref{CLASFLAGS}  the Koszul vector reads $\vec{k}=(n_1+n_2, \cdots,  n_{s-1}+ n_s, n_s+ 2r )\in\bb{Z}_{+}^{s}$, and for  $r=0$ it has the form $\vec{k}=(n_1+n_2, \cdots,  n_{s-1}+ n_s, 2n_s)$. In the first case  we deduce   that $\B(\vec{n})$ is $G$-spin, if and only if the numbers $n_{1}, \ldots, n_{s}$ are even, and in the second case, if and only if the numbers $n_{1}, \ldots, n_{s}$ have the same parity. \\
 {\bf Case of   $\Cc(\vec{n})$.} If $n_{0}>0$, then the Koszul numbers $1+n_{1}, n_1+n_2, \cdots,  n_{s-1}+ n_s$ are all even if and only if $n_{1}, \ldots, n_{s}$ are  odd. Since the last Koszul number is given by $n_{s}+2r+1$,  the first Chern class will be even, independently of $r$. The same is true  for $n_{0}>0$ and $r=0$, i.e. when the Koszul vector changes: $ \vec{k}=(2,\cdots, 2, 1+ n_1, n_1+n_2, \cdots,  n_{s-1}+ n_s, n_s+1)\in\bb{Z}^{n_{0}+s}_{+}$. If $n_{0}=0$ but $r>0$, then $\vec{k}=(n_1+n_2, \cdots,  n_{s-1}+ n_s, n_s+2r+1)\in\bb{Z}^{s}_{+}$, and   similarly  $\Cc(\vec{n})$ is $G$-spin if and only if all $n_{1}, \ldots, n_{s}$ are odd, independently of $r>0$. Finally, if $n_{0}=0=r$, then
 $ \vec{\kappa}=(n_1+n_2, \cdots,  n_{s-1}+ n_s, n_s+1)\in\bb{Z}^{s}_{+}$, which yields our assertion: the first Chern class is even if and only if all $n_{1}, \ldots, n_{s}$ are odd.\\
  {\bf Case of   $\D(\vec{n})$.} Assume first that $n_{0}>0$ and $r>0$. Then,  the Koszul numbers $1+ n_1, n_1+n_2, \cdots,  n_{s-1}+ n_s$ are all even if and only if    $n_{1}, \ldots, n_{s}$ are  odd. Since the Koszul number   of the last black root $\be_{n_{s}}$ root is given by $n_{s}+2r-1$, our claim follows. If $n_{0}>0$ but $r=0$, then the Koszul vector changes $\vec{k} = (2,\cdots, 2, 1+ n_1, n_1+n_2, \cdots,  n_{s-1}+ n_s, 2(n_s-1))\in\bb{Z}_{+}^{n_{0}+s}$.  Because  the Koszul numbers $1+ n_1, n_1+n_2, \cdots,  n_{s-1}+ n_s$ are simultaneously  even if and only if    $n_{1}, \ldots, n_{s}$ are  odd, the results follows due to $\vec{k}$.  Assume now that $n_{0}=0$ and $r>0$. Then, $\vec{k} = (n_1+n_2, \cdots,  n_{s-1}+ n_s, n_{s}+2r-1)\in\bb{Z}_{+}^{s}$ and since $2r-1$ is odd we conclude that $\D(\vec{n})$ is $G$-spin if and only if all $n_{1}, \ldots, n_{s}$ are odd. For  $n_{0}=0=r$, the  Koszul vector is given by $\vec{k} = (n_1+n_2, \cdots,  n_{s-1}+ n_s, 2(n_s-1))\in\bb{Z}_{+}^{s}$, hence the first Chern class associate to $J_{0}$ is even if and only if the numbers $n_{1}, \ldots, n_{s}$ have the same parity.
 \end{proof}
 \begin{remark}\label{exce} \textnormal{
 Given an exceptional flag manifold $F=G/H$,  a simple algorithm for the  computation of the Koszul integers   is given as follows:\\
(a) Consider the natural invariant ordering $R_{F}^{+}=R^{+}\backslash R_{H}^{+}$ induced by the splitting $\Pi=\Pi_{W}\sqcup\Pi_{B}$. Let us denote by $J_{0}$ the corresponding complex structure.  Describe  the root system $R_{H}$ and  compute $\sigma_{H}:=\frac{1}{2}\sum_{\be\in\\R_{H}^{+}}\be$ (in terms of simple roots). For the root systems of  the Lie groups $\G_2, \F_4, \E_6, \E_7, \E_8$ we use the notation of \cite{AA} (see also \cite{Freud, GOV}),  with a   difference in the enumeration of the bases of simple roots for  $\G_2$ and $\F_4$.\\
(b) Apply the formula $2(\sigma_{G}-\sigma_{H})=\sum_{\gamma\in R_{F}^{+}}\gamma:=\sigma^{J_{0}}$.   In particular,  for the exceptional simple Lie groups and with respect to   the   fixed bases  of the associated roots systems, it is  $2\sigma_{\G_{2}}=6\al_1+10\al_2$,
 \begin{eqnarray*}
2\sigma_{\F_{4}}&=&16\al_1+30\al_2+42\al_3+22\al_4,\\
2\sigma_{\E_6}&=&16\al_1+30\al_2+42\al_3+30\al_4+16\al_5+22\al_6,\\
2\sigma_{\E_7}&=&27\al_1+52\al_2+75\al_3+96\al_4+66\al_5+34\al_6+49\al_7,\\
2\sigma_{\E_8}&=&58\al_1+114\al_2+168\al_3+220\al_4+270\al_5+182\al_6+92\al_7+136\al_8.
\end{eqnarray*}
%Recall also that $2\sigma_{\SU_{\ell+1}}=\ell\al_1+2(\ell-1)\al_2+\cdot+j(\ell-j+1)\al_{j}+\cdots+\ell\al_{\ell}$ with $\ell\geq 1$, and
%\begin{eqnarray*}
%2\sigma_{\SU_{\ell+1}}&=&\ell\al_1+2(\ell-1)\al_2+\cdot+j(\ell-j+1)\al_{j}+\cdots+\ell\al_{\ell}, \ \ (\ell\geq 1),\\
%2\sigma_{\SO_{2\ell}}&=&2(\ell-1)\al_1+2(2\ell-3)\al_2+\cdots+2(j\ell-\frac{j(j+1)}{2})\al_{j}+\cdots+\frac{\ell(\ell-1)}{2}(\al_{\ell-1}+\al_{\ell}), \ \ (\ell\geq 3).
%\end{eqnarray*}
(c) Use the Cartan matrix $\cal{C}=(c_{i, j})=\big(\frac{2(\al_{i}, \al_{j})}{(\al_{j}, \al_{j})}\big)$    associated to  the basis $\Pi$ (and its enumeration), to express the simple roots   in terms of fundamental weights   via the formula $\al_{i}=\sum_{j=1}^{\ell} c_{i, j}\Lambda_{j}$. }
 \end{remark}
%\end{comment}
     \subsection{Classical spin flag  manifolds with $b_{2}=1$ or $b_{2}=2$} Let us examine which of the flag manifolds $F=G/H$ of a classical Lie group $G$, with $b_{2}(F)=1$ or $b_{2}(F)=2$ admit a spin (or metaplectic) structure. The classification of flag manifolds with $b_{2}(F)=1$ is well-known (cf. \cite{CS}).  Next  we also  classify all   classical flag manifolds with $b_{2}(F)=2$ (with respect to Lemma \ref{wellknown1}), and present the first Chern class (verifying in specific cases some previous results, which we cite).

       \begin{example}\label{cpn}\textnormal{({\bf Spin flag manifolds of $\A_{n}$ with $b_{2}=1, 2$})}

   \textnormal{(a) Consider the   Hermitian symmetric space $\SU_{n}/\Ss(\U_{p}\times \U_{n-p})$, with $1\leq p\leq n-1$.  The   painted Dynkin diagram has the form
    \[
  \begin{picture}(500, 0)(-10, 3)
    \put(166, 0.5){\circle{5}}
  \put(166, 5){\line(0,3){10}}
\put(166, 15){\line(1,0){11}}
\put(190, 15){\makebox(0,0){{\tiny{$p-1$}}}}
\put(203, 15){\line(1,0){13}}
\put(216, 15){\line(0,-3){10}}
    \put(168, 0.5){\line(1,0){14}}
     \put(180, 0){ $\ldots$}
  \put(200, 0.5){\line(1,0){13.6}}
   \put(216, 0.5){\circle{5}}
  %  \put(216,6){\makebox(0,0){$\al_{l+m+n-1}$}}
    \put(218, 0.5){\line(1,0){13.6}}
\put(232, 0.5){\circle*{5}}
  \put(232,10){\makebox(0,0){$\al_p$}}
  \put(233.9, 0.5){\line(1,0){13}}
  \put(249, 0.5){\circle{5}}
  \put(249, 5){\line(0,3){10}}
\put(249, 15){\line(1,0){5}}
\put(273, 15){\makebox(0,0){{\tiny{$n-p-1$}}}}
\put(291, 15){\line(1,0){8}}
\put(299, 15){\line(0,-3){10}}
    \put(251, 0.5){\line(1,0){13.6}}
  \put(262.5, 0){ $\ldots$}
  \put(282.5, 0.5){\line(1,0){13.6}}
   \put(299, 0.5){\circle{5}}
         \end{picture}
 \] and in terms of Corollary \ref{CLASFLAGS}, it is $n_{0}=0$, $s=2$, $n_{1}=p=n_{s-1}$, and $n_{2}=n-p=n_{s}$.  The Koszul form associated to the unique $G$-invariant complex structure $J_{0}$ is  $\sigma^{J_{0}}=n\Lambda_{p}$, in particular the Koszul number depends only on $n$. We deduce that $\SU_{n}/\Ss(\U_{p}\times \U_{n-p})$ is $G$-spin if and only if $n$ is even, see also \cite[Thm.~8]{Cahen}.  Notice that the special case $p=1$ induces the   complex projective space $\bb{C}P^{n-1}=\SU_{n}/\Ss(\U_{1}\times \U_{n-1})$. Here we get  $\sigma^{J_{0}}=n\Lambda_{1}$ and we recover    the well-known fact that $\bb{C}P^{n-1}$ is spin if and only if $n$ is even (or equivalently, $\bb{C}P^{n}$ is spin if and only if $n$ is odd, see \cite{Fried, Law}). }

 \textnormal{(b) Consider the painted Dynkin diagram
 \[
 \begin{picture}(300, 0)(-10, 3)
 \put(67,10){\makebox(0,0){$\al_{1}$}}
  \put(67, 0.5){\circle*{5}}
%  \put(67,10){\makebox(0,0){$\al_{n_{0}}$}}
\put(68, 0.5){\line(1,0){13}}
\put(83, 0.5){\circle{5}}
\put(83, 5){\line(0,3){10}}
\put(83, 15){\line(1,0){11}}
\put(107, 15){\makebox(0,0){{\tiny{$p-2$}}}}
\put(119, 15){\line(1,0){12}}
\put(131, 15){\line(0,-3){10}}
%\put(89,6){\makebox(0,0){$\al_{k+1}$}}
 \put(85, 0.5){\line(1,0){14}}
\put(97, 0){ $\ldots$}
\put(115, 0.5){\line(1,0){14}}
\put(131, 0.5){\circle{5}}
\put(133, 0.5){\line(1,0){14}}
\put(149, 0.5){\circle*{5}}
 \put(149,10){\makebox(0,0){$\al_{p}$}}
  \put(150.9, 0.5){\line(1,0){12}}
  \put(166, 0.5){\circle{5}}
  \put(166, 5){\line(0,3){10}}
\put(166, 15){\line(1,0){4}}
\put(190, 15){\makebox(0,0){{\tiny{$n-p-1$}}}}
\put(208, 15){\line(1,0){8}}
\put(216, 15){\line(0,-3){10}}
    \put(168, 0.5){\line(1,0){14}}
     \put(180, 0){ $\ldots$}
  \put(200, 0.5){\line(1,0){13.6}}
   \put(216, 0.5){\circle{5}}
  %  \put(216,6){\makebox(0,0){$\al_{l+m+n-1}$}}
 %   \put(218, 0.5){\line(1,0){13.6}}
%\put(232, 0.5){\circle{5}}
 %\put(232,10){\makebox(0,0){$\al_{n-1}$}}
          \end{picture}
 \]
with $2\leq p \leq n-1$ and $n\geq 3$.  This gives rise to    $\SU_{n}/\U_{p-1}\times\U_{n-p}\cong \SU_{n}/\U_{1}\times\Ss(\U_{p-1}\times\U_{n-p})$, with $d=3$.  In particular,  any flag manifold of $\A_{n}$ with $b_{2}=2$ and $n_{0}>0$ is equivalent to  such a coset, for   appropriate $n, p$.  It is $n_{0}=1$, $s=2$, $n_1=p-1=n_{s-1}$, $n_{2}=n-p=n_{s}$. Thus, the Koszul vector is given by $\vec{k}=(1+n_{1}, n_{1}+n_{2})=(p, n-1)\in\bb{Z}_{+}^{2}$.   Consequently, $\SU_{n}/\U_{1}\times\Ss(\U_{p-1}\times\U_{n-p})$ is $G$-spin if and only if $n$ is odd and $p$ is even. For the special case $p=2$, i.e. $\SU_{n}/\U_{1}\times\Ss(\U_{1}\times\U_{n-2})\cong \SU_{n}/\U_{1}^{2}\times\SU_{n-2}\cong\SU_{n}/\U_{1}\times\U_{n-2}$, the Koszul vector is given by $\vec{k}=(2, n-1)$, thus there exists a $G$-invariant spin structure if and only if $n$ is odd. For $n=3$ one gets the full flag manifold $\SU_{3}/{\rm T}^{2}$.}

    \textnormal{(c) Consider the space $\SU_{n}/\Ss(\U_{p}\times\U_{q}\times\U_{n-p-q})$, with painted Dynkin diagram  \[
  \begin{picture}(400, 0)(-10, 8)
  % \put(-20, 0.5){\circle*{5}}
    %\put(-18, 0.5){\line(1,0){17}}
%\put(0, 0.5){\circle*{5}}
%\put(0, 5){\line(0,3){10}}
%\put(0, 15){\line(1,0){25}}
%\put(33, 15){\makebox(0,0){{\tiny{$n_{0}$}}}}
%\put(41, 15){\line(1,0){26}}
%\put(67, 15){\line(0,-3){10}}
%\put(0,7){\makebox(0,0){$\al_{1}$}}
 %\put(2, 0.5){\line(1,0){14}}
%\put(18, 0.5){\circle*{5}}
%\put(18,7){\makebox(0,0){$\al_2$}}
 %\put(20, 0.5){\line(1,0){14}}
%\put(32, 0){ $\ldots$}
%\put(51, 0.5){\line(1,0){14}}
 %\put(67, 0.5){\circle*{5}}
%  \put(67,10){\makebox(0,0){$\al_{n_{0}}$}}
%\put(68, 0.5){\line(1,0){13}}
\put(83, 0.5){\circle{5}}
\put(83, 5){\line(0,3){10}}
\put(83, 15){\line(1,0){11}}
\put(107, 15){\makebox(0,0){{\tiny{$p-1$}}}}
\put(119, 15){\line(1,0){12}}
\put(131, 15){\line(0,-3){10}}
%\put(89,6){\makebox(0,0){$\al_{k+1}$}}
 \put(85, 0.5){\line(1,0){14}}
\put(97, 0){ $\ldots$}
\put(115, 0.5){\line(1,0){14}}
\put(131, 0.5){\circle{5}}
\put(133, 0.5){\line(1,0){14}}
\put(149, 0.5){\circle*{5}}
 \put(149,10){\makebox(0,0){$\al_{p}$}}
  \put(150.9, 0.5){\line(1,0){13}}
  \put(166, 0.5){\circle{5}}
  \put(166, 5){\line(0,3){10}}
\put(166, 15){\line(1,0){11}}
\put(190, 15){\makebox(0,0){{\tiny{$q-1$}}}}
\put(203, 15){\line(1,0){13}}
\put(216, 15){\line(0,-3){10}}
    \put(168, 0.5){\line(1,0){14}}
     \put(180, 0){ $\ldots$}
  \put(200, 0.5){\line(1,0){13.6}}
   \put(216, 0.5){\circle{5}}
  %  \put(216,6){\makebox(0,0){$\al_{l+m+n-1}$}}
    \put(218, 0.5){\line(1,0){13.6}}
\put(232, 0.5){\circle*{5}}
 \put(232,10){\makebox(0,0){$\al_{p+q}$}}
  \put(233.9, 0.5){\line(1,0){13}}
  \put(249, 0.5){\circle{5}}
   \put(249, 5){\line(0,3){10}}
\put(249, 15){\line(1,0){1}}
\put(273.5, 15){\makebox(0,0){{\tiny{$n-p-q-1$}}}}
\put(296.7, 15){\line(1,0){2.5}}
\put(299, 15){\line(0,-3){10}}
    \put(251, 0.5){\line(1,0){13.6}}
  \put(262.5, 0){ $\ldots$}
  \put(282.5, 0.5){\line(1,0){13.6}}
   \put(299, 0.5){\circle{5}}
 %  \put(306, 8){\makebox(0,0){$\al_{n-1}$}}
           \end{picture}
 \]
 \vskip 0.2cm
 \noindent   It is $\rnk R_{T}=2$, in particular any  flag manifold    of $\A_{n-1}=\SU_{n}$ $(n\geq 6)$ with second Betti number $b_{2}=2$ and $n_{0}=0$, is equivalent to this family, for appropriate   $p, q$ with $2\leq p< n-3$ with $4\leq p+q\leq n-1$.  The isotropy representation  decomposes into three isotropy summands, i.e. $d=3$.  In terms of our notation, it is $n_{0}=0$, $s=3$, $n_{1}=p$, $n_{2}=q$, $n_{3}=n-p-q$ and the Koszul vector is given by   $\vec{k}=(n_{1}+n_{2}=p+q, n_{2}+n_{3}=n-p)\in\bb{Z}_{+}^{2}$ (see also \cite[Thm.~3.1, 3.2]{Kim} but be aware of a slightly different normalization of $\delta_{\fr{m}}$).
 %\[
 % \vec{k}=(n_{1}+n_{2}=p+q, n_{2}+n_{3}=n-p)\in\bb{Z}_{+}^{2}.
 % \]
   Hence,  $\SU_{n}/\Ss(\U_{p}\times\U_{q}\times\U_{n-p-q})$ is $G$-spin if and only if $p, q, n$ have all the same parity.}
   \end{example}

  %\begin{example}
  %\textnormal{Consider the Hermitian symmetric space $\D(\vec{n})=\SO_{2n}/\SO_{2}\times\SO_{2(n-1)}$.}
  %\end{example}

  \begin{example}\label{SONFLAGS} \textnormal{({\bf Spin flag manifolds of $\B_{n}$ with $b_{2}=1, 2$})}

  \textnormal{(a) We start with the flag manifold $\SO_{2n +1}/\U_{p}\times \SO_{2(n-p)+1}$ with $2\leq p\leq n$, $n\geq 3$ and PDD
  
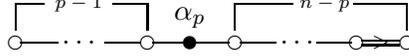
\[
   \begin{picture}(500,0)(-10, 3)
  \put(166, 0.5){\circle{5}}
  \put(166, 5){\line(0,3){10}}
\put(166, 15){\line(1,0){11}}
\put(190, 15){\makebox(0,0){{\tiny{$p-1$}}}}
\put(203, 15){\line(1,0){13}}
\put(216, 15){\line(0,-3){10}}
    \put(168, 0.5){\line(1,0){14}}
     \put(180, 0){ $\ldots$}
  \put(200, 0.5){\line(1,0){13.6}}
   \put(216, 0.5){\circle{5}}
  %  \put(216,6){\makebox(0,0){$\al_{l+m+n-1}$}}
    \put(218, 0.5){\line(1,0){13.6}}
\put(232, 0.5){\circle*{5}}
 \put(232,10){\makebox(0,0){$\al_{p}$}}
  \put(233.9, 0.5){\line(1,0){13}}
  \put(249, 0.5){\circle{5}}
  \put(249, 5){\line(0,3){10}}
\put(249, 15){\line(1,0){22}}
\put(283, 15){\makebox(0,0){{\tiny{$n-p$}}}}
\put(294, 15){\line(1,0){20}}
\put(314, 15){\line(0,-3){10}}
  \put(251, 0.5){\line(1,0){13.6}}
  \put(262.5, 0){ $\ldots$}
  \put(280, 0.5){\line(1,0){13}}
    \put(295, 0.5){\circle{5}}
  \put(297, 1.1){\line(1,0){14}}
\put(297, -0.6){\line(1,0){14}}
\put(301, -1.5){\scriptsize $>$}
%\put(103.5, -3){$\Longleftarrow$}
\put(314, 0.5){\circle{5}}
         \end{picture}
  \]
 and since $\Hgt(\al_{p})=2$ the isotropy representation decomposes into 2 irreducible submodules.  In terms of Corollary \ref{CLASFLAGS}, it is $n_{0}=0$, $s=1$, $n_{1}=p=n_{s}$ and $r=n-p$. For $p< n$, the unique  Koszul number is given by $n_{s}+2r=p+2(n-p)$, hence $\sigma^{J_{0}}=(2n-p)\Lambda_{p}$ and  $\B(\vec{n})$ is $G$-spin if and only if $2n-p$  is even, which is equivalent to say that $p$ is even.  For $p=n$, the Koszul vector is different, in particular one gets the manifold $\SO_{2n+1}/\U_{n}$. In this case,  the last simple root is black  and this results to the expression of the Koszul number;    $\sigma^{J}=2n\Lambda_{n}$. Thus  $\SO_{2n+1}/\U_{n}$ is $G$-spin for any $n$. %This applies also to the very  special case $n=1$, i.e. $\SO_{3}/\U_{1}=\bb{C}P^{1}$, which is spin.
 Finally, the missing case $p=1$, i.e. $\al_{p}=\al_{1}$ induces the isotropy irreducible Hermitian symmetric space $\SO_{2n+1}/\SO_{2}\times \SO_{2n-1}$ (since $\Hgt(\al_{1})=1$ for $G=\B_{n}$). In our notation, it is $n_{0}=1$, $s=1$, $n_{1}=1=n_{s}$ and $r=n-1$. The Koszul form is given by $\sigma^{J_{0}}=(n_{s}+2r)\Lambda_{1}=(2n-1)\Lambda_{1}$ and this Hermitian symmetric space is not spin, see also \cite[Thm.~8]{Cahen}.  }

     \textnormal{(b) Let $\B(\vec{n})=\SO_{2n+1}/\U_{1}\times\U_{1}\times\SO_{2n-3}$, with $n\geq 3$. The painted Dynkin diagram is given by
  
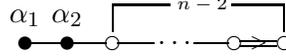
\[
   \begin{picture}(700,0)(0, 3)
      \put(216, 0.5){\circle*{5}}
       \put(216,10){\makebox(0,0){$\al_{1}$}}
  %  \put(216,6){\makebox(0,0){$\al_{l+m+n-1}$}}
    \put(218, 0.5){\line(1,0){13.6}}
\put(232, 0.5){\circle*{5}}
 \put(232,10){\makebox(0,0){$\al_{2}$}}
  \put(233.9, 0.5){\line(1,0){13}}
  \put(249, 0.5){\circle{5}}
  \put(249, 5){\line(0,3){10}}
\put(249, 15){\line(1,0){22}}
\put(283, 15){\makebox(0,0){{\tiny{$n-2$}}}}
\put(294, 15){\line(1,0){20}}
\put(314, 15){\line(0,-3){10}}
  \put(251, 0.5){\line(1,0){13.6}}
  \put(262.5, 0){ $\ldots$}
  \put(280, 0.5){\line(1,0){13}}
    \put(295, 0.5){\circle{5}}
  \put(297, 1.1){\line(1,0){14}}
\put(297, -0.6){\line(1,0){14}}
\put(301, -1.5){\scriptsize $>$}
%\put(103.5, -3){$\Longleftarrow$}
\put(314, 0.5){\circle{5}}
         \end{picture}
  \]
 This is a flag manifold with $\rnk R_{T}=2$ and the isotropy representation decomposes into  four irreducible submodules, $d=\sharp(R_{T}^{+})=4$.  The Koszul form for $J_{0}$ has been computed in \cite{Chry2}; $\sigma^{J_{0}}=2\Lambda_{1}+(2n-3)\Lambda_{2}$.  Indeed, in terms of Corollary \ref{CLASFLAGS} it is $n_{0}=1$, $s=1$, $n_{1}=1=n_{s}$ and $r=\ell-2$, thus the Koszul vector is given by $  \vec{k}=(2, n_{s}+2r)=(2, 2n-3)\in\bb{Z}_{+}^{2}$.       Therefore, for $n\geq 3$  the space $\B(\vec{n})$ is not spin.  Notice that for $n=2$,   $\B(\vec{2})$ coincides with the full flag manifold $\SO_{5}/{\rm T}^{2}$, which still has 4 isotropy summands. However, in this case the Koszul vector changes, since the last simple root is painted black, namely $\vec{\kappa}=(2, 2)\in\bb{Z}_{+}^{2}$ and $\B(\vec{2})$ is spin, see also Example \ref{FULL}.}

    \textnormal{(c) Let us draw now the painted Dynkin diagram
  
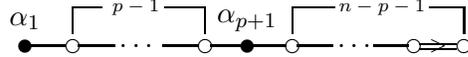
\[
    \begin{picture}(460,0)(-10, 3)
     \put(148, 0.5){\circle*{5}}
      \put(148,10){\makebox(0,0){$\al_{1}$}}
  \put(150, 0.5){\line(1,0){13}}
  \put(166, 0.5){\circle{5}}
  \put(166, 5){\line(0,3){10}}
\put(166, 15){\line(1,0){11}}
\put(190, 15){\makebox(0,0){{\tiny{$p-1$}}}}
\put(201, 15){\line(1,0){15}}
\put(216, 15){\line(0,-3){10}}
    \put(168, 0.5){\line(1,0){14}}
     \put(180, 0){ $\ldots$}
  \put(200, 0.5){\line(1,0){13.6}}
   \put(216, 0.5){\circle{5}}
  %  \put(216,6){\makebox(0,0){$\al_{l+m+n-1}$}}
    \put(218, 0.5){\line(1,0){13.6}}
\put(232, 0.5){\circle*{5}}
 \put(232,10){\makebox(0,0){$\al_{p+1}$}}
  \put(233.9, 0.5){\line(1,0){13}}
  \put(249, 0.5){\circle{5}}
  \put(249, 5){\line(0,3){10}}
\put(249, 15){\line(1,0){16}}
\put(283, 15){\makebox(0,0){{\tiny{$n-p-1$}}}}
\put(301, 15){\line(1,0){13}}
\put(314, 15){\line(0,-3){10}}
  \put(251, 0.5){\line(1,0){13.6}}
  \put(262.5, 0){ $\ldots$}
  \put(280, 0.5){\line(1,0){13}}
    \put(295, 0.5){\circle{5}}
  \put(297, 1.1){\line(1,0){14}}
\put(297, -0.6){\line(1,0){14}}
\put(301, -1.5){\scriptsize $>$}
%\put(103.5, -3){$\Longleftarrow$}
\put(314, 0.5){\circle{5}}
         \end{picture}
    \]
 which determines the flag space $\SO_{2n+1}/\U_{1}\times\U_{p}\times\SO_{2n-2p-1}$ with $n\geq 3$ and $2\leq p\leq n-1$, with five isotropy summands \cite{ACS1}. The Koszul form has been computed in \cite{ACS1} for $p\neq n-1$; it is given by $\sigma^{J_{0}}=(p+1)\Lambda_{1}+(2n-p-2)\Lambda_{p+1}$. In our notation it is $n_{0}=1$, $s=1$, $n_{1}=p=n_{s}$, $r=n-p-1$. Thus, for $r\neq 0$ the Koszul vector reads $ \vec{k}=(1+n_{1}, n_{s}+2r)=(p+1, 2n-p-2)$          and the associated  flag manifold is not  spin.   In the special case $r=0$, i.e. $p=n-1$,  one takes the painted Dynkin diagram
   $
    \begin{picture}(93,0)(225, -2)
 \put(232, 0.5){\circle*{5}}
 %\put(232,7){\makebox(0,0){$\al_{n_2}$}}
  \put(233.9, 0.5){\line(1,0){13}}
  \put(249, 0.5){\circle{5}}
  \put(251, 0.5){\line(1,0){13.6}}
  \put(262.5, 0){ $\ldots$}
  \put(280, 0.5){\line(1,0){13}}
    \put(295, 0.5){\circle{5}}
  \put(297, 1.1){\line(1,0){14}}
\put(297, -0.6){\line(1,0){14}}
\put(301, -1.5){\scriptsize $>$}
%\put(103.5, -3){$\Longleftarrow$}
\put(314, 0.5){\circle*{5}}
         \end{picture},
   $     with $\Pi_{M}=\{\al_1, \al_{n}\}$.  It defines the space $\SO_{2n+1}/\U_{1}\times\U_{n-1}$.  In this case, the last simple root is painted black, so the Koszul vector is given by $\vec{k} =(1+n_1, 2n_{s})=(n, 2(n-1))$.
     Hence,  $\SO_{2n+1}/\U_{1}\times\U_{n-1}$ is $G$-spin if and only if $n$ is even.  }

  \textnormal{(d) The final class of flag manifolds of $\B_{n}$ with $b_{2}(M)=2=\rnk R_{T}$ is given by the painted Dynkin diagram
  
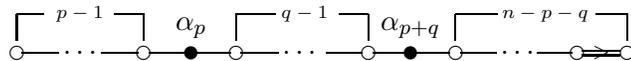
\[
   \begin{picture}(400,0)(-10, -3)
 \put(83, 0.5){\circle{5}}
\put(83, 5){\line(0,3){10}}
\put(83, 15){\line(1,0){10}}
\put(107, 15){\makebox(0,0){{\tiny{$p-1$}}}}
\put(120, 15){\line(1,0){11}}
\put(131, 15){\line(0,-3){10}}
%\put(89,6){\makebox(0,0){$\al_{k+1}$}}
 \put(85, 0.5){\line(1,0){14}}
\put(97, 0){ $\ldots$}
\put(115, 0.5){\line(1,0){14}}
\put(131, 0.5){\circle{5}}
\put(133, 0.5){\line(1,0){14}}
\put(149, 0.5){\circle*{5}}
 \put(149,10){\makebox(0,0){$\al_{p}$}}
  \put(150.9, 0.5){\line(1,0){13}}
  \put(166, 0.5){\circle{5}}
  \put(166, 5){\line(0,3){10}}
\put(166, 15){\line(1,0){15}}
\put(192, 15){\makebox(0,0){{\tiny{$q-1$}}}}
\put(205, 15){\line(1,0){11}}
\put(216, 15){\line(0,-3){10}}
    \put(168, 0.5){\line(1,0){14}}
     \put(180, 0){ $\ldots$}
  \put(200, 0.5){\line(1,0){13.6}}
   \put(216, 0.5){\circle{5}}
  %  \put(216,6){\makebox(0,0){$\al_{l+m+n-1}$}}
    \put(218, 0.5){\line(1,0){13.6}}
\put(232, 0.5){\circle*{5}}
 \put(232,10){\makebox(0,0){$\al_{p+q}$}}
  \put(233.9, 0.5){\line(1,0){13}}
  \put(249, 0.5){\circle{5}}
  \put(249, 5){\line(0,3){10}}
\put(249, 15){\line(1,0){16}}
\put(283, 15){\makebox(0,0){{\tiny{$n-p-q$}}}}
\put(302, 15){\line(1,0){12}}
\put(314, 15){\line(0,-3){10}}
  \put(251, 0.5){\line(1,0){13.6}}
  \put(262.5, 0){ $\ldots$}
  \put(280, 0.5){\line(1,0){13}}
    \put(295, 0.5){\circle{5}}
  \put(297, 1.1){\line(1,0){14}}
\put(297, -0.6){\line(1,0){14}}
\put(301, -1.5){\scriptsize $>$}
%\put(103.5, -3){$\Longleftarrow$}
\put(314, 0.5){\circle{5}}
         \end{picture}
  \]
\noindent with $2\leq p\leq  n-2$, $4\leq p+q\leq n$ and $n\geq 4$,  such that $\Hgt(\al_{p})=2=\Hgt(\al_{p+q})$.  It gives rise to $\SO_{2n+1}/\U_{p}\times\U_{q}\times\SO_{2(n-p-q)+1}$, which   is a flag manifold with six isotropy summands. It is $n_{0}=0$, $s=2$, $n_{1}=p=n_{s-1}$, $n_{2}=q=n_{s}$, $r=n-p-q$.  Assume that $q\neq n-p$, i.e. that the black root $\al_{p+q}$ is not the end simple root. Then, the Koszul vector is given by $\vec{k}=(n_{1}+n_{2}=p+q, n_{s}+2r=2n-2p-q)$ and  this coset is $G$-spin if and only if both $p, q$ are even, independent of $n$. Suppose now that $q=n-p$. In this case we take the flag manifold $\SO_{2n+1}/\U_{p}\times \U_{n-p}$ with $n\geq 4$ and $2\leq p\leq n$ which still has $d=6$ isotropy summands, but the Koszul vector changes. In particular, since $n_{0}=0$, $n_{1}=p=n_{s-1}$, $n_{2}=n-p=n_{s}$ and $r=0$, we conclude that  $\vec{k}=(n_{1}+n_{2}, 2n_{s})=(n, 2(n-p))\in\bb{Z}_{+}^{2}$.  Thus, this coset is $G$-spin if and only if $n$ is even, independently of $p$. Such an interesting case occurs already for $n=4$ and $\B_{4}$. Consider for example  the painted Dynkin diagram
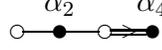
\[
  \begin{picture}(600,0)(-10,  2)
    \put(216, 0.5){\circle{5}}
  %  \put(216,6){\makebox(0,0){$\al_{l+m+n-1}$}}
    \put(218, 0.5){\line(1,0){13.6}}
\put(232, 0.5){\circle*{5}}
 \put(232,10){\makebox(0,0){$\al_{2}$}}
  \put(233.9, 0.5){\line(1,0){13}}
  \put(249, 0.5){\circle{5}}
  %\put(249, 5){\line(0,3){10}}
\put(251, 1.1){\line(1,0){13.6}}
\put(251, -0.6){\line(1,0){13.6}}
\put(255, -1.5){\scriptsize $>$}
 \put(267, 0.5){\circle*{5}}
  \put(267,10){\makebox(0,0){$\al_{4}$}}
          \end{picture}
\] It gives rise to the flag manifold $F=\SO_{9}/\U_{2}\times\U_{2}$. It is $R_{H}^{+}=\{\al_1, \al_3\}$ and $2\sigma_{H}=\al_1+\al_3$. Since $2\sigma_{\SO_{9}}=7\al_1+12\al_2+15\al_3+16\al_4$, we conclude that $\sigma^{J_{0}}=6\al_1+12\al_2+14\al_3+16\al_4$. Using the Cartan matrix of $\SO_{9}$ we finally get   $\sigma^{J_{0}}=4\Lambda_{2}+4\Lambda_{4}$. Thus $F$ admits a unique spin structure.}
  \end{example}

 \begin{example}\label{SpFLAGS} \textnormal{({\bf Spin flag manifolds of $\Cc_{n}$ with $b_{2}=1, 2$})}

 \textnormal{(a) We start with  the flag manifold $\Cc(\vec{n})=\Sp_{n}/\U_{p}\times\Sp_{n-p}$ with $1\leq p\leq n-1$ and PDD
  
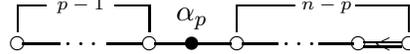
\[
   \begin{picture}(520,0)(-10, 3)
  \put(166, 0.5){\circle{5}}
  \put(166, 5){\line(0,3){10}}
\put(166, 15){\line(1,0){11}}
\put(190, 15){\makebox(0,0){{\tiny{$p-1$}}}}
\put(203, 15){\line(1,0){13}}
\put(216, 15){\line(0,-3){10}}
    \put(168, 0.5){\line(1,0){14}}
     \put(180, 0){ $\ldots$}
  \put(200, 0.5){\line(1,0){13.6}}
   \put(216, 0.5){\circle{5}}
  %  \put(216,6){\makebox(0,0){$\al_{l+m+n-1}$}}
    \put(218, 0.5){\line(1,0){13.6}}
\put(232, 0.5){\circle*{5}}
 \put(232,10){\makebox(0,0){$\al_{p}$}}
  \put(233.9, 0.5){\line(1,0){13}}
  \put(249, 0.5){\circle{5}}
  \put(249, 5){\line(0,3){10}}
\put(249, 15){\line(1,0){22}}
\put(283, 15){\makebox(0,0){{\tiny{$n-p$}}}}
\put(294, 15){\line(1,0){20}}
\put(314, 15){\line(0,-3){10}}
  \put(251, 0.5){\line(1,0){13.6}}
  \put(262.5, 0){ $\ldots$}
  \put(280, 0.5){\line(1,0){13}}
    \put(295, 0.5){\circle{5}}
  \put(297, 1.1){\line(1,0){14}}
\put(297, -0.6){\line(1,0){14}}
\put(301, -1.5){\scriptsize $<$}
%\put(103.5, -3){$\Longleftarrow$}
\put(314, 0.5){\circle{5}}
         \end{picture}
  \]
 It is  $\Pi_{B}=\{\al_{p} : 1\leq p\leq n-1\}$ and hence  $\Hgt(\al_{p})=2$ and $\fr{m}=\fr{m}_{1}\oplus\fr{m}_{2}$.  In terms of Corollary \ref{CLASFLAGS} it is $n_{0}=0$, $s=1$, $n_{1}=p=n_{s}$ and $r=n-p$.  Thus, the Koszul form is given by  $\sigma^{J}=n_{s}+2r+1=(2n-p+1)\Lambda_{p}$, so $\Cc(\vec{n})$ is spin if and only if $2n-p+1$   is even, which is equivalent to say that $p$ is odd.
   An important case here occurs for $p=1$, which defines the complex projective space $\bb{C}P^{2n-1}=\Sp_{n}/\U_{1}\times\Sp_{n-1}$. In particular,  for $p=1$  we see that $\sigma^{J_{0}}=2n\Lambda_{1}$ and $\bb{C}P^{2n-1}$ is spin for any $n$, as it should be since the coset $\Sp_{n}/\U_{1}\times\Sp_{n-1}$  gives  rise to   odd dimensional complex projective spaces. It worths also to remark the  exceptional case $p=n$, which induces the isotropy irreducible Hermitian symmetric space $\Sp_{n}/\U_{n}$ (since $\Hgt(\al_{n})=1$ for $G=\Sp_{n}$).  In this case $n_{0}=0$, $s=1$, $n_{1}=n=n_{s}$, $r=0$, and the unique Koszul number is given by $k_{n}=n_{s}+1=n+1$, i.e. $\sigma^{J_{0}}=(n+1)\Lambda_{n}$. Hence, $\Sp_{n}/\U_{n}$ is spin, if and only if $n$ is odd, see also \cite[Thm.~8]{Cahen}.}

  \textnormal{(b) Consider the painted Dynkin diagram
  
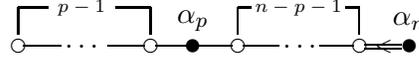
\[
   \begin{picture}(480,0)(-10, 3)
  \put(166, 0.5){\circle{5}}
  \put(166, 5){\line(0,3){10}}
\put(166, 15){\line(1,0){11}}
\put(190, 15){\makebox(0,0){{\tiny{$p-1$}}}}
\put(203, 15){\line(1,0){13}}
\put(216, 15){\line(0,-3){10}}
    \put(168, 0.5){\line(1,0){14}}
     \put(180, 0){ $\ldots$}
  \put(200, 0.5){\line(1,0){13.6}}
   \put(216, 0.5){\circle{5}}
  %  \put(216,6){\makebox(0,0){$\al_{l+m+n-1}$}}
    \put(218, 0.5){\line(1,0){13.6}}
\put(232, 0.5){\circle*{5}}
 \put(232,10){\makebox(0,0){$\al_{p}$}}
  \put(233.9, 0.5){\line(1,0){13}}
  \put(249, 0.5){\circle{5}}
  \put(249, 5){\line(0,3){10}}
\put(249, 15){\line(1,0){5}}
\put(272, 15){\makebox(0,0){{\tiny{$n-p-1$}}}}
\put(290, 15){\line(1,0){5}}
\put(295, 15){\line(0,-3){10}}
  \put(251, 0.5){\line(1,0){13.6}}
  \put(262.5, 0){ $\ldots$}
  \put(280, 0.5){\line(1,0){13}}
    \put(295, 0.5){\circle{5}}
  \put(297, 1.1){\line(1,0){14}}
\put(297, -0.6){\line(1,0){14}}
\put(301, -1.5){\scriptsize $<$}
%\put(103.5, -3){$\Longleftarrow$}
\put(314, 0.5){\circle*{5}}
 \put(314,10){\makebox(0,0){$\al_{n}$}}
         \end{picture}
  \] }
\textnormal{with $1\leq p\leq n-1$ and $n\geq 3$. Then it is always $\Hgt(\al_{p})=2$ and $\Hgt(\al_{n})=1$, hence we get a   flag manifold  with four isotropy summands, namely $\Sp_{n}/\U_{p}\times\U_{n-p}$. The Koszul form has been computed in  \cite{Chry2}; $\sigma^{J_{0}}= n\Lambda_{p}+(n-p+1)\Lambda_{n}$. Indeed, in terms of Corollary \ref{CLASFLAGS}, it is $n_{0}=0$, $s=2$, $n_{1}=p=n_{s-1}$, $n_{2}=n-p=n_{s}$ and $r=0$. Thus, the Koszul vector is given by $\vec{k}=(n_{1}+n_{2}, n_{s}+1)=(n, n-p+1)\in\bb{Z}^{2}_{+}$ and $\Sp_{n}/\U_{p}\times\U_{n-p}$ is spin or metaplectic,  if and only if $n$ is even and $p$ is odd.}

\textnormal{(c) Let us describe now the painted Dynkin diagram
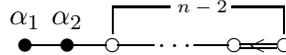
\[
 \begin{picture}(700,0)(0, 3)
      \put(216, 0.5){\circle*{5}}
       \put(216,10){\makebox(0,0){$\al_{1}$}}
  %  \put(216,6){\makebox(0,0){$\al_{l+m+n-1}$}}
    \put(218, 0.5){\line(1,0){13.6}}
\put(232, 0.5){\circle*{5}}
 \put(232,10){\makebox(0,0){$\al_{2}$}}
  \put(233.9, 0.5){\line(1,0){13}}
  \put(249, 0.5){\circle{5}}
  \put(249, 5){\line(0,3){10}}
\put(249, 15){\line(1,0){22}}
\put(283, 15){\makebox(0,0){{\tiny{$n-2$}}}}
\put(294, 15){\line(1,0){20}}
\put(314, 15){\line(0,-3){10}}
  \put(251, 0.5){\line(1,0){13.6}}
  \put(262.5, 0){ $\ldots$}
  \put(280, 0.5){\line(1,0){13}}
    \put(295, 0.5){\circle{5}}
  \put(297, 1.1){\line(1,0){14}}
\put(297, -0.6){\line(1,0){14}}
\put(300.5, -1.5){\scriptsize $<$}
%\put(103.5, -3){$\Longleftarrow$}
\put(314, 0.5){\circle{5}}
         \end{picture}
   \] It induces the flag manifold $\Sp_{n}/\U_{1}\times\U_{1}\times\Sp_{n-2}$, with six isotropy summands since $\Hgt(\al_1)=\Hgt(\al_2)=2$. It is $n_{0}=1$, $s=1$, $n_{1}=1=n_{s}$, $r=n-2$, so the Koszul vector associated to $J_{0}$ is given by $   \vec{k}=(2, n_{s}+2r+1)=(2, 2(n-1))\in\bb{Z}_{+}^{2}$.
      Hence this coset is spin for any $n\geq 3$. A generalization of the above painted Dynkin diagram is described  below.}

    \textnormal{(d) Set $\Pi_{B}=\{\al_{p}, \al_{p+q}\}$, with $1\leq p\leq n-3$, $3\leq p+q\leq n-1$ and $n\geq 4$, such that $\Hgt(\al_{p})=\Hgt(\al_{p+q})=2$. This choice corresponds to the flag manifold $  \Sp_{n}/\U_{p}\times\U_{q}\times\Sp_{n-p-q}$ with $\rnk R_{T}=2$ and  $d=\sharp(R_{T}^{+})=6$.    In terms of Corollary \ref{CLASFLAGS}, we get $n_{0}=0$, $s=2$, $n_{1}=p=n_{s-1}$, $n_{2}=2=n_{s}$, and $r=n-p-q$. Thus, the Koszul vector reads    $\vec{k}=(n_{1}+n_{2}, n_{s}+2r+1)=(p+q, 2n-2p-q+1)\in\bb{Z}^{2}_{+}$, see also \cite{ACS2}.   We deduce that the  coset $\Sp_{n}/\U_{p}\times\U_{q}\times\Sp_{n-p-q}$ is spin if and only if $p$ and $q$ are both odd, independently of $n\geq 3$.}
  \end{example}

  \begin{example}\label{SO2NFLAGS} \textnormal{({\bf Spin flag manifolds of $\D_{n}$ with $b_{2}=1, 2$})}

\textnormal{(a) Consider the painted Dynkin diagram
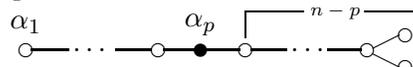
\[
\begin{picture}(390,-10)(20, 3)
 %\put(0, 0.5){\circle*{5}}
%\put(0, 5){\line(0,3){10}}
%\put(0, 15){\line(1,0){25}}
%\put(33, 15){\makebox(0,0){{\tiny{$n_{0}$}}}}
%\put(41, 15){\line(1,0){26}}
%\put(67, 15){\line(0,-3){10}}
 % \put(2, 0.5){\line(1,0){14}}
%\put(18, 0.5){\circle*{5}}
% \put(20, 0.5){\line(1,0){14}}
%\put(32, 0){ $\ldots$}
%\put(51, 0.5){\line(1,0){14}}
 %\put(67, 0.5){\circle*{5}}
  %\put(67,7){\makebox(0,0){$\al_{n_{0}}$}}
%\put(68, 0.5){\line(1,0){13}}
%\put(83, 0.5){\circle{5}}
%\put(83, 5){\line(0,3){10}}
%\put(83, 15){\line(1,0){11}}
%\put(107, 15){\makebox(0,0){{\tiny{$n_{1}-1$}}}}
%\put(119, 15){\line(1,0){12}}
%\put(131, 15){\line(0,-3){10}}
%\put(89,6){\makebox(0,0){$\al_{k+1}$}}
 %\put(85, 0.5){\line(1,0){14}}
%\put(97, 0){ $\ldots$}
%\put(115, 0.5){\line(1,0){14}}
%\put(131, 0.5){\circle{5}}
%\put(133, 0.5){\line(1,0){14}}
%\put(149, 0.5){\circle*{5}}
 %\put(149,10){\makebox(0,0){$\be_{n_{1}}$}}
  %\put(150.9, 0.5){\line(1,0){13}}
  \put(166, 0.5){\circle{5}}
  \put(166,10){\makebox(0,0){$\al_{1}$}}
    \put(168, 0.5){\line(1,0){14}}
     \put(180, 0){ $\ldots$}
  \put(200, 0.5){\line(1,0){13.6}}
   \put(216, 0.5){\circle{5}}
  %  \put(216,6){\makebox(0,0){$\al_{l+m+n-1}$}}
    \put(218, 0.5){\line(1,0){13.6}}
\put(232, 0.5){\circle*{5}}
 \put(232,10){\makebox(0,0){$\al_{p}$}}
  \put(233.9, 0.5){\line(1,0){13}}
  \put(249, 0.5){\circle{5}}
   \put(249, 5){\line(0,3){10}}
\put(249, 15){\line(1,0){22}}
\put(283, 15){\makebox(0,0){{\tiny{$n-p$}}}}
\put(294, 15){\line(1,0){20}}
\put(314, 15){\line(0,-3){20}}
  \put(251, 0.5){\line(1,0){13.6}}
  \put(262.5, 0){ $\ldots$}
  \put(280, 0.5){\line(1,0){13}}
    \put(295, 0.5){\circle{5}}
\put(297.3, 1){\line(2,1){10}}
\put(297.3, -1){\line(2,-1){10}}
%\put(115, 3){\line(-1,1){3}}
%\put(115, -3){\line(1,1){3}}
\put(309.5, 6){\circle{5}}
\put(309.5, -6){\circle{5}}
%\put(123.5, 14){\makebox(0,0){$\al_{\ell-1}$}}
%\put(118, -16){$\al_\ell$}
         \end{picture}
\]
\vskip 0.2cm
\noindent with $2\leq p\leq n-2$. It determines the unique flag manifold of $\SO_{2n}$ with two isotropy summands, namely $\D(\vec{n})=\SO_{2n}/\U_{p}\times \SO_{2(n-p)}$.  It is $n_{0}=0$, $s=1$, $n_{1}=p=n_{s}$ and $r=n-p$. Hence, the  Koszul form with respect to the unique invariant complex structure $J_{0}$  is $\sigma^{J_{0}}=(n_{s}+2r-1)\Lambda_{p}=(2n-p-1)\Lambda_{p}$. Consequently, $D(\vec{n})$ is $G$-spin if and only if  $2n-p-1$ is even, which is equivalent to say that $p$ is odd. There are some special cases,  which are not included in the above painted Dynkin diagram and all of them correspond to a  Hermitian symmetric space. In particular, if $\Pi_{B}=\{\al_1\}$, then we get the isotropy irreducible space $\SO_{2n}/\SO_{2}\times\SO_{n-2}$; in our notation it is $n_{0}=0$, $s=1$, $n_{1}=1=n_{s}$, $r=n-1$ and hence $\sigma^{J_{0}}=(n_{s}+2r-1)\Lambda_{1}=2(n-1)\Lambda_{1}$.  Thus, $\SO_{2n}/\SO_{2}\times\SO_{n-2}$ is spin for any $n$, see also \cite[Thm.~8]{Cahen}. On the other hand, painting black one of the end right roots (say $\al_{n}$) we get the isotropy irreducible Hermitian symmetric space $\SO_{2n}/\U_{n}$. Then, $n_{0}=0$, $s=1$, $n_{1}=n=n_{s}$, $r=0$ and the Koszul vector changes, i.e.  $\sigma^{J_{0}}=2(n_{s}-1)\Lambda_{n}=2(n-2)\Lambda_{n}$, see Corollary \ref{CLASFLAGS}.  Hence,  the coset $\SO_{2n}/\U_{n}$ is spin for any $n\geq 3$, see \cite[Thm.~8]{Cahen}.  }

     \textnormal{(b)
  Examine  flag manifolds $F=G/H$ of $G=\SO_{2n}$ with $b_{2}=2$, we start with the space $F=\SO_{2n}/\U_{1}\times\U_{n-1}$ $(n\geq 4)$, defined by
     the painted Dynkin diagram \[
   \begin{picture}(700,0)(0, -2)
 \put(232, 0.5){\circle{5}}
 %\put(232,7){\makebox(0,0){$\al_{1}$}}
  \put(233.9, 0.5){\line(1,0){13}}
  \put(249, 0.5){\circle{5}}
  % \put(249,7){\makebox(0,0){$\al_{2}$}}
  \put(251, 0.5){\line(1,0){13.6}}
  \put(262.5, 0){ $\ldots$}
  \put(280, 0.5){\line(1,0){13}}
    \put(295, 0.5){\circle{5}}
\put(297.5, 1){\line(2,1){10}}
\put(297.2, -1){\line(2,-1){10}}
%\put(115, 3){\line(-1,1){3}}
%\put(115, -3){\line(1,1){3}}
\put(310, 6){\circle*{5}}
\put(326,7){\makebox(0,0){$\al_{n-1}$}}
\put(310, -6){\circle*{5}}
\put(321,-7){\makebox(0,0){$\al_{n}$}}
%\put(123.5, 14){\makebox(0,0){$\al_{\ell-1}$}}
%\put(118, -16){$\al_\ell$}
         \end{picture}
 \]
 Since $\Hgt(\al_1)=\Hgt(\al_{n-1})=\Hgt(\al_{n})=1$, the same space occurs by setting $\Pi_{M}=\{\al_1, \al_{n-1}\}$ or $\Pi_{M}=\{\al_1, \al_{n}\}$. By \cite{Kim} it is known that $d=\sharp(R_{T}^{+})=3$.  In terms of Corollary \ref{CLASFLAGS}, it is $n_{0}=0=r$,   $n_{1}=1=n_{s-1}$ and $n_{s}=n-1$ with $s=2$.  Since $r=0$ and the end simple root is black, we get $\vec{k}=(n_{s-1}+n_{s}=n, 2(n_{s}-1)=2(n-2))\in\bb{Z}_{+}^{2}$ and $F$ is spin if and only if $n\geq 4$ is even (cf.  \cite[Thm.~3.2]{Kim}).
     Let us explain now  the painted Dynkin diagram 
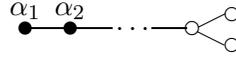
\[
   \begin{picture}(700,0)(0, -2)
 \put(232, 0.5){\circle*{5}}
 \put(232,7){\makebox(0,0){$\al_{1}$}}
  \put(233.9, 0.5){\line(1,0){13}}
  \put(249, 0.5){\circle*{5}}
   \put(249,7){\makebox(0,0){$\al_{2}$}}
  \put(251, 0.5){\line(1,0){13.6}}
  \put(262.5, 0){ $\ldots$}
  \put(280, 0.5){\line(1,0){13}}
    \put(295, 0.5){\circle{5}}
\put(297.5, 1){\line(2,1){10}}
\put(297.2, -1){\line(2,-1){10}}
%\put(115, 3){\line(-1,1){3}}
%\put(115, -3){\line(1,1){3}}
\put(310, 6){\circle{5}}
\put(310, -6){\circle{5}}
%\put(123.5, 14){\makebox(0,0){$\al_{\ell-1}$}}
%\put(118, -16){$\al_\ell$}
         \end{picture}
 \]
The associated flag manifold $\D(\vec{n})$ is given by $\SO_{2n}/\U_{1}\times\U_{1}\times\SO_{2(n-2)}$ $(n\geq 4)$ with $\rnk R_{T}=2$ and  $d=\sharp(R_{T}^{+})=4$ \cite{Chry2}. In terms of Corollary \ref{CLASFLAGS}, it is $n_{0}=1$, $s=1$, $n_{1}=1=n_{s}$ and $r=n-2$. Thus the Koszul vector is given by $\vec{k}=(1+n_{1}, n_{s}+2r-1)=(2, 2(n-2))\in\bb{Z}_{+}^{2}$ (see also \cite[Thm.~3]{Chry2}).
%\[
%\vec{k}=(1+n_{1}, n_{s}+2r-1)=(2, 2(n-2))\in\bb{Z}_{+}^{2}
%\]
 We deduce that $\D(\vec{n})$ is $G$-spin for any $n$.}

  \textnormal{(c) Consider the flag manifold $\SO_{2n}/
  \U_{p}\times
  \U_{n-p}$ with $2\leq p\leq n-2$ and $n\geq 4$. The painted Dynkin diagram is given by  
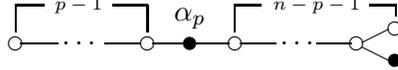
\[
   \begin{picture}(500,0)(-10, 3)
    \put(166, 0.5){\circle{5}}
     \put(166, 5){\line(0,3){10}}
\put(166, 15){\line(1,0){11}}
\put(190, 15){\makebox(0,0){{\tiny{$p-1$}}}}
\put(203, 15){\line(1,0){13}}
\put(216, 15){\line(0,-3){10}}
    \put(168, 0.5){\line(1,0){14}}
     \put(180, 0){ $\ldots$}
  \put(200, 0.5){\line(1,0){13.6}}
   \put(216, 0.5){\circle{5}}
  %  \put(216,6){\makebox(0,0){$\al_{l+m+n-1}$}}
    \put(218, 0.5){\line(1,0){13.6}}
\put(232, 0.5){\circle*{5}}
 \put(232,10){\makebox(0,0){$\al_{p}$}}
  \put(233.9, 0.5){\line(1,0){13}}
  \put(249, 0.5){\circle{5}}
   \put(249, 5){\line(0,3){10}}
\put(249, 15){\line(1,0){10}}
\put(280, 15){\makebox(0,0){{\tiny{$n-p-1$}}}}
\put(300, 15){\line(1,0){10}}
\put(310, 15){\line(0,-3){4}}
  \put(251, 0.5){\line(1,0){13.6}}
  \put(262.5, 0){ $\ldots$}
  \put(280, 0.5){\line(1,0){13}}
    \put(295, 0.5){\circle{5}}
\put(297.3, 1){\line(2,1){10}}
\put(297.3, -1){\line(2,-1){10}}
%\put(115, 3){\line(-1,1){3}}
%\put(115, -3){\line(1,1){3}}
\put(309.5, 6){\circle{5}}
\put(309.5, -6){\circle*{5}}
%\put(123.5, 14){\makebox(0,0){$\al_{\ell-1}$}}
%\put(118, -16){$\al_\ell$}
         \end{picture}
  \]
  \vskip 0.2cm
  \noindent which  induces the second flag manifold of $\D_{n}$ with $\rnk R_{T}=2$ and $d=4$ \cite{Chry2}. Notice that the same occurs if we paint black the simple root $\al_{n-1}$ (since $\Hgt(\al_{n})=\Hgt(\al_{n-1})=1$). It is $n_{0}=0$, $s=2$, $n_{1}=p=n_{s-1}$, $n_{2}=n-p=n_{s}$ and $r=0$.  Observing that the end simple root is black, by Corollary  \ref{CLASFLAGS} it follows that $  \vec{k}=(n_{1}+n_{2}, 2(n_{s}-1))=(n, 2(n-p-1))\in\bb{Z}_{+}^{2}$.    %  (see also \cite[Thm.~3]{Chry2})
  Thus, the coset $\SO_{2n}/\U_{p}\times
  \U_{n-p}$ is $G$-spin if and only if $n\geq 4$ is even.}

  \textnormal{(d) We describe now  the painted Dynkin diagram
  
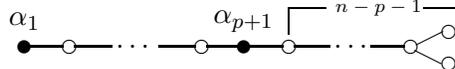
\[
 \begin{picture}(420,-10)(20, 3)
 %\put(0, 0.5){\circle*{5}}
%\put(0, 5){\line(0,3){10}}
%\put(0, 15){\line(1,0){25}}
%\put(33, 15){\makebox(0,0){{\tiny{$n_{0}$}}}}
%\put(41, 15){\line(1,0){26}}
%\put(67, 15){\line(0,-3){10}}
 % \put(2, 0.5){\line(1,0){14}}
%\put(18, 0.5){\circle*{5}}
% \put(20, 0.5){\line(1,0){14}}
%\put(32, 0){ $\ldots$}
%\put(51, 0.5){\line(1,0){14}}
 %\put(67, 0.5){\circle*{5}}
  %\put(67,7){\makebox(0,0){$\al_{n_{0}}$}}
%\put(68, 0.5){\line(1,0){13}}
%\put(83, 0.5){\circle{5}}
%\put(83, 5){\line(0,3){10}}
%\put(83, 15){\line(1,0){11}}
%\put(107, 15){\makebox(0,0){{\tiny{$n_{1}-1$}}}}
%\put(119, 15){\line(1,0){12}}
%\put(131, 15){\line(0,-3){10}}
%\put(89,6){\makebox(0,0){$\al_{k+1}$}}
 %\put(85, 0.5){\line(1,0){14}}
%\put(97, 0){ $\ldots$}
%\put(115, 0.5){\line(1,0){14}}
%\put(131, 0.5){\circle{5}}
%\put(133, 0.5){\line(1,0){14}}
\put(149, 0.5){\circle*{5}}
  \put(149,10){\makebox(0,0){$\al_{1}$}}
   \put(150.9, 0.5){\line(1,0){13}}
  \put(166, 0.5){\circle{5}}
 % \put(166,10){\makebox(0,0){$\al_{1}$}}
    \put(168, 0.5){\line(1,0){14}}
     \put(180, 0){ $\ldots$}
  \put(200, 0.5){\line(1,0){13.6}}
   \put(216, 0.5){\circle{5}}
  %  \put(216,6){\makebox(0,0){$\al_{l+m+n-1}$}}
    \put(218, 0.5){\line(1,0){13.6}}
\put(232, 0.5){\circle*{5}}
 \put(232,10){\makebox(0,0){$\al_{p+1}$}}
  \put(233.9, 0.5){\line(1,0){13}}
  \put(249, 0.5){\circle{5}}
   \put(249, 5){\line(0,3){10}}
\put(249, 15){\line(1,0){14}}
\put(283, 15){\makebox(0,0){{\tiny{$n-p-1$}}}}
\put(300, 15){\line(1,0){14}}
\put(314, 15){\line(0,-3){20}}
  \put(251, 0.5){\line(1,0){13.6}}
  \put(262.5, 0){ $\ldots$}
  \put(280, 0.5){\line(1,0){13}}
    \put(295, 0.5){\circle{5}}
\put(297.3, 1){\line(2,1){10}}
\put(296.8, -1){\line(2,-1){10}}
%\put(115, 3){\line(-1,1){3}}
%\put(115, -3){\line(1,1){3}}
\put(309.5, 6){\circle{5}}
\put(309.5, -6){\circle{5}}
%\put(123.5, 14){\makebox(0,0){$\al_{\ell-1}$}}
%\put(118, -16){$\al_\ell$}
         \end{picture}
  \]
  \vskip 0.2cm
 \noindent with $2\leq p\leq n-3$. It corresponds to the flag manifold $\D(\vec{n})=\SO_{2n}/\U_{1}\times\U_{p}\times
  \SO_{2(n-p-1)}$ with $\rnk R_{T}=2$ and $d=\sharp(R_{T}^{+})=5$ \cite{ACS1}. In terms of Corollary \ref{CLASFLAGS} it is $n_{0}=1$, $s=1$, $n_{1}=p=n_{s}$, and $r=n-p-1$. Hence, the Koszul vector is given by  $ \vec{k}=(1+n_{1}, n_{s}+2r-1)=(1+p, 2n-p-3)\in\bb{Z}_{+}^{2}$ (see also \cite[Exm.~5.2]{ACS1}),
 and  the homogeneous space $\D(\vec{n})$ is $G$-spin if and only if $p$ is odd.} %Notice that there are some special cases which are not included in the above painted Dynkin diagram.   These concern the case where $\al_{p+1}$ is one of the right end  simple roots, say  $p=n-1$ such that $\al_{p+1}=\al_{n}$. Then, as we proved in (b)  the associated flag manifold  $\SO_{2n}/\U_{1}\times\U_{n-1}$ is spin if and only $n\geq 4$ is even. }% , with $\rnk R_{T}=2$ and $d=3$ and is the flag manifold \cite{Kim}.  It is also $n_{0}=1$, $s=1$, $n_{1}=n-1=n_{s}$ and $r=0$. Since the end simple root is black, an short application of Corollary \ref{CLASFLAGS} shows  that   (see also \cite[Thm.~3.2]{Kim}):
% \[
%  \vec{k}=(1+n_{1}, 2(n_{s}-1))=(n, 2(n-2))\in\bb{Z}_{+}^{2}.
 % \]
%$ (see also \cite[Thm.~3.2]{Kim}).
  %\[
  %\vec{k}=(1+n_{1}, 2(n_{s}-1))=(n, 2(n-2))\in\bb{Z}_{+}^{2}.
  %\]
%  \[
%  \vec{k}=(1+n_{1}, 2(n_{s}-1))=(n, 2(n-2))\in\bb{Z}_{+}^{2}.
%  \]
%  Consequently, the flag manifold $\SO_{2n}/\U_{1}\times\U_{n-1}$ is $G$-spin if and only if $n$ is even.}

  \textnormal{(e) There is one more non-isomorphic  flag manifold of $\D_{n}$ with $b_{2}=2$, and this occurs  when both painted black roots have height two. This corresponds to the painted Dynkin diagram
  
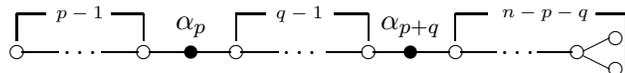
\[
  \begin{picture}(400,0)(-10, 6)
 \put(83, 0.5){\circle{5}}
\put(83, 5){\line(0,3){10}}
\put(83, 15){\line(1,0){11}}
\put(107, 15){\makebox(0,0){{\tiny{$p-1$}}}}
\put(119, 15){\line(1,0){12}}
\put(131, 15){\line(0,-3){10}}
%\put(89,6){\makebox(0,0){$\al_{k+1}$}}
 \put(85, 0.5){\line(1,0){14}}
\put(97, 0){ $\ldots$}
\put(115, 0.5){\line(1,0){14}}
\put(131, 0.5){\circle{5}}
\put(133, 0.5){\line(1,0){14}}
\put(149, 0.5){\circle*{5}}
 \put(149,10){\makebox(0,0){$\al_{p}$}}
  \put(150.9, 0.5){\line(1,0){13}}
  \put(166, 0.5){\circle{5}}
  \put(166, 5){\line(0,3){10}}
\put(166, 15){\line(1,0){11}}
\put(190, 15){\makebox(0,0){{\tiny{$q-1$}}}}
\put(203, 15){\line(1,0){13}}
\put(216, 15){\line(0,-3){10}}
    \put(168, 0.5){\line(1,0){14}}
     \put(180, 0){ $\ldots$}
  \put(200, 0.5){\line(1,0){13.6}}
   \put(216, 0.5){\circle{5}}
  %  \put(216,6){\makebox(0,0){$\al_{l+m+n-1}$}}
    \put(218, 0.5){\line(1,0){13.6}}
\put(232, 0.5){\circle*{5}}
 \put(232,10){\makebox(0,0){$\al_{p+q}$}}
  \put(233.9, 0.5){\line(1,0){13}}
  \put(249, 0.5){\circle{5}}
   \put(249, 5){\line(0,3){10}}
\put(249, 15){\line(1,0){12}}
\put(283, 15){\makebox(0,0){{\tiny{$n-p-q$}}}}
\put(302, 15){\line(1,0){12}}
\put(314, 15){\line(0,-3){20}}
  \put(251, 0.5){\line(1,0){13.6}}
  \put(262.5, 0){ $\ldots$}
  \put(280, 0.5){\line(1,0){13}}
    \put(295, 0.5){\circle{5}}
\put(297.3, 1){\line(2,1){10}}
\put(297.3, -1){\line(2,-1){10}}
%\put(115, 3){\line(-1,1){3}}
%\put(115, -3){\line(1,1){3}}
\put(309.5, 6){\circle{5}}
\put(309.5, -6){\circle{5}}
%\put(123.5, 14){\makebox(0,0){$\al_{\ell-1}$}}
%\put(118, -16){$\al_\ell$}
         \end{picture}
  \]
\vskip 0.3cm
\noindent  which induces the family $\SO_{2n}/\U_{p}\times\U_{q}\times\SO_{2(n-p-q)}$ with $n\geq 5$, $2\leq p\leq n-4$, $4\leq p+q\leq n-2$ such that $\Hgt(\al_{p})=\Hgt(\al_{p+q})=2$. We compute $d=6$. Moreover, it is $n_{0}=0$, $s=2$, $n_{1}=p=n_{s-1}$, $n_{2}=q=n_{s}$ and $r=n-p-q$. Since $p+q\leq n-2$,  the two simple roots in the end of the painted Dynkin diagram cannot be black (this was examined in case (c) above); hence  the associated  Koszul form is given by
$
\sigma^{J_{0}}=(p+q)\Lambda_{p}+(2n-2p-q-1)\Lambda_{p+q}.
$
  So, $\SO_{2n}/\U_{p}\times\U_{q}\times\SO_{2(n-p-q)}$ is spin if and only if both $p, q$ are odd. }
  \end{example}
Let us summarize the results of Examples \ref{cpn}, \ref{SONFLAGS}, \ref{SpFLAGS} and \ref{SO2NFLAGS}   In Table 1. %   we summarise  the results of Examples \ref{cpn}, \ref{SONFLAGS}, \ref{SpFLAGS} and \ref{SO2NFLAGS} and present the Koszul number/vector associated to the standard invariant complex structure $J_{0}$  of Lemma \ref{wellknown1}. We also state  the number $d:=\sharp(R_{T}^{+})$ of isotropy summands.

  %Notice that we    cite  \cite{AP} for the case of flag manifolds with $b_{2}=1$ and \cite{Cahen} for the case of Hermitian symmetric spaces. %Note that the restrictions on the parameters $p, q, n$ related with such flags, have been taken  with respect to the corresponding complex structure $J_{0}$. We summarise these conclusions in a theorem where we also cite  \cite{AP} for the case of flag manifolds with $b_{2}=1$ and \cite{Cahen} for the case of Hermitian symmetric spaces.
  \begin{theorem}\label{OURGOOD1} % \textnormal{(\cite{AP, Cahen})}
 % \label{b21}
  Let $G$ be a simple classical Lie group  and  $F=G/H=G^{\bb{C}}/P$ the associated flag manifold, such that $b_{2}(F)=1$ or $b_{2}(F)=2$.   Then, $F$ admits a (unique)   $G$-invariant   spin or metaplectic  structure, if and only $F$ is equivalent to one of the spaces appearing in Table 1 and satisfying the given conditions, if any.
\end{theorem}

\smallskip
 {\small{ \begin{center}
{\bf{Table 1.}}  Spin or metaplectic  classical  flag manifolds  with $b_{2}=1, 2$.
   \[
    \begin{tabular}{ l l c l l }
  $F=G/H$ with  $b_{2}(F)=1={\rm rnk}R_{T}$ &  conditions  &  $d$ & $k_{\al_{i_{o}}}\in\bb{Z}_{+}$  &     $G$-spin $(\Leftrightarrow)$    \\
 \thickline
 $\SU_{n}/\Ss(\U_{p}\times \U_{n-p})$ & $n\geq 2, 1\leq p\leq n-1$  & $1$ & $(n)$ &    $n$ even\\
%$\SO(2\ell+1)/\SO(2)\times \SO(2\ell-1)$ & $(2\ell-1)\Lambda_{1}$ & $2\ell-1$ & no\\
  $\Sp_{n}/\U_{n}$ & $n\geq 3$ & $1$ &  $(n+1)$   & $n$ odd \\
  $\SO_{2n}/\SO_{2}\times\SO_{2n-2}$ & $n\geq 4$ & $1$  & $(2n-2)$  &  $\forall \ n \geq 4$ \\
  $\SO_{2n}/\U_{n}$ & $n\geq 3$ & $1$ & $(2n-4)$ & $\forall \ n \geq 3$\\
   $\SO_{2n +1}/\U_{p}\times \SO_{2(n-p)+1}$ & $n\geq 2, 2\leq p\leq n$ & $2$  & $(2n-p)$   &   $p$ even $\geq 2$\\
   $\SO_{2n+1}/\U_{n}$ (special case) & $n\geq 2$ & $2$ & $(2n)$ & $\forall \ n \geq 2 $\\
  $\Sp_{n}/\U_{p}\times \Sp_{n-p}$ & $n\geq 3,1\leq p\leq n-1$&  $2$ & $(2n-p+1)$    &    $p$ odd $\geq 1$ \\
    $\Sp_{n}/\U_{1}\times\Sp_{n-1}=:\bb{C}P^{2n-1}$ & $n\geq 3$ & $2$ & $(2n)$ & $\forall \ n\geq 3$\\
   $\SO_{2n}/\U_{p}\times \SO_{2(n-p)}$ & $n\geq 4, 2\leq p\leq n-2$ & $2$ &  $(2n-p-1)$    & $p$ odd $\geq 2$ \\
\thickline \\
     $F=G/H$   $\text{with}$ $ b_{2}(F)=2={\rm rnk}R_{T}$ & $\text{conditions}$ &  $d$ & $\vec{k}\in\bb{Z}_{+}^{2}$ &     $G$-spin     \\
    \thickline
      $\SU_{n}/\U_{1}\times\Ss(\U_{p-1}\times\U_{n-p})$ & $n\geq 3, 2\leq p\leq n-2$ & $3$  & $(p, n-1)$ &  $n$ odd $\&$ $p$ even\\
       $\SU_{3}/{\rm T}^{2}$ (special case) & - & $3$ & $(2, 2)$ & yes  \\
    $\SU_{n}/\Ss(\U_{p}\times\U_{q}\times\U_{n-p-q})$ & $n\geq 6, 2\leq p\leq n-3$ & $3$ & $(p+q, n-p)$ &  $p, q, n$ same parity\\
    & $4\leq p+q\leq n-1$ & \\
   % $\SO_{2n}/\U_{1}\times\U_{n-1}$ & $n\geq 4$ & $3$ & $(n, 2(n-2))$ & $n$ even \\
    \hline
    $\SO_{5}/{\rm T}^{2}$ (special case) & - & $4$ & $(2, 2)$ & yes \\
    $\SO_{2n+1}/\U_{1}\times\U_{n-1}$ & $n\geq 3$ & $5$ & $(n, 2(n-1))$ & $n$ even\\
     $\SO_{2n+1}/\U_{p}\times\U_{n-p}$ & $n\geq 4, 2\leq p\leq n$ & $6$ & $(n, 2(n-p))$ & $n$ even\\
    $\SO_{2n+1}/\U_{p}\times\U_{q}\times\SO_{2(n-p-q)+1}$ & $n\geq 4, 2\leq p\leq n-1$ & $6$ & $(p+q, 2n-2p-q)$ & $p \ \& \ q$ even\\
    & $4\leq p+q\leq n-1$ & \\
       \hline
    $\Sp_{n}/\U_{p}\times\U_{n-p}$ & $n\geq 3, 1\leq p\leq n-1$ & $4$ & $(n, n-p+1)$ & $n$ even $\&$ $p$ odd\\
      $\Sp_{3}/{\rm T^{2}}$ (special case) & - & $4$  & $(2, 2)$  &  yes\\
     $\Sp_{n}/\U_{1}\times\U_{1}\times\Sp_{n-2}$ & $n\geq 3$ & $6$  & $(2, 2(n-1))$  &  $\forall \ n\geq 3$\\
    $\Sp_{n}/\U_{p}\times\U_{q}\times\Sp_{n-p-q}$ & $n\geq 3, 1\leq p\leq n-3$ & $6$ & $(p+q, 2n-2p-q+1)$ & $p \ \& \ q$ odd\\
    & $3\leq p+q\leq n-1$ & \\
    \hline
     $\SO_{2n}/\U_{1}\times\U_{n-1}$ & $n\geq 4$ & $3$ & $(n, 2(n-2))$ & $n$ even\\
    $\SO_{2n}/\U_{1}\times\U_{1}\times\SO_{2(n-2)}$ & $n\geq 4$ & $4$ & $(2, 2(n-2))$ &  $\forall \ n\geq 4$\\
    $\SO_{2n}/\U_{p}\times\U_{n-p}$ & $n\geq 4, 2\leq p\leq n-2$ & $4$ & $(n, 2(n-p-1))$ &  $n$ even\\
    $\SO_{2n}/\U_{1}\times\U_{p}\times\SO_{2(n-p-1)}$ & $n\geq 4, 2\leq p\leq n-3$ & $5$ & $(1+p, 2n-p-3)$ & $p$ odd\\
   $\SO_{2n}/\U_{p}\times\U_{q}\times\SO_{2(n-p-q)}$ & $n\geq 5, 2\leq p\leq n-4$ & $6$ & $(p+q, 2n-2p-q-1)$ &  $p \ \& \ q$ odd\\
   & $4\leq p+q\leq n-2$ &\\
    \thickline
\end{tabular}
\]
\end{center} }}

 \subsection{Invariant spin and metaplectic structures on  exceptional flag manifolds}
 Given an exceptional Lie group $G\in\{\G_2, \F_4, \E_6, \E_7, E_8\}$ with root system $R$ and a basis of simple roots $\Pi=\{\al_1, \ldots, \al_{\ell}\}$, we shall denote by $G(\al_{{j_{1}}}, \ldots, \al_{{j_{u}}})\equiv G({j_{1}}, \ldots, {j_{u}})$ the flag manifold $F=G/H$ where the semisimple part $\fr{h}'$ of the stability subalgebra $\fr{h}=T_{e}H$ corresponds to  the simple roots $\Pi_{W}:=\{\al_{{j_{1}}}, \ldots, \al_{{j_{u}}}\}$ for some $1\leq j_{1}<\ldots < j_{u}\leq \ell$. % i.e. $u=\sharp(\Pi_{W})=\rnk H_{ss}$. The remaining $v:=\ell-u$ nodes in the Dynkin diagram $\Gamma(\Pi)$  of $G$ have been painted black such that $\fr{h}=\fr{u}(1)\oplus\cdots\oplus\fr{u}(1)\oplus\fr{h}'$ with $v$ copies of $\fr{u}(1)$ generating the center of $\fr{h}$. For instance, $G(0)$ denotes a full flag manifold in this notation.
There are 101 non-isomorphic flag manifolds $F=G/H$ corresponding to a   simple exceptional Lie group $G$, see \cite{Bor, AA}.
In Tables 2, 3 and 4 we list them, together with their second Betti number, the Koszul form associated to the natural invariant ordering $R_{F}^{+}=R^{+}\backslash R_{H}^{+}$ and the number $d=\sharp(R^{+}_{T})$.
    For  details on the calculation of the Koszul form we refer to Remark \ref{exce}, see also \cite{Chry2}. %  \cite{Freud, AA, Chry2}.

 {\bf Notation in Tables 2, 3, 4.}
We  denote by ${\rm T}^{1}=\U_1$ the circle group and  by ${\rm T}^{\mu}$ the product ${\rm T}^{1}\times \cdots\times {\rm T}^{1}$ ($\mu$-times).  Also, $\U_{2}^{l}$ (resp. $\U_{2}^{s}$) denotes the Lie group (diffeomorphic to $\U_2\cong\SU_2\times\U_1=\A_1\times{\rm T}^{1}$) generated by the long (resp. short) simple root  of $\G_2$. Notice that these groups are not conjugate in $\G_2$. Similarly, for $\F_4$, we will denote by   $\A_{\nu}^{l}$ (resp. $\A_{\nu}^{s}$)  the non-conjugate    subgroups defined by the long (resp. short) simple root(s) of $\F_4$. Finally, $(\A_{\nu})^{\mu}$ means $\A_{\nu}\times\cdots\times\A_{\nu}$ ($\mu$-times) for some special unitary Lie group $\A_{\nu}=\SU_{\nu+1}$. For the notation $[0, 0]$, or $[1, 1]$, which we use in Table 3 with aim to emphasize on  non-isomorphic flag  manifolds of $\E_7$, we refer to \cite{Gr2}.

 \begin{comment}
\begin{example}\label{simexam}
\textnormal{Consider for example the Lie group $G=\E_7$. Then $E_{7}(2, 4, 5)$ corresponds to the following painted Dynkin diagram
 { \small{\[
    \begin{picture}(100,45)(-25,-28)
  %\put(-55,4){\makebox(0,0){ \textnormal{Class B} \quad  }}
   \put(-18, 0){\circle*{4}}
\put(-18,8.5){\makebox(0,0){$\alpha_1$}}
%\put(-18,-8){\makebox(0,0){1}}
\put(-16, 0){\line(1,0){14}}
\put(0, 0){\circle{4}}
\put(0,8.5){\makebox(0,0){$\alpha_2$}}
%\put(0,-8){\makebox(0,0){2}}
\put(2, 0){\line(1,0){14}}
\put(18, 0){\circle*{4}}
\put(18,8.5){\makebox(0,0){$\alpha_3$}}
%\put(18,-8){\makebox(0,0){3}}
\put(20, 0){\line(1,0){14}}
\put(36, 0){\circle{4}}
\put(36,8.5){\makebox(0,0){$\alpha_4$}}
%\put(32,-8){\makebox(0,0){4}}
\put(38, 0){\line(1,0){14}}
\put(54, 0){\circle{4}}
\put(54,8.5){\makebox(0,0){$\alpha_5$}}
%\put(54,-8){\makebox(0,0){3}}
\put(56, 0){\line(1,0){14}}
\put(72, 0){\circle*{4}}
\put(72,8.5){\makebox(0,0){$\alpha_6$}}
%\put(72,-8){\makebox(0,0){2}}
\put(36, -2){\line(0,-1){14}}
\put(36, -18){\circle*{4}}
\put(43,-25){\makebox(0,0){$\alpha_7$}}
%\put(33,-25){\makebox(0,0){2}}
\end{picture}
\]}}
This defines the flag manifold $F=\E_7/\SU_{3}\times\SU_{2}\times \U_{1}^{4}$, with $\Pi_{W}=\{\al_2, \al_4, \al_5\}$ and $\Pi_{B}=\{\al_1, \al_3, \al_6, \al_7\}$, respectively. Hence $\dim \fr{t}=4=\rnk R_{T}=b_{2}(F)$.}
 \end{example}
 \end{comment}

\medskip
 {\small{ \begin{center}
{\bf{Table 2.}} Cardinality of $T$-root system  and Koszul form of non-isomorphic    flag manifolds $(F=G/H, g, J)$  of  an exceptional   Lie group $G\in\{\G_2, \F_4, \E_6\}$.
  \end{center}  }}

    {\small{\begin{center}
  \begin{tabular}{ | c | l | l |  l | l |  }%c | }%|| c  | l | c | c | l  }
   \thickline & & & &  \\
   $G$ &  $F=G/H$ & $b_{2}(F)$ &    $d=\sharp(R^{+}_{T})$ &  $\sigma^{J}$ \\%& $G$-spin \\  %   $G$ &  $F=G/H$ & $b_{2}(F)$ & $q=|R^{+}_{T}|$ &  $\sigma^{J}$           \\
 \thickline
 $\G_2$ & $\G_{2}(0)=\G_2/{\rm T}^{2}$  &  2 &   6 & $2(\Lambda_{1}+\Lambda_{2})$ \\%& $\checkmark$ \\
 \hline
              & $\G_{2}(1)=\G_2/\U_2^{l}$ & 1 &    3  & $ 5\Lambda_{2}$  \\%& -   \\

             &  $\G_{2}(2)=\G_2/\U_{2}^{s}$ & 1 &   2  & $3\Lambda_{1}$ \\%& -   \\
             \thickline
 $\F_4$ &  $\F_{4}(0)=\F_{4}/{\rm T}^{4}$ & 4 &   24 & $2(\Lambda_{1}+\Lambda_{2}+\Lambda_{3}+\Lambda_{4})$ \\%  & $\checkmark$ \\
 \hline
             & $\F_{4}(1)=\F_{4}/\A_{1}^{l}\times{\rm T}^{3}$ & 3 &  16 & $3\Lambda_{2}+2(\Lambda_{3}+\Lambda_{4})$\\% & - \\
             &  $\F_{4}(4)=\F_{4}/\A_1^{s}\times{\rm T}^{3}$ & 3 &   13 & $2(\Lambda_{1}+\Lambda_{2})+\Lambda_{3}$\\% & - \\
             \hline
              & $\F_{4}(1, 2)=\F_4/\A_{2}^{l}\times{\rm T}^{2}$ & 2 &   9 &  $6\Lambda_{3}+2\Lambda_{4}$ \\%& $\checkmark$ \\
             & $\F_{4}(1, 4)=\F_4/\A_{1}\times\A_{1}\times{\rm T}^{2}$ & 2  & 8 & $3\Lambda_{2}+3\Lambda_{3}$ \\%& - \\
             & $\F_{4}(2, 3)=\F_4/\B_{2}\times{\rm T}^{2}$ & 2   & 6 & $5\Lambda_{1}+6\Lambda_{4}$\\% & - \\
             & $\F_{4}(3, 4)=\F_4/\A_{2}^{s}\times{\rm T}^{2}$ & 2   & 6 & $ 2\Lambda_{1}+4\Lambda_{2}$ \\%& $\checkmark$ \\
             \hline
             & $\F_{4}(1, 2, 4)=\F_4/\A_{2}^{l}\times\A_{1}^{s}\times{\rm T}^{1}$ & 1   & 4 & $7\Lambda_{3}$\\% & - \\
             & $\F_{4}(1, 3, 4)=\F_4/\A_{2}^{s}\times\A_{1}^{l}\times{\rm T}^{1}$ & 1   & 3 & $5\Lambda_{2}$\\% & - \\
              & $\F_{4}(1, 2, 3)=\F_4/\B_{3}\times{\rm T}^{1}$ & 1   & 2 & $11\Lambda_{4}$\\% & - \\
             & $\F_{4}(2, 3, 4)=\F_4/\Cc_{3}\times{\rm T}^{1}$ & 1  & 2 & $8\Lambda_{1}$\\% & $\checkmark$ \\
                         \thickline
$\E_6$  & $\E_6(0)=\E_6/{\rm T}^{6}$ & 6   &36 & $2(\Lambda_{1}+\cdots+\Lambda_{6})$ \\%& $\checkmark$ \\
\hline
              & $\E_{6}(1)=\E_6/\A_{1}\times{\rm T}^{5}$ & 5   &25 & $3\Lambda_{2}+2(\Lambda_{3}+\cdots+\Lambda_{6})$ \\% & - \\
              \hline
                 & $\E_6(3, 5)=\E_6/\A_{1}\times\A_{1}\times{\rm T}^{4}$ & 4   &17 & $2\Lambda_{1}+3\Lambda_{2}+4\Lambda_{4}+3\Lambda_{6}$ \\%& -\\

              & $\E_{6}(4, 5)=\E_6/\A_{2}\times{\rm T}^{4}$ & 4  &15  & $2(\Lambda_{1}+\Lambda_{2}+2\Lambda_{3}+\Lambda_{6})$ \\%& $\checkmark$ \\
              \hline
              & $\E_{6}(1, 3, 5)=\E_6/\A_1\times\A_1\times\A_1\times{\rm T}^{3}$ & 3   &11 & $4(\Lambda_{2}+\Lambda_{4})+3\Lambda_{6}$\\% &  - \\
              & $\E_{6}(2, 4, 5)=\E_6/\A_2\times\A_1\times{\rm T}^{3}$ & 3   & 10 & $3\Lambda_{1}+5\Lambda_{3}+2\Lambda_{6}$\\% & - \\
              & $\E_{6}(3, 4, 5)=\E_6/\A_{3}\times{\rm T}^{3}$ & 3   & 8 &  $2\Lambda_{1}+5(\Lambda_{2}+\Lambda_{6})$\\% & - \\
              \hline
              & $\E_{6}(2, 3, 4, 5)=\E_6/\A_{4}\times{\rm T}^{2}$ & 2   & 4 & $6\Lambda_{1}+8\Lambda_{6}$ \\
              & $\E_{6}(1, 3, 4, 5)=\E_6/\A_{3}\times\A_{1}\times{\rm T}^{2}$ & 2  & 5 & $6\Lambda_{2}+5\Lambda_{6}$  \\
                           & $\E_{6}(1, 2, 4, 5)=\E_6/\A_{2}\times\A_{2}\times{\rm T}^{2}$ & 2   & 6 & $6\Lambda_{3}+2\Lambda_{6}$  \\
              & $\E_{6}(2, 4, 5, 6)=\E_6/\A_{2}\times\A_{1}\times\A_{1}\times{\rm T}^{2}$ & 2   & 6  & $3\Lambda_{1}+6\Lambda_{3}$\\
              & $\E_{6}(2, 3, 4, 6)=\E_{6}/\D_{4}\times{\rm T}^{2}$ & 2  & 3 & $8(\Lambda_{1}+\Lambda_{5})$  \\
              \hline
              & $\E_{6}(1, 2, 4, 5, 6)=\E_6/\A_{2}\times\A_{2}\times\A_{1}\times{\rm T}^{1}$ & 1   & 3 & $7\Lambda_{3}$\\% & - \\
              & $\E_{6}(1, 2, 3, 4, 5)=\E_6/\A_{5}\times{\rm T}^{1}$ & 1   & 2 & $11\Lambda_{6}$\\% & - \\
              & $\E_{6}(1, 3, 4, 5, 6)=\E_6/\A_{1}\times\A_{1}\times{\rm T}^{1}$ & 1   & 2 & $9\Lambda_{2}$ \\%& - \\
              & $\E_{6}(2, 3, 4, 5, 6)=\E_6/\D_{5}\times{\rm T}^{1}$ & 1  & 1  &    $12\Lambda_{1}$\\% & $\checkmark$ \\
           \thickline
           \end{tabular}
\end{center}}}

\medskip
    Due to Proposition \ref{chernclass2},  Corollary \ref{intChern}, and  the results in Table   2,  we conclude that
\begin{theorem}\label{g2}
 (1) For $G=\G_2$ there is a unique $G$-spin (or $G$-metaplectic) flag manifold, namely the full flag $\G_2(0)=\G_2/{\rm T}^{2}$.  \\
  (2) For $G=\F_4$ the associated $G$-spin (of $G$-metaplectic) flag manifolds are the cosets defined by $\F_{4}(0)$,  $\F_{4}(1, 2)$, $ \F_{4}(3, 4)$, $\F_{4}(2, 3, 4)$, and the flag manifolds isomorphic to them.   In particular, $\F_{4}(2, 3, 4)=\F_4/\Cc_{3}\times{\rm T}^{1}$ is the unique (up to equivalence) flag manifold   of $G=\F_4$ with $b_2(F)=1=\rnk R_{T}$ which admits a   $G$-invariant spin and metaplectic structure. Moreover,  there are not exist flag manifolds $F=G/H$ of $G=F_4$ with $b_{2}(F)=3=\rnk R_{T}$ carrying a ($G$-invariant) spin structure or a metaplectic structure.\\
 (3) For $G=\E_6$  the associated $G$-spin (or $G$-metaplectic) flag manifolds are the cosets defined by $\E_{6}(0)$, $\E_{6}(4, 5)$, $\E_{6}(2, 3, 4, 5)$, $\E_{6}(1, 2, 4, 5)$, $\E_{6}(2, 3, 4, 6)$, $\E_{6}(2, 3, 4, 5, 6)$, and the flag manifolds isomorphic to them. In particular,
   $\E_{6}(4, 5)=\E_6/\A_{2}\times{\rm T}^{4}$ is the unique (up to equivalence) flag manifold  of $G=\E_6$ with $b_2(F)=4=\rnk R_{T}$ which admits a  $G$-invariant spin    and metaplectic structure. Moreover,
  $\E_{6}(2, 3, 4, 5, 6)=\E_6/\D_{5}\times{\rm T}^{1}$ is the unique (up to equivalence) flag manifold  of $G=\E_6$ with $b_2(F)=1=\rnk R_{T}$ which admits a  $G$-invariant spin   and metaplectic structure and
there are not exist flag manifolds $F=G/H$ of $G=\E_6$ with $b_{2}(F)=3=\rnk R_{T}$ carrying a ($G$-invariant) spin   or metaplectic structure.
 \end{theorem}

We proceed now with flag manifolds associated to the Lie group $\E_7$.

\smallskip
           {\small{ \begin{center}
{\bf{Table 3.}} Cardinality of $T$-root system  and Koszul form of non-isomorphic    flag manifolds $(F=G/H, g, J)$  of  the exceptional  Lie group $G=\E_7$
  \end{center}  }}

    {\small{\begin{center}
  \begin{tabular}{ | c | l | l |  l | l |  }%c | }%|| c  | l | c | c | l  }
   \thickline & & & &  \\
   $G$ &  $F=G/H$ & $b_{2}(F)$ &    $d=\sharp(R^{+}_{T})$ &  $\sigma^{J}$ \\%& $G$-spin \\  %   $G$ &  $F=G/H$ & $b_{2}(F)$ & $q=|R^{+}_{T}|$ &  $\sigma^{J}$           \\
 \thickline

$\E_7$  & $\E_7(0)=\E_7/{\rm T}^{7}$ & 7   & 63 & $2(\Lambda_{1}+\cdots+\Lambda_{7})$\\% & $\checkmark$ \\
\hline
             & $\E_7(1)=\E_7/\A_{1}\times{\rm T}^{6}$ & 6  & 46 & $3\Lambda_{2}+2(\Lambda_{3}+\cdots+\Lambda_{7})$ \\
             \hline
             & $\E_{7}(4, 6)=\E_7/\A_1\times\A_1\times{\rm T}^{5}$ & 5   & 33 & $2(\Lambda_{1}+\Lambda_{2})+3(\Lambda_3+\Lambda_7)+4\Lambda_{5}$\\
             & $\E_7(5, 6)=\E_7/\A_{2}\times{\rm T}^{5}$ & 5  & 30 & $2(\Lambda_{1}+\cdots+\Lambda_{4}+\Lambda_{7})$ \\
             \hline
             & $\E_7(1, 3, 5)=\E_7/\A_{1}\times\A_{1}\times\A_{1}\times{\rm T}^{4}$ \ $[1, 1]$ & 4 & 23 & $4\Lambda_{2}+4\Lambda_{4}+3\Lambda_{6}+2\Lambda_{7}$ \\
             & $\E_7(1, 3, 7)=\E_7/\A_{1}\times\A_{1}\times\A_{1}\times{\rm T}^{4}$ \ $[0, 0]$ & 4 & {24}   & $4\Lambda_{2}+4\Lambda_{4}+2\Lambda_{5}+2\Lambda_{6}$\\
             & $\E_7(3, 5, 6)=\E_7/\A_2\times\A_1\times{\rm T}^{4}$ & 4 & 21 & $2\Lambda_{1}+3\Lambda_{2}+5\Lambda_{4}+2\Lambda_{7}$\\
             & $\E_7(4, 5, 6)=\E_7/\A_3\times {\rm T}^{4}$ & 4 & 18 & $2\Lambda_{1}+2\Lambda_{2}+5\Lambda_{3}+5\Lambda_{7}$ \\
             \hline
             & $\E_7(1, 2, 3, 4)=\E_7/\A_4\times{\rm T}^{3}$ & 3 & 10 & $6\Lambda_{5}+2\Lambda_{6}+6\Lambda_{7}$ \\
             & $\E_7(1, 2, 3, 5)=\E_7/\A_3\times\A_1\times{\rm T}^{3}$ \ $[1, 1]$ & 3 & 12 & $6\Lambda_{4}+3\Lambda_{6}+2\Lambda_{7}$ \\
             & $\E_7(1, 2, 3, 7)=\E_7/\A_3\times\A_1\times{\rm T}^{3}$ \ $[0, 0] $ & 3 & {13}   & $6\Lambda_{4}+2\Lambda_{5}+2\Lambda_{6}$ \\
             & $\E_{7}(1, 2, 4, 5)=\E_7/\A_2\times\A_2\times{\rm T}^{3}$ & 3 & 13 & $6\Lambda_{3}+4\Lambda_{6}+4\Lambda_{7}$ \\
             & $\E_{7}(1, 2, 4, 6)=\E_7/\A_2\times \A_1\times \A_1\times{\rm T}^{3}$ & 3 & 14 & $5\Lambda_{3}+4\Lambda_{5}+3\Lambda_{7}$ \\
             & $\E_7(1, 3, 5, 7)=\E_7/(\A_{1})^{4}\times{\rm T}^{3}$ & 3 & 16 & $4\Lambda_{2}+5\Lambda_{4}+3\Lambda_{6}$ \\
             & $\E_7(3, 4, 5, 7)=\E_7/\D_{4}\times{\rm T}^{3}$ & 3 & 9 & $2\Lambda_{1}+8\Lambda_{2}+8\Lambda_{6}$ \\
             \hline
             & $\E_{7}(1, 2, 3, 4, 5)=\E_7/\A_{5}\times{\rm T}^{2}$ \ $[1, 1]$ & 2 & 5 & $7\Lambda_{6}+10\Lambda_{7}$ \\
             & $\E_{7}(1, 2, 3, 4, 7)=\E_7/\A_{5}\times{\rm T}^{2}$ \ $[0, 0]$ & 2 &   6  & $10\Lambda_{5}+2\Lambda_{6}$ \\
             & $\E_{7}(1, 2, 3, 4, 6)=\E_7/\A_{4}\times\A_{1}\times{\rm T}^{2}$ & 2 & 6 & $7\Lambda_{5}+6\Lambda_{7}$ \\
             & $\E_{7}(1, 2, 3, 5, 6)=\E_7/\A_{3}\times\A_{2}\times{\rm T}^{2}$ & 2 & 7 &  $7\Lambda_{4}+2\Lambda_{7}$\\
             & $\E_{7}(1, 2, 3, 5, 7)=\E_7/\A_{3}\times\A_{1}\times\A_{1}\times{\rm T}^{2}$ & 2 & 8 & $7\Lambda_{4}+3\Lambda_{6}$ \\
             & $\E_7(1, 3 , 4, 5, 7)=\E_7/\D_{4}\times\A_{1}\times{\rm T}^{2}$ & 2 & 6 & $9\Lambda_{2}+4\Lambda_{6}$\\
             & $\E_7(1, 2, 5, 6, 7)=\E_7/\A_{2}\times\A_{1}\times\A_{1}\times{\rm T}^{2}$ & 2 & 8 & $4\Lambda_{3}+5\Lambda_{4}$\\
             & $\E_7(1, 3, 5, 6, 7)=\E_7/\A_{2}\times(\A_{1})^{3}\times{\rm T}^{2}$ & 2 & 9  & $4\Lambda_{2}+6\Lambda_{4}$ \\
             & $\E_{7}(3, 4, 5, 6, 7)=\E_7/\D_{5}\times{\rm T}^{2}$ & 2 & 4 & $2\Lambda_{1}+12\Lambda_{2}$\\
             \hline
             & $\E_{7}(1, 2, 3, 4, 5, 6)=\E_7/\A_{6}\times{\rm T}^{1}$ & 1   & 2 & $14\Lambda_{7}$ \\
             & $\E_7(2, 3, 4, 5, 6, 7)=\E_7/\E_6\times{\rm T}^{1}$ & 1   & 1  & $18\Lambda_{1}$ \\
             & $\E_7(1, 3, 4, 5, 6, 7)=\E_7/\D_{5}\times\A_{1}\times{\rm T}^{1}$ & 1  & 2 & $13\Lambda_{2}$ \\
             & $\E_7(1, 2, 4, 5, 6, 7)=\E_7/\A_{4}\times\A_{2}\times{\rm T}^{1}$ & 1   & 3 & $10\Lambda_{3}$ \\
             & $\E_7(1, 2, 3, 5, 6, 7)=\E_7/\A_{3}\times\A_{2}\times\A_{1}\times{\rm T}^{1}$ & 1   & 4 & $8\Lambda_{4}$ \\
             & $\E_7(1, 2, 3, 4, 6, 7)=\E_7/\A_{5}\times\A_{1}\times{\rm T}^{1}$ & 1  & 2 & $12\Lambda_{5}$ \\
             & $\E_7(1, 2, 3, 4, 5, 7)=\E_7/\D_{6}\times{\rm T}^{1}$ & 1   & 2 & $17\Lambda_{6}$\\
                                        \thickline
\end{tabular}
\end{center}}}

\medskip
Using Table 3 and Proposition \ref{chernclass2}, Corollary \ref{intChern}  we conclude that
 \begin{theorem}\label{e7-8}
 For the Lie group  $G=\E_7$ the associated $G$-spin (or $G$-metaplectic) flag manifolds are the cosets defined by $\E_7(0)$,  $\E_7(5, 6)$, $\E_7(1, 3, 7)$, $\E_7(1, 2, 3, 4)$, $\E_7(1, 2, 3, 7)$, $\E_{7}(1, 2, 4, 5)$, $\E_7(3, 4, 5, 7)$, $\E_{7}(1, 2, 3, 4, 7)$,  $\E_7(1, 3, 5, 6, 7)$, $\E_{7}(3, 4, 5, 6, 7)$, $\E_{7}(1, 2, 3, 4, 5, 6)$, $\E_7(2, 3, 4, 5, 6, 7)$, $\E_7(1, 2, 4, 5, 6, 7)$, $\E_7(1, 2, 3, 5, 6, 7)$, $\E_7(1, 2, 3, 4, 6, 7)$ and the flag manifolds isomorphic to them. In particular, $\E_7(5, 6)=\E_7/\A_{2}\times{\rm T}^{5}$ is the unique (up to equivalence)  flag manifold  of $G=\E_7$ with second Betti number $b_{2}(F)=5=\rnk R_{T}$, which admits a  $G$-invariant  spin   and   metaplectic structure.
  Moreover, the space $\E_7(1, 3, 7)=\E_7/\A_{1}\times\A_{1}\times\A_{1}\times{\rm T}^{4}$ is the unique (up to equivalence)  flag manifold  of $G=\E_7$ with second Betti number $b_{2}(F)=4=\rnk R_{T}$, which admits a  $G$-invariant  spin   and metaplectic structure and there are not exist flag manifolds $F=G/H$ of $G=\E_7$ with $b_{2}(F)=\rnk R_{T}=6$, carrying a ($G$-invariant) spin or metaplectic structure.          \end{theorem}

  Let us finally present the results for the Lie group $\E_8$

  \smallskip

  {\small{ \begin{center}
{\bf{Table 4.}} Cardinality of $T$-root system  and Koszul form of non-isomorphic    flag manifolds $(F=G/H, g, J)$  of  the exceptional   Lie group $G=\E_8$.
  \end{center}  }}

   {\small{\begin{center}
  \begin{tabular}{ | c | l | l |  l | l |  }%c | }%|| c  | l | c | c | l  }
   \thickline & & & &   \\
   $G$ &  $F=G/H$ & $b_{2}(F)$ &   $d=\sharp(R^{+}_{T})$ &  $\sigma^{J}$ \\%& $G$-spin \\  %   $G$ &  $F=G/H$ & $b_{2}(F)$ & $q=|R^{+}_{T}|$ &  $\sigma^{J}$           \\
 \thickline
 $\E_8$  & $\E_8(0)=\E_8/{\rm T}^{8}$ & 8   & 120  & $2(\Lambda_{1}+\cdots+\Lambda_{8})$\\% & $\checkmark$ \
 \hline
              & $\E_{8}(1)=\E_8/\A_{1}\times{\rm T}^{7}$ & 7 & 91  & $3\Lambda_{2}+2(\Lambda_{3}+\cdots+\Lambda_{8})$ \\
              \hline
              & $\E_{8}(1, 2)=\E_8/\A_{2}\times{\rm T}^{6}$ & 6 & 63 & $ 4\Lambda_{3}+2(\Lambda_{4}+\cdots+\Lambda_{8})$\\
              & $\E_{8}(1, 3)=\E_8/\A_{1}\times\A_{1}\times{\rm T}^{6}$ & 6 & 68 & $4\Lambda_{2}+3\Lambda_{4}+2(\Lambda_{5}+\cdots+\Lambda_{8})$ \\
              \hline
             & $\E_{8}(1, 2, 3)=\E_8/\A_3\times{\rm T}^{5}$ & 5 & 41 & $5\Lambda_{4}+2(\Lambda_{5}+\cdots+\Lambda_{8})$\\
             & $\E_{8}(1, 2, 4)=\E_8/\A_{2}\times\A_{1}\times{\rm T}^{5}$ & 5 & 46 & $5\Lambda_{3}+3\Lambda_{5}+2\Lambda_{6}+2\Lambda_{7}+2\Lambda_{8}$ \\
             & $\E_{8}(1, 3, 5)=\E_8/(\A_{1})^{3}\times{\rm T}^{5}$ & 5 & {50} & $4\Lambda_{2}+4\Lambda_{4}+3\Lambda_{6}+2\Lambda_{7}+3\Lambda_{8}$\\
             \hline
             & $\E_8(1, 2, 3, 4)=\E_8/\A_4\times{\rm T}^{4}$ & 4 & 25 & $6\Lambda_{5}+2(\Lambda_{6}+\Lambda_{7}+\Lambda_{8})$ \\
             & $\E_8(1, 2, 3, 5)=\E_8/\A_3\times\A_1\times{\rm T}^{4}$ & 4 & 29 & $6\Lambda_{4}+3\Lambda_{6}+2\Lambda_{7}+3\Lambda_{8}$ \\
             & $\E_{8}(1, 2, 4, 5)=\E_8/\A_2\times\A_2\times{\rm T}^{4}$ & 4 & 30 & $6\Lambda_{3}+4\Lambda_{6}+2\Lambda_{7}+4\Lambda_{8}$ \\
             & $\E_{8}(1, 2, 4, 6)=\E_8/\A_2\times\A_{1}\times\A_{1}\times{\rm T}^{4}$ & 4 & 33 & $5\Lambda_{3}+4\Lambda_{5}+3\Lambda_{7}+2\Lambda_{8}$ \\
             & $\E_8(1, 3, 5, 7)=\E_8/(\A_{1})^{4}\times{\rm T}^{4}$ & 4 & 36 & $2(\Lambda_{2}+\Lambda_{4}+\Lambda_{6})+3\Lambda_{8}$\\
             & $\E_8(4, 5, 6, 8)=\E_8/\D_{4}\times{\rm T}^{4}$ & 4 & 24 & $2(\Lambda_{1}+\Lambda_{2})+8(\Lambda_{3}+\Lambda_{7})$ \\
             \hline
             & $\E_8(1, 2, 3, 4, 5)=\E_8/\A_{5}\times{\rm T}^{3}$ & 3 & 14 & $7\Lambda_{6}+2\Lambda_{7}+7\Lambda_{8}$ \\
             & $\E_8(1, 2, 3, 4, 6)=\E_8/\A_{4}\times\A_{1}\times{\rm T}^{3}$ & 3 & 17 & $7\Lambda_{5}+3\Lambda_{7}+2\Lambda_{8}$ \\
             & $\E_8(1, 2, 3, 5, 6)=\E_8/\A_{3}\times\A_{2}\times{\rm T}^{3}$ & 3 & 18 & $7\Lambda_{4}+4\Lambda_{7}+4\Lambda_{8}$ \\
             & $\E_8(1, 2, 3, 5, 7)=\E_8/\A_3\times\A_1\times\A_1\times{\rm T}^{3}$ & 3 & {20} & $6\Lambda_{4}+2\Lambda_{6}+3\Lambda_{8}$ \\
             & $\E_8(1, 2, 4, 5, 7)=\E_8/\A_2\times\A_2\times\A_1\times{\rm T}^{3}$ & 3 & {21}  & $6\Lambda_{3}+5\Lambda_{6}+4\Lambda_{8}$ \\
             & $\E_8(1, 2, 4, 6, 8)=\E_8/\A_2\times(\A_{1})^{3}\times{\rm T}^{3}$ & 3 & {23} & $5\Lambda_{3}+5\Lambda_{5}+3\Lambda_{7}$ \\
             & $\E_8(1, 4, 5, 6, 8)=\E_8/\D_{4}\times\A_1\times{\rm T}^{3}$ & 3 & 16 &     $3\Lambda_{2}+8\Lambda_{3}+8\Lambda_{7}$ \\
             & $\E_8(4, 5, 6, 7, 8)=\E_8/\D_{5}\times{\rm T}^{3}$ & 3 & 13 & $2(\Lambda_{1}+\Lambda_{2}+\Lambda_{3})$ \\
             \hline
             & $\E_8(1, 2, 3, 4, 5, 6)=\E_8/\A_6\times{\rm T}^{2}$ & 2 & 7 & $8\Lambda_{7}+12\Lambda_{8}$\\
             & $\E_8(1, 2, 3, 4, 5, 7)=\E_8/\A_5\times\A_1\times{\rm T}^{2}$ & 2 & 9 & $8\Lambda_{6}+7\Lambda_{8}$\\
             & $\E_8(1, 2, 3, 4, 6, 7)=\E_8/\A_4\times\A_2\times{\rm T}^{2}$ & 2 & 10 & $8\Lambda_{5}+2\Lambda_{8}$\\
             & $\E_8(1, 2, 3, 5, 6, 7)=\E_8/\A_3\times\A_3\times{\rm T}^{2}$ & 2 & 10 & $8\Lambda_{4}+5\Lambda_{8}$\\
             & $\E_8(1, 2, 3, 4, 6, 8)=\E_8/\A_4\times\A_1\times\A_1\times{\rm T}^{2}$ & 2 & 11 & $8\Lambda_{5}+3\Lambda_{7}$ \\
             & $\E_8(1, 2, 4, 5, 6, 8)=\E_8/\D_4\times\A_2\times{\rm T}^{2}$ & 2 & 9 & $10\Lambda_{3}+8\Lambda_{7}$\\
             & $\E_8(1, 4, 5, 6, 7 ,8)=\E_8/\D_5\times\A_1\times{\rm T}^{2}$ & 2 & 8 & $3\Lambda_{2}+12\Lambda_{3}$ \\
             & $\E_8(2, 3, 4, 5, 6, 8)=\E_8/\D_6\times{\rm T}^{2}$ & 2 & 6 & $12\Lambda_{1}+17\Lambda_{7}$\\
             & $\E_8(1, 2, 3, 6, 7, 8)=\E_8/\A_3\times\A_2\times\A_1\times{\rm T}^{2}$ & 2 & 12 & $5\Lambda_{4}+5\Lambda_{5}$\\
             & $\E_8(1, 2, 4, 6, 7, 8)=\E_8/\A_2\times\A_{2}\times\A_{1}\times\A_1\times{\rm T}^{2}$ & 2 & 14 & $5\Lambda_{3}+6\Lambda_{5}$\\
             & $\E_8(3, 4, 5, 6, 7, 8)=\E_8/\E_6\times{\rm T}^{2}$ & 2 & 6 & $2\Lambda_{1}+18\Lambda_{2}$\\
            \hline
             & $\E_8(1, 2, 3, 4, 5, 6, 7)=\E_8/\A_{7}\times{\rm T}^{1}$ & 1   & 3 & $17\Lambda_{8}$ \\
             & $\E_8(2, 3, 4, 5, 6, 7, 8)=\E_8/\E_7\times{\rm T}^{1}$ & 1   & 2 & $29\Lambda_{1}$ \\
             & $\E_8(1, 3, 4, 5, 6, 7, 8)=\E_8/\E_6\times\A_{1}\times{\rm T}^{1}$ & 1   & 3 & $19\Lambda_{2}$ \\
             & $\E_8(1, 2, 4, 5, 6, 7, 8)=\E_8/\D_{5}\times\A_{2}\times{\rm T}^{1}$ & 1   & 4 & $14\Lambda_{3}$ \\
             & $\E_8(1, 2, 3, 5, 6, 7, 8)=\E_8/\A_{4}\times\A_{3}\times{\rm T}^{1}$ & 1  & 5 & $11\Lambda_{4}$ \\
             & $\E_8(1, 2, 3, 4, 6, 7, 8)=\E_8/\A_{4}\times\A_{2}\times\A_{1}\times{\rm T}^{1}$ & 1   & 6 & $9\Lambda_{5}$\\
             & $\E_8(1, 2, 3, 4, 5, 7, 8)=\E_8/\A_6\times\A_1\times{\rm T}^{1}$ & 1  & 4 & $13\Lambda_{6}$ \\
             & $\E_8(1, 2, 3, 4, 5, 6, 8)=\E_8/\D_{7}\times{\rm T}^{1}$ & 1   & 2 & $23\Lambda_{7}$\\
                      \thickline
\end{tabular}
\end{center}}}

\medskip
Using the results  in Table 4  in combination with Proposition \ref{chernclass2} and  Corollary \ref{intChern} we get that
 \begin{theorem}\label{e8}
  For $G=\E_8$ the associated $G$-spin (or $G$-metaplectic)  flag manifolds are the cosets defined by $\E_8(0)$, $\E_{8}(1, 2)$, $\E_8(1, 2, 3, 4)$, $\E_{8}(1, 2, 4, 5)$, $\E_8(4, 5, 6, 8)$, $\E_8(4, 5, 6, 7, 8)$, $\E_8(1, 2, 3, 4, 5, 6)$, $\E_8(1, 2, 3, 4, 6, 7)$, $\E_8(1, 2, 4, 5, 6, 8)$, $\E_8(1, 2, 4, 5, 6, 7, 8)$ and the flag manifolds isomorphic to them. In particular, $\E_{8}(1, 2)=\E_8/\A_{1}\times{\rm T}^{6}$ is the unique (up to equivalence)  flag manifold  of $G=\E_8$ with second Betti number $b_{2}(F)=6=\rnk R_{T}$, which admits a  $G$-invariant  spin  and metaplectic structure.  Moreover, $\E_8(4, 5, 6, 7, 8)=\E_8/\D_{5}\times{\rm T}^{3}$  is the unique (up to equivalence)  flag manifold   of $G=\E_8$ with second Betti number $b_{2}(F)=3=\rnk R_{T}$, which admits a   $G$-invariant spin  and metaplectic structure. Similarly, $\E_8(1, 2, 4, 5, 6, 7, 8)=\E_8/\D_{5}\times\A_{2}\times{\rm T}^{1}$   is the unique (up to equivalence)  flag manifold   of $G=\E_8$ with second Betti number $b_{2}(F)=1=\rnk R_{T}$ which admits a  $G$-invariant spin  and metaplectic structure. Finally, there are not exist flag manifolds $F=G/H$ of $G=\E_8$ with $b_{2}(F)=\rnk R_{T}=5$, or  $b_{2}(F)=\rnk R_{T}=7$, carrying a ($G$-invariant) spin or metaplectic structure.
 \end{theorem}

\section{C-spaces and invariant spin structures}\label{Cspace}
\subsection{Topology of C-spaces} A natural  generalization of   flag manifolds  are    C{\it-spaces}, i.e.  compact simply-connected  homogeneous complex manifolds
  $M = G/L$  of    a   compact  semisimple Lie  group  $G$.   In this case, the stability subgroup  $L$ is a closed connected subgroup of $G$ whose semisimple part coincides with the semisimple part of the centralizer of a torus in $G$.
    According  to  H. C. Wang \cite{Wang},  any C-space is the  total  space  of  a  homogeneous
   bundle $ M = G/L \to  F = G/H$ over a flag manifold $F=G/H$ with structure group  a complex torus ${\rm T}^{2k}$  of real even dimension $2k:=\rnk G-\rnk L$. Thus, $F$ has  $b_{2}(F):=\sharp(\Pi_{B})=v\geq 2$, i.e. $F=G/H=G/H'\cdot {\rm T}^{v}$ $(v\geq 2)$.

  Let $\mathfrak{g}= \mathfrak{h}+ \mathfrak{m}=  (Z(\mathfrak{h})+ \mathfrak{h}')+\mathfrak{m}$
  be   reductive  decomposition   of  the flag manifold $F=G/H$. A C-space  is  defined
   by  a  decomposition of the space $\fr{t}=iZ(\fr{h})$  into a  direct  sum     of a
 (commutative)  subalgebra  $\mathfrak{t}_1$  of even  dimension  $2k$, generating the torus ${\rm T}_{1}^{2k}\equiv {\rm T}^{2k}$ and   a complementary subalgebra  $\mathfrak{t}_0$,   which generates   a   torus  ${\rm T}_0$ of  $H$  with  $\dim_{\bb{R}} {\rm T}_{0}:=m=v-2k$, that is
  \[
 \fr{t}:=i Z(\mathfrak{h}) = \mathfrak{t}_0 + \mathfrak{t}_1.
   \]
   Then, $\rnk G=\rnk H=\dim {\rm T}^{v}+\rnk H'$, $\rnk L=\dim {\rm T}_{0}+\rnk H'$,   $H'$ (the semisimple part   of $H$) coincides  with the simisimple part of $L$  and the homogeneous manifold
 $M =  G/L := G/ T_0 \cdot H'$ is    a C-space. In particular,     any  C-space  can be obtained by this construction (see  \cite{Wang}). The reductive  decomposition of $M=G/L$ is given by
 \begin{equation}\label{redc}
   \mathfrak{g}= \mathfrak{l}+ \mathfrak{q}= (\mathfrak{h}'+ i\mathfrak{t}_0)+(   i\mathfrak{t}_1 +  \mathfrak{m}), \quad \fr{q}:=(i\mathfrak{t}_1+\mathfrak{m})\cong T_{eL}M.
\end{equation}

  Changing $G$ to its universal covering,   we  may assume (and we do) that $G$ is  simply-connected    and  moreover that the  action of $G$ on $M =G/L$ is      effective.
Notice that  any  extension of an invariant complex     $J_F$ on $F$ (which corresponds to an $\Ad_{H}$-invariant endomorphism  $J_{\mathfrak{m}} : \mathfrak{m}\to\fr{m}$ with $J^{2}_{\fr{m}}=-\Id$),    induces  an invariant  complex  structure  $J_{M}$ on $M$, such that  the projection  $\pi : M \to F$ is holomorphic. This is given by
 $J_{\mathfrak{q}} := J_{\mathfrak{t}_1}+  J_{\mathfrak{m}}$, for some complex structure $J_{\fr{t}_{1}}$ on $\fr{t}_{1}$. For later use, notice  that   the complex  structure   $J_M$  has   the  same  Koszul   form as   $J_F$ : $\sigma_{J_M} = \sigma_{J_F} = \sum_{R_F}\alpha$.

 \begin{lemma} \label{dmitri}
 The  space $\Omega^2_{cl}(M)^{G}$ of  closed invariant  2-forms on a C-space $M=G/L$ is  isomorphic   to
 $C_{\mathfrak{\mathfrak{g}}}(\mathfrak{l}) = i \mathfrak{t} + \mathfrak{q}_0$, where $\mathfrak{q}_0:= C_{\mathfrak{q}}(L)$,   and  consists of  the  form
  $\omega_{\xi} = \frac{i}{2 \pi} d\xi,\, \xi  \in \mathfrak{t}^* + i\mathfrak{q}_0^*  $.  The   form  $\omega_{\xi}$ on  $M$ is  exact  if and only if
  $\xi \in   i \mathfrak{t}_1 + \mathfrak{q}_0$.  In particular,  the   cohomology group  $H^2(M, \mathbb{R}) \cong \mathfrak{t}_0 \cong  Z(\mathfrak{l}) $.
 \end{lemma}

 \begin{proof}
Since  $H^2(\mathfrak{g}, \mathbb{R})=0$, any closed   2-form $\rho \in \Lambda^2(\mathfrak{g}^*) $ is   exact, $\rho = d \xi.$
 Such a form  defines   a  closed  form  $\rho = d \xi \in \Omega^2_{cl}(M)$  if and only if $\xi  \in    C_{\mathfrak{g}}(L) =  \mathfrak{t} + \mathfrak{q}_0$, where  $\mathfrak{q}_0:= C_{\mathfrak{q}}(L)  $.
Moreover, the form  $\rho  = d \xi \in \Omega^2_{cl}(M)$   % associated  with  $\xi \in \mathfrak{t}$
is exact if and  only if  $\xi \in i\mathfrak{t}_1^* + \mathfrak{q}_0^*$,
since  any   such  form  defines  an invariant  1-form on  $M$.
 So  the  cohomology  $H^2(M,\mathbb{R})\cong \mathfrak{t}_0$.
\end{proof}

Next we shall frequently call   {\it M-space}  a  C-space  $M=G/L$ whose stability group $L$ is semisimple, $L=H'$.  This is equivalent to say that  $\fr{t}_{0}$ is trivial and hence in this case Lemma \ref{dmitri} yields
 \begin{corol} \label{mspaces}
A  M-space $M=G/H'$  has trivial second cohomology group, $H^{2}(M; \bb{R})=0$. % In particular, $M$ is spin (see also Corollary \ref{zeroc1}).
\end{corol}

 Since $\pi_{1}({\rm T}^{2k})\cong \bb{Z}^{2k}$ is the only  non-trivial homotopy group of the torus ${\rm T}^{2k}$, the principal bundle   ${\rm T}^{2k}\overset{\iota}{\hookrightarrow}M=G/L\overset{\pi}{\to} F=G/H$ induces an exact sequence
\[
\pi_{2}({\rm T}^{2k})=0\to\pi_{2}(M)\to \pi_{2}(F)\overset{\delta}{\to}\pi_{1}({\rm T}^{2k})\to \pi_{1}(M)\to\pi_{1}(F)\to 0,
\]
where $\delta : \pi_{s}(F)\to \pi_{s-1}({\rm T}^{2k})$ is a homomorphism of homotopy groups. By assumption, it is $b_{2}(F)=v$ and thus $\pi_{2}(F)\cong \bb{Z}^{v}$. Hence, if $k\neq 0$ (i.e. $\fr{t}_{1}\neq\{0\}$) then $\delta : \bb{Z}^{v}\to\bb{Z}^{2k}$ is non-zero.  In particular, $\pi_{2}(M)$ is  finite.  % $x\mapsto x \ {\rm mod} \bb{Z}^{m}$ for any $x\in\bb{Z}^{v}$.
%Since both $M$ and $F$ are simply-connected,  we get $0\to\pi_{2}(M)\to \bb{Z}^{v}\overset{\delta}{\to}\bb{Z}^{2k}\to 0$, hence,
% $\pi_{2}(M)=\bb{Z}^{v}/\bb{Z}^{2k}\cong \bb{Z}^{m}$.
 In the level of  homology groups we have an exact sequence
\[
H_{2}({\rm T}^{2k}; \bb{Z})\overset{\iota_{*}}{\to}H_{2}(M; \bb{Z})\overset{\pi_{*}}{\to} H_{2}(F; \bb{Z})\overset{\delta_{*}}{\to} H_{1}({\rm T}^{2k}; \bb{Z})\overset{\iota_{*}}{\to}H_{1}(M; \bb{Z})\to 0.
\]
Recall that a flag manifold  $F$ does not admit any $G$-invariant real 1-form. In particular, $b_{1}(F)=0$ and $H^{1}(F; \bb{Z})=0$. Thus, dualy we can write
\begin{equation}\label{seq}
0=H^{1}(F; \bb{Z})\to H^{1}(M; \bb{Z})\overset{\iota^{*}}{\to} H^{1}({\rm T}^{2k}; \bb{Z})\overset{\delta^{*}}{\to} H^{2}(F; \bb{Z})\overset{\pi^{*}}{\to} H^{2}(M; \bb{Z})\overset{\iota^{*}}{\to} H^{2}({\rm T}^{2k}; \bb{Z}).
\end{equation}
 In fact, (\ref{seq})  holds even if we consider the cohomology groups with  complex coefficients (see \cite{Ise, Hof}).

\begin{prop}\label{cspacespin}%\textnormal{(\cite{Ise, Hof})}
Consider a ${\rm T}_{1}^{2k}$-principal bundle  $M=G/L\overset{\pi}{\to} F=G/H$ of a non-K\"ahler C-space $M=G/L$ over   a flag manifold  $F= G/H $. Then\\
$(i)$   $H^{1}(M; \bb{Z})=0$, $H^{2}(M; \bb{Z})$ is torsion free and the pull back  $\iota^{*} : H^{2}(M; \bb{Z})\to H^{2}({\rm T}^{2k}; \bb{Z})$ is zero.\\
$(ii)$ $b_{1}(M)=0$. \\
%$(ii)$ $b_{1}(M)=0$.\\ %The pull-back  $\pi^{*} : H^{2}(F; \bb{C})\to H^{2}(M; \bb{C})$ is surjective  (even over $\bb{Z}$).\\%$(ii)$ {\color{red} $b_{1}(M)=0$???}\\
%$(iii)$  $H^{1}(M; \bb{C})\cong H^{1}({\rm T}^{2k}; \bb{C})$, in particular $M=G/L$ is K\"ahler, if and only if $H^{1}(M; \bb{C})=0$.\\
$(iii)$  The pull-back  $\pi^{*} : H^{2}(F; \bb{C})\to H^{2}(M; \bb{C})$ is surjective, $H^{2}(M; \bb{Z})\cong H^{2}(F; \bb{Z})/\delta^{*}\big(H^{1}(T^{2k}; \bb{Z})\big)$ and $b_{2}(M)=b_{2}(F)-b_{1}({\rm T}^{2k})=b_{2}(F)-2k=m=\dim_{\bb{R}}\fr{t}_{0}$.

%$b_{1}(M)=0$ and
%$(iv)$  The pull-back  $\pi^{*} : H^{2}(F; \bb{C})\to H^{2}(M; \bb{C})$ is surjective and the same holds if we consider the cohomology groups with integer coefficients. Moreover,
%\[
%H^{2}(M; \bb{Z})\cong H^{2}(F; \bb{Z})/\delta^{*}\big(H^{1}(T^{2k}; \bb{Z})\big)
%\]
 %and $b_{2}(M)=b_{2}(F)-b_{1}({\rm T}^{2k})=b_{2}(F)-2k=v-2k=m=\dim_{\bb{R}}\fr{t}_{0}$.
 \end{prop}
\begin{proof}
$(i)$  As long as $M$ is not K\"ahlerian, $G$-invariant 1-forms  exist, in contrast to $F$. However, since $M$ is a  (closed) simply-connected manifold, similarly with $F$ we immediately  conclude that   $H^{1}(M; \bb{R})=0$ (cf. \cite[Thm.~15.17]{Lee}) and $H^{1}(M; \bb{Z})=0$.   In particular,  the map $\delta_{*} : H_{2}(F; \bb{Z})\to H_{1}({\rm T}^{2k}; \bb{Z})$ coincides with $\pi_{2}(F)\to\pi_{1}({\rm T}^{2k})$ and so it must be surjective. %see   \cite[Prop.~11.5]{Hof}.
Then,  by the universal coefficient theorem it follows that  the transgression $\delta^{*} : H^{1}({\rm T}^{2k}; \bb{Z})\to H^{2}(F; \bb{Z})$ must be injective. Thus, by (\ref{seq}) we get $H^{1}(M; \bb{Z})=0$  and   $\iota^{*} : H^{2}(M; \bb{Z})\to H^{2}({\rm T}^{2k}; \bb{Z})$ is zero.   Also, since $H_{1}(M; \bb{Z})=0$,  the universal coefficient theorem implies that    $H^{2}(M; \bb{Z})$ is torsion free,
\[
H^{2}(M; \bb{Z})={\rm Hom}(H_{2}(M; \bb{Z}); \bb{Z})\oplus {\rm Ext}(H_{1}(M; \bb{Z});  \bb{Z})={\rm Hom}(H_{2}(M; \bb{Z}); \bb{Z}).
\]
 $(ii)$ %Although $M$ is a compact complex manifold and the vanishing of $H^{1}(M; \bb{Z})$ immediately  implies that $b_{1}(M)=0$,  let us describe a different proof, useful for $(ii)$.
Given of  a torus principal bundle ${\rm T}^{2k}{\hookrightarrow} M\to F$ over a complex manifold $F$,     the first Betti number  of the total space vanishes $b_{1}(M)=0$, if and only if $b_{1}(F)=0$ and $\delta^{*}_{\bb{C}} : H^{1}({\rm T}^{2k}; \bb{C})\to H^{2}(F; \bb{C})$ is injective, see   \cite[Prop.~11.3]{Hof}.   Here, it is  $H^{1}({\rm T}^{2k}; \bb{C})=H^{1}({\rm T}^{2k}; \bb{Z})\otimes\bb{C}$ and $ H^{2}(F; \bb{C})= H^{2}(F; \bb{Z})\otimes\bb{C}$.  But $\delta^{*}_{\bb{C}} : H^{1}({\rm T}^{2k}; \bb{C})\to H^{2}(F; \bb{C})$ is obtained by $\delta^{*}$ by scalar multiplication, i.e. $\delta^{*}_{\bb{C}}=\delta^{*}\otimes \Id_{\bb{C}}$ and our assertion follows.\\ % Thus,   $\delta^{*}_{\bb{C}}$ is   injective and since $b_{1}(F)=0$, it is also  $b_{1}(M)=0$. \\
$(iii)$ The first claim in $(iii)$ occurs by  the injectivity of  $\delta^{*}_{\bb{C}}$ and (\ref{seq}) (written with complex coefficients). Moreover, since $H^{1}(M; \bb{Z})=0$ and $\iota^{*} : H^{2}(M; \bb{Z})\to H^{2}({\rm T}^{2k}; \bb{Z})$ is the zero map,  (\ref{seq}) reduces to the following short exact sequence
\begin{equation}\label{seq3}
0\to H^{1}({\rm T}^{2k}; \bb{Z})\overset{\delta^{*}}{\to} H^{2}(F; \bb{Z})\overset{\pi^{*}}{\to} H^{2}(M; \bb{Z})\to 0.
\end{equation}
Thus, we infer  the isomorphism $H^{2}(M; \bb{Z})\cong H^{2}(F; \bb{Z})/\delta^{*}\big(H^{1}(T^{2k}; \bb{Z})\big)$. % so  $b_{2}(M)=b_{2}(F)-b_{1}({\rm T}^{2k})$.%by using $(i)$ and the exact sequence $0\to H^{1}({\rm T}^{2k}; \bb{Z})\overset{\delta^{*}}{\to} H^{2}(F; \bb{Z})\overset{\pi^{*}}{\to} H^{2}(M; \bb{Z})\to 0$.
\end{proof}
\begin{example}
\textnormal{The even dimensional simply-connected  simple Lie groups $G$  (or equivalently of even rank), can be viewed as M-spaces  over the associated full flag manifolds $G/{\rm T}^{2y}$, i.e. ${\rm T}^{2y}\hookrightarrow G\to G/{\rm T}^{2y}$ with $\ell=\rnk G=2y$. Thus,  Corollary \ref{mspaces} and Proposition \ref{cspacespin}, together they yield the well-known    $H^{1}(G; \bb{R})=H^{2}(G; \bb{R})=0$.}
\end{example}
Let us relate now  some  characteristic classes of the   C-space $M=G/L$ and  those of the  associated  flag manifold   $F=G/H$.
\begin{prop}\label{c1w2}
 The second Stiefel-Whitney classes and the first Chern classes of $M$ and $F$   are respectively related  by
\[
w_2(TM)=\pi^{*}(w_{2}(TF)), \quad c_{1}(M, J_{M})=\pi^{*}(c_{1}(F, J_{F})), \quad w_{2}(TM)=\pi^{*}(c_{1}(F)) \ {\rm mod} 2.
\]
\end{prop}
 \begin{proof}
  The reductive decomposition (\ref{redc}) induces a splitting of tangent bundle of $M=G/L$, given by
   \begin{equation}\label{TM}
TM=G\times_{L}\fr{q}=(G\times_{L}i\fr{t}_{1})\oplus \pi^{*}(TF)=\tau_{{\rm T}^{2k}}\oplus \pi^{*}(TF).
\end{equation}
  Since  any invariant  subbundle  of even rank over $M$ admits   an invariant  almost  complex  structure, it is oriented. Thus, the tangent bundle along the fibres $\tau_{{\rm T}^{2k}}=G\times_{L}i\fr{t}_{1}$ satisfies $w_{1}(\tau_{{\rm T}^{2k}})=0$.  In fact, the vector bundle $\tau_{{\rm T}^{2k}}\to M$ is trivial and hence $w_{1}(\tau_{{\rm T}^{2k}})=w_{2}(\tau_{{\rm T}^{2k}})=0$. This comes true since $\fr{l}\subset\fr{h}$ is an ideal of $\fr{h}$ and so the structure group $\Ad_{L}$ acts trivially on $\fr{h}/\fr{l}\cong\fr{t}_{1}$. % see also \cite[p.~246]{Ise}.  %Another argument for the vanishing of $w_{2}$  relies on the fact that  any homogeneous vector bundle over a compact complex parallelizable manifold, e.g. ${\rm T}^{2k}$, has trivial Chern classes.
Therefore, the  relation $w_2(TM)=\pi^{*}(w_{2}(TF))$ follows by (\ref{TM}) and the naturality of Stiefel-Whitney classes (see Proposition \ref{genw2}).  To prove the  relation between the first Chern classes $c_{1}(M, J_{M})$ and $c_{1}(F, J_{F})$,    recall   that   the first  Chern  class  of $(F, J_F)$ is    represented  by  the invariant  Chern   form  $\gamma_{J_F} = \omega_{\sigma^{J_F}}$ associated  with  the Koszul  form.  Since the pull-back $\pi^{*} : H^{2}(F; \bb{Z})\to H^{2}(M; \bb{Z})$ is surjective, by the definition of the first Chern class on $M$ and using (\ref{seq3}), we  get $c_{1}(M, J_{M})=\pi^{*}(c_{1}(F, J_{F}))$ (see also \cite{Hano, Hof}),
and   the relation $w_{2}(TM)=\pi^{*}(c_{1}(F)) \ {\rm mod} \  2$ follows. %Thus, the   pull back $\gamma_{J_M} = \pi^* \gamma_{J_F}$   of  the
\end{proof}
 \subsection{Spin structures on C-spaces}
Let us examine now spin structures on a $C$-space $M=G/L$. If $M$ admits a  spin structure  then it will be $G$-invariant  and  unique,  because  $M$ is oriented and  both $ M, G$  are simply-connected   (see Proposition \ref{springer}).  Moreover,    Proposition \ref{c1w2} implies that

% If $M$ admits a  spin structure  then it will be $G$-invariant  since   $M$ is oriented and  $G$ can be assumed to be simply-connected (see Proposition \ref{springer}). Moreover,  if such a structure exists  it will be unique, since $M$ is simply-connected. Also,  due to Proposition \ref{cspacespin} one easily concludes that
\begin{corol}\label{concl}
 (1) The C-space  $M$ is spin if and only if $w_{2}(TF)$ belongs to the kernel of $\pi^{*} : H^{2}(F; \bb{Z}_{2})\to H^{2}(M; \bb{Z}_{2})$.  Therefore, if $F$ is $G$-spin,   so is $M$. \\ % i.e. it admits a unique $G$-invariant spin structure.\\
 (2) \textnormal {(\cite{Hof})}  The first Chern class $c_{1}(M)$ of a C-space $(M, J_{M})$ vanishes if and only if the first Chern class $c_{1}(F)$  of $(F, J_{F})$ belongs to the image of the transgression $\delta^{*} : H^{1}({\rm T}^{2k}; \bb{Z})\to H^{2}(F; \bb{Z})$, i.e. $ c_{1}(M)\in{\rm Im} \ \delta^{*}$.
 \end{corol}

Based on Corollary  \ref{concl}, we shall use now the results of Section \ref{flags} to describe all C-spaces fibered over  spin flag manifolds of  exceptional Lie groups.\footnote{For $\E_7$,  we denote by $M\overset{*}{\to}F$ the fibrations associated to  $\E_7$-flag manifolds of type $[0, 0]$.}  % (and we adopt the same notation for the corresponding C-spaces).}
% One  can  easily construct similar  examples     for classical Lie groups (for more details we refer to \cite{Gran}).
 \begin{prop}\label{cprop}
 There are 45 non-biholomorphic C-spaces $M=G/L$ fibred over a spin flag manifold $F=G/H$ of an exceptional Lie group $G\in\{\G_2, \F_4, \E_6, \E_7, \E_8\}$, and any such space carries a unique $G$-invariant spin structure.   The associated fibrations are given as follows:
   \medskip
 %{\small{ \begin{center}
%{\bf{Table 1.}}  $G$-spin  classical  flag manifolds  with $b_{2}=1, 2$.
 \[
   \begin{tabular}{l l l l l || l l l l l}
   ${\rm T}^{2}$ & $\hookrightarrow$ & $\G_2$ & $\longrightarrow$ & $\G_2/{\rm T}^{2}$
 &  ${\rm T}^{6}$ & $\hookrightarrow$ & $\E_7/{\rm T}^{1}$ & $\longrightarrow$ & $\E_7/{\rm T}^{7}$\\
      ${\rm T}^{4}$ & $\hookrightarrow$ & $ \F_4$  &$\longrightarrow$& $\F_4/{\rm T}^{4}$
 &  ${\rm T}^{4}$  & $\hookrightarrow$ & $\E_7/{\rm T}^{3}$ & $\longrightarrow$ & $\E_7/{\rm T}^{7}$\\
         ${\rm T}^{2}$& $\hookrightarrow$ & $\F_4/ {\rm T}^{2}$&$\longrightarrow$& $\F_4/{\rm T}^{4}$
  &  ${\rm T}^{2}$ & $\hookrightarrow$ & $\E_7/{\rm T}^{5}$ & $\longrightarrow$ & $\E_7/{\rm T}^{7}$\\
  $ {\rm T}^{2}$&$\hookrightarrow$ &$\F_4/ \A_{2}^{l}$  &$\longrightarrow$& $\F_4/ \A_{2}^{l}\times{\rm T}^{2}$ &  ${\rm T}^{4}$ & $\hookrightarrow$ & $\E_7/\A_{2}\times{\rm T}^{1}$ & $\longrightarrow$ & $\E_7/\A_{2}\times{\rm T}^{5}$\\
  ${\rm T}^{2}$&$\hookrightarrow$ &$\F_4/ \A_{2}^{s}$ &$\longrightarrow$& $\F_4/ \A_{2}^{s}\times{\rm T}^{2}$ & ${\rm T}^{2}$ & $\hookrightarrow$ & $\E_7/\A_{2}\times{\rm T}^{3}$ & $\longrightarrow$ & $\E_7/\A_{2}\times{\rm T}^{5}$\\
   ${\rm T}^{6}$ & $\hookrightarrow$ & $\E_6$ & $\longrightarrow$ & $\E_6/{\rm T}^{6}$
   &     ${\rm T}^{4}$ & $\hookrightarrow$ & $\E_7/(\A_{1})^{3}$ & $\overset{*}{\longrightarrow}$ & $\E_7/(\A_{1})^{3}\times{\rm T}^{4}$   \\
  ${\rm T}^{4}$ & $\hookrightarrow$ & $\E_6/{\rm T}^{2}$ & $\longrightarrow$ & $\E_6/{\rm T}^{6}$
  &  ${\rm T}^{2}$ & $\hookrightarrow$ & $\E_7/(\A_{1})^{3}\times {\rm T}^{2}$ & $\overset{*}{\longrightarrow}$ & $\E_7/(\A_{1})^{3}\times{\rm T}^{4}$  \\
    ${\rm T}^{2}$ & $\hookrightarrow$ & $\E_6/{\rm T}^{4}$ & $\longrightarrow$ & $\E_6/{\rm T}^{6}$
    &  ${\rm T}^{2}$ & $\hookrightarrow$ & $\E_7/\A_{4}\times {\rm T}^{1}$ & $\longrightarrow$ & $\E_7/\A_{4}\times {\rm T}^{3}$\\
       ${\rm T}^{4}$ & $\hookrightarrow$ & $\E_6/\A_{2}$ & $\longrightarrow$ & $\E_6/\A_{2}\times{\rm T}^{4}$
       &  ${\rm T}^{2}$ & $\hookrightarrow$ & $\E_7/\A_{3}\times\A_{1}\times {\rm T}^{1}$ & $\overset{*}{\longrightarrow}$ & $\E_7/\A_{3}\times\A_{1}\times {\rm T}^{3}$   \\
      ${\rm T}^{2}$ & $\hookrightarrow$ & $\E_6/\A_{2}\times{\rm T}^{2}$ & $\longrightarrow$ & $\E_6/\A_{2}\times{\rm T}^{4}$
      & ${\rm T}^{2}$ & $\hookrightarrow$ & $\E_7/\A_{2}\times\A_{2}\times {\rm T}^{1}$ & $\longrightarrow$ & $\E_7/\A_{2}\times\A_{2}\times {\rm T}^{3}$\\
               ${\rm T}^{2}$ & $\hookrightarrow$ & $\E_6/\A_{4}$ & $\longrightarrow$ & $\E_6/\A_{4}\times{\rm T}^{2}$
    & ${\rm T}^{2}$ & $\hookrightarrow$ & $\E_7/\D_{4}\times {\rm T}^{1}$ & $\longrightarrow$ & $\E_7/\D_{4}\times {\rm T}^{3}$\\
            ${\rm T}^{2}$ & $\hookrightarrow$ & $\E_6/\A_{2}\times\A_{2}$ & $\longrightarrow$ & $\E_6/\A_{2}\times\A_{2}\times{\rm T}^{2}$
   & ${\rm T}^{2}$ & $\hookrightarrow$ & $\E_7/\A_{5}$ & $\longrightarrow$ & $\E_7/\A_{5}\times {\rm T}^{2}$\\
               ${\rm T}^{2}$ & $\hookrightarrow$ & $\E_6/\D_{4}$ & $\longrightarrow$ & $\E_6/\D_{4}\times{\rm T}^{3}$
 & ${\rm T}^{2}$ & $\hookrightarrow$ & $\E_7/\A_{2}\times(\A_{1})^{3}$ & $\overset{*}{\longrightarrow}$ & $\E_7/\A_{2}\times(\A_{1})^{3}\times {\rm T}^{2}$ \\
  &  & & & &
  ${\rm T}^{2}$ & $\hookrightarrow$ & $\E_7/\D_{5}$ & $\longrightarrow$ & $\E_7/\D_{5}\times{\rm T}^{2}$ \\
  \thickline
       ${\rm T}^{8}$ & $\hookrightarrow$ & $\E_8$ & $\longrightarrow$ & $\E_8/{\rm T}^{8}$ &  ${\rm T}^{4}$ & $\hookrightarrow$ & $\E_8/\D_{4}$ & $\longrightarrow$ & $\E_8/\D_{4}\times{\rm T}^{4}$  \\
       ${\rm T}^{6}$ & $\hookrightarrow$ & $\E_8/{\rm T}^{2}$ & $\longrightarrow$ & $\E_8/{\rm T}^{8}$ &  ${\rm T}^{2}$ & $\hookrightarrow$ & $\E_8/\D_{4}\times{\rm T}^{2}$ & $\longrightarrow$ & $\E_8/\D_{4}\times{\rm T}^{4}$  \\
         ${\rm T}^{4}$ & $\hookrightarrow$ & $\E_8/{\rm T}^{4}$ & $\longrightarrow$ & $\E_8/{\rm T}^{8}$ &
          ${\rm T}^{4}$ & $\hookrightarrow$ & $\E_8/\A_{2}\times\A_{2}$ & $\longrightarrow$ & $\E_8/\A_{2}\times\A_{2}\times{\rm T}^{4}$  \\
           ${\rm T}^{2}$ & $\hookrightarrow$ & $\E_8/{\rm T}^{6}$ & $\longrightarrow$ & $\E_8/{\rm T}^{8}$ &
            ${\rm T}^{2}$ & $\hookrightarrow$ & $\E_8/\A_{2}\times\A_{2}\times{\rm T}^{2}$ & $\longrightarrow$ & $\E_8/\A_{2}\times\A_{2}\times{\rm T}^{4}$  \\
  ${\rm T}^{6}$ & $\hookrightarrow$ & $\E_8/\A_{2}$ & $\longrightarrow$ & $\E_8/\A_{2}\times{\rm T}^{6}$ &
  ${\rm T}^{2}$ & $\hookrightarrow$ & $\E_8/\D_{5}\times{\rm T}^{1}$ & $\longrightarrow$ & $\E_8/\D_{5}\times{\rm T}^{3}$ \\
    ${\rm T}^{4}$ & $\hookrightarrow$ & $\E_8/\A_{2}\times{\rm T}^{2}$ & $\longrightarrow$ & $\E_8/\A_{2}\times{\rm T}^{6}$ &
  ${\rm T}^{2}$ & $\hookrightarrow$ & $\E_8/\A_{6}$ & $\longrightarrow$ & $\E_8/\A_{6}\times{\rm T}^{2}$ \\
      ${\rm T}^{2}$ & $\hookrightarrow$ & $\E_8/\A_{2}\times{\rm T}^{4}$ & $\longrightarrow$ & $\E_8/\A_{2}\times{\rm T}^{6}$ &  ${\rm T}^{2}$ & $\hookrightarrow$ & $\E_8/\A_{4}\times\A_{2}$ & $\longrightarrow$ & $\E_8/\A_{4}\times\A_{2}\times{\rm T}^{2}$ \\
        ${\rm T}^{4}$ & $\hookrightarrow$ & $\E_8/\A_{4}$ & $\longrightarrow$ & $\E_8/\A_{4}\times{\rm T}^{4}$ &
         ${\rm T}^{2}$ & $\hookrightarrow$ & $\E_8/\D_{4}\times\A_{2}$ & $\longrightarrow$ & $\E_8/\D_{4}\times\A_{2}\times{\rm T}^{2}$ \\
          ${\rm T}^{2}$ & $\hookrightarrow$ & $\E_8/\A_{4}\times{\rm T}^{2}$ & $\longrightarrow$ & $\E_8/\A_{4}\times{\rm T}^{4}$ &
           ${\rm T}^{2}$ & $\hookrightarrow$ & $\E_8/\E_{6}$ & $\longrightarrow$ & $\E_8/\E_{6}\times{\rm T}^{2}$
       \end{tabular}
       \]
 \end{prop}
%\begin{proof} The proof is based on our   classification of invariant spin structures on flag manifolds $F=G/H$ of an exceptional Lie group $G\in\{\G_2, \F_4, \E_6, \E_7, \E_8\}$, see Theorems \ref{g2}, \ref{e7-8} and \ref{e8}.      For $\E_7$  we denote by $M\overset{*}{\to}F$ the fibrations associated to  $\E_7$-flag manifolds of type $[0, 0]$ (and we adopt the same notation for the corresponding C-spaces).   \end{proof}

Let $M=G/{\rm T}_{0}^{m}\cdot H'$ be a C-space, and $\pi : M=G/{\rm T}_{0}^{m}\cdot H'\to F=G/{\rm T}^{v}\cdot H'$ the associated     principal ${\rm T}_{1}^{2k}$-bundle over  a flag manifold $F$, such that $v=b_{2}(F)=2k+m$.  Any such principal bundle is classified by $2k$   elements $\be_1, \ldots, \be_{2k}\in  \mathcal{P}_T  \cong H^{2}(F; \bb{Z})$  (see \cite{Wang2}), of  the  $T$-weight lattice $\mathcal{P}_T = \mathrm{span}_{\mathbb{Z}}\{\Lambda_1, \cdots, \Lambda_v\}$ (spanned by    the    black    fundamental  weights of $F$).  In fact, the generators in $H^{1}({\rm T}^{2k}; \bb{Z})$ transgress to $\be_{i}\in\cal{P}_{T}$ and each $\be_{i}$ can be thought as the Euler class of the orientable circle bundle $M/{\rm T}^{2k-1}\to  F$, where ${\rm T}^{2k-1}\subset{\rm T}^{2k}$ is the subtorus with the $i$th factor $\Ss^{1}$ removed.

% and  since the   transgression $\delta^{*} : H^{1}({\rm T}^{2k}; \bb{Z})\to H^{2}(F; \bb{Z})$ is injective, the generators in $H^{1}({\rm T}^{2k}; \bb{Z})$ transgress to $\be_{i}$, see also \cite[p.~223]{Wz2}. %Note also that  if $\theta : M\to \fr{t}_{1}$ is a principal connection with curvature $\Omega=d\theta=\pi^{*}\eta$ for some invariant  closed 2-form $\eta$ of $F$, we can write $\eta=\sum_{i=1}^{2k}\eta_{i}e_{i}$ where $\{e_{1}, \ldots, e_{2k}\}$ is  a basis of $\fr{t}_{1}\cong T_{e}{\rm T}_{1}^{2k}$, with
%$H^{2}(F; \bb{Z})\ni [\eta_{i}]=2\pi\be_{i}$.
 Denote  by    $\beta^{\vee} = \frac{2}{(\beta, \beta)} B^{-1}\beta   \in  \mathfrak{a}_0 $     the  coroot     associated   with  a simple  black  root  $\beta \in \Pi_B$  and  by  $\Gamma $   the   lattice   in   $\mathfrak{t}$,  spanned   by  the  projection    $t_{\beta}$  of    $\beta^{\vee}$ $(\beta  \in \Pi_B)$ onto  $\mathfrak{t} = iZ(\mathfrak{h})$ (see   Section \ref{bflags}  for  notation).
Since   $G$ is   simply-connected, the   central   subgroup   ${\rm T}^{v}$   of    $H \subset  G$ generated   by   the   subalgebra   $i \mathfrak{t}$, is  identified   with
${\rm T}^{v} = \mathfrak{t}/ 2 \pi i\Gamma$,  see  \cite[Ch.~3, 2.4]{GOV}.  Since by assumption   the    subgroup  ${\rm T}_1^{2k}$ is  closed, the  intersection $\Gamma_1 := \Gamma \cap \mathfrak{t}_1$ is   a  lattice (of full rank, recall  that  by  a lattice in  a  real  $n$-dimensional vector  space  $V$, we always  mean  a     discrete    subgroup  of  the  vector  group $V$  of rank $n$),  and   we  get   %  $T_1^{2k}  = \mathfrak{t}_1/2 \pi \Gamma_1$.

\begin{lemma}
There are  isomorphisms  ${\rm T}^{2k}  \cong  \mathfrak{t}_1/ \Gamma_1$
and   $H^1(\mathrm{T}^{2k},\mathbb{Z}) = \Gamma_{1}^*$.
\end{lemma}
 \begin{proof}
 We only remark that  we can identify $\pi_{1}({\rm T}^{2k})=H_{1}({\rm T}^{2k}; \bb{Z})=\Gamma_1$. Then,  the universal coefficient theorem induces the desired  isomorphism, $H^{1}({\rm T}^{2k}; \bb{Z})={\rm Hom}(\pi_{1}({\rm T}^{2k}), \bb{Z})={\rm Hom}(\Gamma_1, \bb{Z})=\Gamma_{1}^{*}$.  \end{proof}

  %Since $\fr{t}_{0}\subset\fr{t}$ generates a closed toral subgroup ${\rm T}_{0}\subset {\rm T}^{v}$,  $\cal{P}_{1}$ is a sublattice of $\cal{P}_{T}$  of rank $v-m=2k$.
%It is easy to see that $\cal{P}_{1}$ is an integer sublattice of $\cal{P}_{T}$,  given by $\cal{P}_{1}=\cal{P}_{T}\cap i\fr{t}_{1}^{*}$.
% Since $\cal{P}_{1}$  is integer, we have $\cal{P}_{1}\subset\cal{P}_{1}^{*}$ where $\mathcal{P}_1^* = \{  t \in \mathfrak{t}_1 :  \varphi(t) \in  \mathbb{Z}, \   \forall  \varphi \in \mathcal{P}_1 \}$.
 %If $\fr{t}_{1}$ is trivial, then $\cal{P}_{1}$ is trivial.   Below we will show that  $\cal{P}_{1}$ is a full rank lattice in  $\bb{R}^{2k}$, which is equivalent to say that  $\fr{t}^{*}_{1}$ contains $2k$ linear independent elements of $\cal{P}_{T}\subset i\fr{t}^{*}$. For  example, for a M-space, it is $\cal{P}_{1}=\cal{P}_{T}$. % (recall that $T$-weight lattice $\cal{P}_{T}=\cal{P}\cap  i\fr{t}^{*}$ is a full lattice in $\fr{t}^{*}\cong \bb{R}^{v}$).

Let us denote by $\cal{P}_{1}\subset\cal{P}_{T}$ the annihilator  of $\fr{t}_0$  in   $\cal{P}_T$,
\[
\cal{P}_{1}:=\{\lambda\in\cal{P}_{T} : \lambda|_{\fr{t}_{0}} = 0\}=\cal{P}_{T}\cap \fr{t}_{1}^{*}.
\]
 If $\fr{t}_{1}$ is trivial, then the same is $\cal{P}_{1}$.   Below we will show that  $\cal{P}_{1}$ is a (full rank) lattice in  $\fr{t}^{*}_{1}\cong\bb{R}^{2k}$, which is equivalent to say that  $\fr{t}^{*}_{1}$ contains $2k$ linear independent elements of $\cal{P}_{T}$. For  example, for a M-space, it is $\cal{P}_{1}=\cal{P}_{T}$.

\begin{remark}\label{martinet}
\textnormal{Consider the direct sum $\fr{t}=\fr{t}_{0}+\fr{t}_{1}$ and its dual version, $\fr{t}^{*}=\fr{t}_{0}^{*}+\fr{t}_{1}^{*}$. Since $\fr{t}_{0}^{*}\cap\fr{t}_{1}^{*}=\{0\}$,  the projection  of the $T$-weight lattice   onto $\fr{t}_{0}^{*}$  will be a  full rank  lattice in $\fr{t}_{0}^{*}$, {\it if and only if} $\cal{P}_{1}=\cal{P}_{T}\cap \fr{t}_{1}^{*} $ is a full rank lattice in $\fr{t}_{1}^{*}$, i.e.     $\cal{P}_{T}/\cal{P}_{1}$ is torsion free (see \cite[Prop.~1.1.4]{Martinet}).} %In this case, the lattice structure that acquires  $\cal{P}_{T}/\cal{P}_{1}$ can be described as follows. %The space $\fr{t}_{1}^{*}$ can be viewed as the subspace of $\fr{t}^{*}$ spanned by $\cal{P}_{1}$ and $\fr{t}_{0}^{*}$ is its orthogonal complement in $\fr{t}^{*}$.
% Let $c : \fr{t}^{*}\to \fr{t}_{0}^{*}$ the natural projection, i.e. the  $\bb{R}$-linear map which vanishes on $\fr{t}_{1}^{*}$ and is identity on $\fr{t}_{0}^{*}$. Since $\cal{P}_{T}/\cal{P}_{1}$ has been assumed to be torsion free, the image $c(\cal{P}_{T})$ is discrete group of $\fr{t}_{0}^{*}$ and thus a lattice. Moreover, $c$ induces a natural isomorphism  $ \cal{P}_{T}/\cal{P}_{1}\cong c(\cal{P}_{T})$.}
\end{remark}

\begin{lemma}\label{REUSE1}
Any $\lambda\in\cal{P}_{1}$ induces an non-zero $G$-invariant 2-form $\omega_{\lambda}=\frac{i}{2\pi}d\lambda\neq 0$ on $M=G/L$ (notice that $d\lambda\neq 0$ since $H^{1}(M; \bb{Z})=0$), which is cohomologous to zero.
\end{lemma}
\begin{proof}
   It is obvious that any weight $\lambda\in\cal{P}_{1}$ gives rise to an integer $G$-invariant 1-form on $M=G/L$, since it vanishes on $\fr{t}_{0}$ and consequently on $\fr{l}=\fr{h}'+\fr{t}_{0}$, i.e. $\lambda(\mathfrak{t}_0) = \lambda(\mathfrak{l})=0$. Thus our claim follows.
\end{proof}
 Consider the natural projection of $\cal{P}_{T}$ onto the quotient  group   $\cal{P}_{0}:=c(\cal{P}_{T})=\cal{P}_{T}/\cal{P}_{1}$,  which is   a  discrete  subgroup  of  the  vector  group   $\mathfrak{t}^*/\mathfrak{t}_1^{*} = \mathfrak{t}_0^*$,
 \[
 c: \mathcal{P}_T \to  \mathcal{P}_T/ \mathcal{P}_1,\,    \lambda \mapsto c(\lambda) := \lambda  + \mathcal{P}_1.
 \]
\begin{prop}\label{newnew}
(i)  There is a natural  isomorphism
\[
 \cal{P}_{0}:= \cal{P}_{T}/\cal{P}_{1} \cong   H^{2}(M; \bb{Z}), \quad   \lambda +\cal{P}_{1} \mapsto [\omega_{\lambda}]\in H^{2}(M; \bb{Z}),
\]
  where $\omega_{\lambda}:=\frac{i}{2\pi}d\lambda$. \\
  (ii)   The  group   $\cal{P}_1$ is  a   lattice
 in   $\mathfrak{t}_1^*$,  in particular,  $\cal{P}_{1}\cong\delta^{*}\big(H^{1}(T^{2k}; \bb{Z})\big)$,  and   the  group $\cal{P}_0$ is  a lattice  in $\mathfrak{t}_0^*$.
 \end{prop}

\begin{proof}
 (i) The  normalizer     $N_{\mathfrak{g}}(\mathfrak{l})  = \mathfrak{l} +\mathfrak{q}_0  $  is a direct  sum of  two ideals. On the other hand,
  $C_{\mathfrak{g}}(\mathfrak{l})= i\mathfrak{t} + \mathfrak{q}_0.$ Any closed  invariant 2-form  on $M$ has  the  form    $\omega_{\xi}$ where $\xi  $  is  a  real linear    form on $iC_{\mathfrak{g}}(\mathfrak{l})  = \mathfrak{t} + i \mathfrak{q}_0$.
The  form   $\xi = \xi_{\mathfrak{t}} + i\xi_{\mathfrak{q}_0}$ is integer (i.e., $\omega_{\xi}$   defines   an integer class   $[ \omega_{\xi} ] \in  H^2(M, \mathbb{Z})$)  if  and only if   the  component   $\xi_{\mathfrak{t}}$  is   integer,  that is belongs  to    $\mathcal{P}_T$. It    follows   that    the  second   component     defines    the   trivial cohomology  class    $[\omega_{\xi_{\fr{q}_0}}]$.  Now, according to Lemma \ref{REUSE1},  a     form  $\xi \in  \mathcal{P}_T$   defines     the  trivial       class  in   $H^2(M,\mathbb{Z})$ if  and only if $\xi|_{\mathfrak{t}_0} =0$,
that is  iff   $\xi \in \mathcal{P}_1$.  Hence   $H^2(M,\mathbb{Z}) \simeq  \mathcal{P}_T/\mathcal{P}_1$, which proves (i). 

\vskip 0.1cm
\noindent  (ii) Recall by Proposition \ref{cspacespin}   that   $H^{2}(M; \bb{Z})$ is torsion free, in particular $H^{2}(M; \bb{Z})$ is a finitely generated free $\bb{Z}$-module.  Because $H^2(M,\mathbb{Z}) \cong  \mathcal{P}_T/\mathcal{P}_1$,  it immediately follows that   $\cal{P}_{0}:=\cal{P}_{T}/\cal{P}_{1}$  is a lattice in $\fr{t}_{0}^{*}\cong \bb{R}^{m}$ of full rank, and we conclude for $\cal{P}_{1}$ by Remark \ref{martinet}. The isomorphism $\cal{P}_{1}\cong\delta^{*}\big(H^{1}(T^{2k}; \bb{Z})\big)$  follows now by  (i) and Proposition  \ref{cspacespin}.   Another way to prove (ii) reads as follows: It is  clear  that   $\cal{P}_1$ is  a discrete    subgroup  of  $\mathfrak{t}_1^{*}$.  Let   $\be_1, \cdots \be_{2k}$ be its  basis (notice that $H^{1}(M; \bb{Z})=0$ if and only if $H^{1}(F ; \bb{Z})=0$ and $\be_1, \ldots, \be_{2k}$ are linearly independent; thus Proposition \ref{cspacespin} yields the linear independence of $\be_1, \ldots, \be_{2k}$).   It  can be  extended   by elements  $\delta^1 , \ldots , \delta^{m}$  to
   a basis of the lattice $\cal{P}_T$.   Let   $\delta^j = \delta^j_0 + \delta^j_1$  be  the  decomposition    according  to   the   decomposition  $\mathfrak{t}^* = \mathfrak{t}_0^* + \mathfrak{t}_1^*$.  It is  sufficient  to  show  that  elements $\delta^j_0, \,  j=1, \ldots, m$ are linearly independent. Assume  that they  are linearly  dependent.  Then,   there is  a non  trivial  linear combination  with integer coefficients such  that  $\sum   k_j \delta^j_0  =0$.  But  then $\sum k_j \delta^j \in \cal{P}_1$,  which is impossible and gives rise to a contradiction.
   \end{proof}

\begin{theorem}\label{final3}
Let   $M=G/L=G/{\rm T}_{0}^{m}\cdot H'$ be a C-space,  viewed as a ${\rm T}_{1}^{2k}$-principal bundle over a flag manifold  $F=G/{\rm T}^{v}\cdot H'$    with   an invariant  complex  structure
$J_F$. Then, $M$  admits a $G$-invariant spin structure if and only if     the   projection    $c(\sigma^{J_{F}})\in \cal{P}_{T}/\cal{P}_{1}$   of  the  Koszul   form   $\sigma^{J_F}$ of $(F, J_{F})$  is   even, i.e.  is  divided  by  two  in  $\cal{P}_{T}/\cal{P}_{1}$.
  \end{theorem}

\begin{proof}
As we noticed before, $\sigma^{J_M} = \sigma^{J_F} \in   \mathcal{P}_T$ and  Proposition  \ref{c1w2} shows that  the   pull back $\gamma_{J_M} = \pi^* \gamma_{J_F}$   of  the
   invariant Chern  form  $\gamma_{J_F} $ of the flag manifold $(F, J_{F})$,     represents                                                                                                                                                                                                                                                                                                                                                                                                                                                                                                                                                                                                                                                                                                                                                                                                                                                                                                                                                                                                                                                                                                                                                                                                                                                                                                                                                                                                                                                                                                                                                                                                                                                                                                                                                                                                                                                                                                                                                                                                                                                    the first  Cherm  class   of  the C-space $(M=G/L, J_M)$.   Then, the   result  follows  by Proposition \ref{newnew}.
\end{proof}

  Finally, a direct combination of Theorem \ref{final3} and Lemma \ref{REUSE1} yields   that (see also \cite{Gran})

 \begin{corol} \textnormal{}\label{zeroc1}
 The first Chern class $c_{1}(M, J_{M})$ of  a C-space $(M, J_{M})$ vanishes, if and only if $\sigma^{J_{F}}\in\cal{P}_{1}$, where $J_F$  is   the  projection  of     the invariant  complex  structure
 $J_M$   to  $F$.
\end{corol}

 \begin{comment}
\begin{remark}
\textnormal{By  Propositions \ref{cspacespin} and \ref{newnew}, we know that $H^{2}(M; \bb{Z})=H^{2}(F; \bb{Z})/\delta^{*}\big(H^{1}(T^{2k}; \bb{Z})\big)\cong \cal{P}_{T}/\cal{P}_{1}$.  Thus, there is a natural isomorphism $\cal{P}_{1}=\delta^{*}\big(H^{1}(T^{2k}; \bb{Z})\big)$. Now, the  statement of Corollary \ref{zeroc1}  is  equivalent with  Corollary \ref{concl}, (2), which we reformulate as
\[
c_{1}(M)=0 \ \ \Longleftrightarrow \ \ \  c_{1}(F)\in \cal{P}_{1}\cong {\rm Im}\delta^{*} \ \ \Longleftrightarrow \ \ \  c_{1}(F)=\sum_{i=1}^{2k}c_{i}\be_{i}\in\cal{P}_{1}.%  \ \ \Longleftrightarrow \ \ \ \fr{c}(\sigma^{J_{F}})\equiv 0.
\]
Such cases appear for example for M-spaces $M=G/H'$, or when $G$ is odd dimensional and the C-space is of the form $M=G/\U_{1}$. In the latter case,  $M$ is fibered over the full flag manifold  $G/{\rm T}^{\ell}$ of $G$, e.g. ${\rm T}^{6}\to \E_7/\U_1\to \E_7/{\rm T}^{7}$, and thus  $L=H'\cdot {\rm T}_{0}=\U_{1}$ and   $\fr{t}_{0}\cong H^{2}(M; \bb{R})\neq 0$ is non-trivial.}
\textnormal{More general, consider the kernel of the Koszul form, $\cal{A}:=\ker(\sigma^{J_{M}})=\ker(\sigma^{J_{F}})=\{Y\in\fr{t} : \sigma^{J_{F}}(Y)=0\}$.  Since $\dim_{\bb{R}}\cal{A}=v-1$, this is a hypersurface of $\fr{t}$ and if our C-space $M=G/L=G/{\rm T}_{0}\cdot H'$ is such that $\fr{t}_{0}\subset \cal{A}$, then $c_{1}(M)=0$ and $M=G/L$ is spin. }
\end{remark}
 \end{comment}

\subsection{A special construction} If     the  flag manifold  $F=G/H$ is not  spin, then      we  can   construct  all spin   C-spaces  $M=G/L$ over  $F = G/H$   as  follows.
   Assume that $b_{2}(F)\geq 3$ and recall  that the  complex  structure  $J_F$   corresponds  to     the  decomposition  $\Pi = \Pi_W \sqcup \Pi_B$ of  simple  roots into   white  and   black. The Koszul  form   has  the  form  $\sigma^{J_F} =  \sum_{\beta \in \Pi_B} k_j \Lambda_j$,          where  $\Lambda_{j}$   is   fundamental weights   associated  with the   simple black root  $\beta_j\in\Pi_{B}$.  We  choose   a   subset  $\Pi_0 \subset \Pi_B$
  such  that:
 \par (i)   the  coefficients  $k_j$   associated  with   $\beta_j \in \Pi_0$ are   even,     and
\par  (ii)    the  cardinality   $\sharp(\Pi_1) = \sharp(\Pi_B \setminus \Pi_0)$ is   even.\\
  Denote  by   $\mathfrak{t}  = \mathfrak{t}_0 + \mathfrak{t}_1 = \mathrm{ker \Pi_1 } + \mathrm{ker}\Pi_0$  the  associated   direct sum decomposition of    the  space $\mathfrak{t} = i Z(\mathfrak{h})$. The    subalgebra  $\mathfrak{l}= \mathfrak{h}' \oplus \mathfrak{t}_0$  defines a  closed  connected subgroup  $L = H' \cdot  T_0$     and  $M := G/L$ is  a  C-space   over  the  flag manifold
  $F = G/H.$ Notice that the condition $\sharp(\Pi_1)=\text{even}$, certifies that the fibre is  even-dimensional.
  Moreover, Theorem \ref{final3}  implies

  \begin{corol}   \label{cspacespin2}
  The $C$-space $M= G/L$ constructed above, is  spin  and   any  spin  $C$-space is  obtained  by  this  construction.
\end{corol}
\begin{example}
\textnormal{Consider  the flag manifold $F=\E_7(1, 2, 3, 5)=\E_7/\A_{3}\times\A_{1}\times{\rm T}^{3}$ of type $[1, 1]$, with $\Pi_{B}=\{\al_4, \al_6, \al_7\}$.  We compute  $\sigma^{J}=6\Lambda_{4}+3\Lambda_{6}+2\Lambda_{7}$ and $\E_7(1, 2, 3, 5)$ is not a spin manifold. Set $\Pi_{0}=\{\al_4\}$ and $\Pi_{1}=\Pi_{B}\backslash\Pi_{0}=\{\al_6, \al_7\}$. Decompose  $\fr{t}=\fr{t}_{0}\oplus\fr{t}_{1}=\ker\Pi_{0}\oplus\ker\Pi_{1}$ and define the fibration ${\rm T}^{2}\hookrightarrow M\to F$, where $M=\E_7/\A_{3}\times\A_{1}\times{\rm T}^{1}$ is of type $[1, 1]$. Then, according to Corollary   \ref{cspacespin2}  $M$ is spin.  The same   occurs by Proposition \ref{cprop}, since the C-space $M=\E_7/\A_{3}\times\A_{1}\times{\rm T}^{1}$ of type [1, 1]   coincides with  the C-space $M=\E_7/\A_{3}\times\A_{1}\times{\rm T}^{1}$ of type [0,0], which is spin.}
\end{example}

 \end{document}